\documentclass[11pt,reqno]{amsart}

\usepackage{amsmath,amsthm,amssymb}
\usepackage{epsf}
\usepackage{amssymb}
\usepackage{graphics}
\usepackage{hyperref}
\newcommand{\sh}{h} 
\newcommand{\wt}{k} 

\newcommand{\lt}{\left}
\newcommand{\rt}{\right}
\newcommand{\hf}{\frac 12}
\newcommand{\ol}{\overline}

\newtheorem{thm}{Theorem}[section]

\newtheorem{cor}[thm]{Corollary}
\newtheorem{lem}[thm]{Lemma}
\newtheorem{prop}[thm]{Proposition}

\newtheorem{theorem}{Theorem}[section]

\newtheorem{lemma}[theorem]{Lemma}

\numberwithin{equation}{section}

\theoremstyle{definition}

\newcommand\ben{\begin{enumerate}}
\newcommand\een{\end{enumerate}}

\newcommand\be{\begin{equation}}
\newcommand\ee{\end{equation}}
\newcommand\benn{\begin{equation*}}
\newcommand\eenn{\end{equation*}}
\newcommand\bea{\begin{eqnarray}}
\newcommand\eea{\end{eqnarray}}
\newcommand\beann{\begin{eqnarray*}}
\newcommand\eeann{\end{eqnarray*}}

\newcommand{\boldH}{\mathbb{H}}
\newcommand{\R}{\mathbb{R}}
\newcommand{\C}{\mathbb{C}}
\newcommand{\Z}{\mathbb{Z}}
\newcommand{\Q}{\mathbb{Q}}

\newcommand{\qtr}{\frac 14}

\newcommand{\cI}{\mathcal{I}}
\newcommand{\cM}{\mathcal{M}}
\newcommand{\Or}{\ensuremath{{\mathcal O}}}

\newcommand{\gd}{\delta}     
\newcommand{\G}{\Gamma}      
\newcommand{\Gi}{\Gamma_{\infty}}      
\newcommand{\g}{\gamma}      

\newcommand{\bH}{\mathbb{H}}
\newcommand{\thf}{\tfrac 12}

\newcommand{\cusp}{\mathfrak{a}} 
\newcommand{\beq}{\begin{equation}}
\newcommand{\eeq}{\end{equation}}

\newcommand{\sm}{\left(\begin{smallmatrix}}
\newcommand{\esm}{\end{smallmatrix}\right)}
\newcommand{\bpm}{\begin{pmatrix}}
\newcommand{\ebpm}{\end{pmatrix}}

\newcommand{\bsl}{\backslash}

\newcommand{\Res}{\operatorname{Res}} 

\newcommand{\vol}{\mathcal{V}} 

\renewcommand{\mod}{\operatorname{mod}}

\renewcommand{\hat}{\widehat} 

\renewcommand{\Re}{{\mathfrak{Re}}}
\renewcommand{\Im}{{\mathfrak{Im}}}

\newcommand{\<}{\left\langle}
\renewcommand{\>}{\right\rangle}
\newcommand{\bk}{\backslash}

\newcommand{\cP}{\mathcal{P}}
\newcommand{\cO}{\mathcal{O}}

\newcommand{\cuspb}{\mathfrak b}
\newcommand{\m}{{\mathfrak m}}

\newcommand{\tG}{\tilde{G}}

\newcommand{\fund}{\mathfrak{F}}


\begin{document}

\title{Second Moments and simultaneous non-vanishing of $GL(2)$ automorphic $L$-series}

\author{Jeff Hoffstein} \email{jhoff@math.brown.edu}
\address{Department of Mathematics,
Brown University, Providence, RI, $02912$}
 \date{\today}
 \author{Min Lee}
 \email{minlee@math.brown.edu}
 \address{Department of Mathematics,
Brown University, Providence, RI, $02912$}

\thanks{ Both authors are grateful for the support and working environment provided by ICERM. 
The first author is supported in part by NSF grant DMS-0652312 and 
the second author by a Simons Fellowship. }

\maketitle

\begin{abstract}
We obtain a second moment formula for  the $L$-series of holomorphic cusp forms,  averaged over twists by Dirichlet characters modulo a fixed conductor $Q$.   The estimate obtained has no restrictions on $Q$, with an error  term that has a close to optimal power savings in the exponent.    However, one of the contributions to the main term is a special value of a shifted double Dirichlet  series.    We show that this special value is small on average, and obtain a corresponding estimate for a mean value of the second moment over $Q$.  This mean value is non-zero even when applied to a product of two distinct $L$-series, leading to a simultaneous non-vanishing result.  Our approach uses the theory of shifted multiple Dirichlet series to obtain some refined estimates for double shifted sums.   These estimates are the key ingredient in the second moment estimate. 
\end{abstract}


\section{Introduction}
Let $f$ and $g$ be modular forms of even weight $k$, square-free level $N_0$, 
arbitrary nebentypus, and Fourier expansions
\be\label{def1}
	f(z) = \sum_{m \ge 1}A(m)m^{(k-1)/2}e^{2 \pi i m z}
\ee
and
\be\label{defg1}
	g(z) = \sum_{m \ge 1}B(m)m^{(k-1)/2}e^{2 \pi i m z}.
\ee
We assume that $A(1)=1$ and $B(1)=1$, 
and that $f$ and $g$ are eigenfunctions of all Hecke operators.   

If $Q$ is a positive integer, and $\chi$ is a Dirichlet character modulo $Q$, the ``twisted" $L$-series $L(s,f,\chi)$ is defined by 
$$
L(s,f,\chi) = \sum_{m\ge 1}\frac{A(m)\chi(m)}{m^s}.
$$
This converges absolutely when $\Re(s) >1$ and possesses a functional equation as $s \rightarrow 1-s$.

The behavior of $L(s,f,\chi) $ inside the critical strip has been a subject of great interest for many years.  The most basic question one can ask is about the size of the central value $L(1/2,f,\chi)$.    If one asks for what values of $\delta >0$ 
\be\label{delta}
L(1/2,f,\chi) \ll _\epsilon Q^{\delta+ \epsilon},
\ee
 it has long been known via the Phragmen-Lindel\"of principle that \eqref{delta} is true for $\delta = 1/2$.  The present optimal value for $\delta$ is $3/8$, as shown in \cite{HH}, \cite{BH} and \cite{BH10}.   The Lindel\"of Hypothesis, however, predicts that \eqref{delta} should be true with $\delta = 0$.  
 
 One way of making headway is to approach this question by studying moments of families of $L$-series.   The problem of estimating second moments of $GL(2)$ $L$-series is particularly intriguing as some examples are just within reach of current technology, and others remain stubbornly out of reach.   
  
In this work we investigate  the particular second moment
 \begin{equation}\label{e:SQorig}
  S_{f, g}(Q) := 
	 \frac{1}{\varphi(Q)} \sum_{\chi \mod Q}
	 L(1/2,f,\chi) \overline{L(1/2,g,\chi)},
  \end{equation}
where $f$ and $g$ are holomorphic cusp forms, not necessarily distinct.  The sum is taken over all characters mod $Q$, \emph{not} all primitive characters.  
As there are $\varphi(Q)$ terms in this sum, a mean value on the order of $ \log Q$ in the case of $f=g$ would demonstrate that the Lindel\"of Hypothesis is at least true on average.   
We are not aware of previous results on this problem in the case $f\ne g$, 
but the case $f=g$, particularly when $f$ is an Eisenstein series, 
has attracted a great deal of attention.   
In \cite{HB}, Heath-Brown found an asymptotic formula in the case where $f$ was an Eisenstein series, when Q was subject to certain constraints.   
In \cite{S}, Soundararajan found an asymptotic formula valid, in this case, without constraints on $Q$, 
and in \cite{Y}, Young gives such a formula with a power saving error term, also in the Eisenstein series case, but with $Q$ restricted to be prime.
The case where $f$ is a modular form was treated in \cite{St}, by Stefanicki, who found an asymptotic for this case when the factorization of $Q$ obeyed certain constraints, and in \cite{GKR}, Gao, Khan, and Ricotta significantly weakened these constraints.    
It should be noted that in these works the average was taken over primitive characters.   

In the following theorems, we state a specific formula with a close to optimal power saving error term, 
valid for any two newforms $f$ and $g$ of square-free level $N_0$ that are eigenfunctions of all Hecke operators, 
and of the same weight $k$ and nebentypus, without constraints on $Q$.  
As usual, we let $\theta$, with $0 \le \theta \le \hf$ be an approximation toward the Ramanujan-Petersson conjecture.  (Here $\theta = 7/64$ is the best known bound.  See  \cite{KS} for $\Q$.)
The case when $N_0=1$ and $Q$ is a prime, is particularly easy to state: 
\begin{multline}\label{e:SffQ_prime}
	S_{f, f}(Q)
	=
	2 \frac{L(1,f, \vee^2)}{\zeta(2)} \log Q 
	+
	C'(k) \frac{L(1,f, \vee^2)}{\zeta(2)}
	+
	2\left.\frac{d}{ds} \left(\frac{L(s, f, \vee^2)}{\zeta(2s)}\right)\right|_{s=1}
	\\
	+
	2\tilde Z_Q(1-k/2, 1/2; f, f) 
	+ \cO(Q^{\theta+\epsilon-\frac{1}{2}})
\end{multline}
	when $f=g$, where $C'(k)$ is a constant depending on $k$.
When $f\ne g$, 
\begin{multline}\label{e:SfgQ_prime}
	S_{f, g}(Q)
	\\	
	=
	2L(1, f\otimes g) 
	+
	\tilde Z_Q(1-k/2, 1/2; f, g) + \tilde Z_Q(1-k/2, 1/2; g, f)
	+
	\cO(Q^{\theta+\epsilon-\frac{1}{2}})
	.
\end{multline}
Here $L(1, f, \vee^2)$ is the symmetric square $L$-function of $f$ evaluated at $1$ 
and $L(1, f\otimes g)$ is the Rankin-Selberg convolution evaluated at $1$. 
The precise formula for general $Q$, square free $N_0$, with $(Q, N_0)=1$, has a more complex form depending on the divisors of $N_0$ and $Q$, and is given in Theorem \ref{thm:SfgQ}. 
It should be possible, with a little extra work, to eliminate $\theta$ from the exponent of the error term as in \cite{HH}. 

The key question here is what is meant by $\tilde Z_Q(1-k/2, 1/2; f, g)$. 
As we will explain, this represents a special value of a certain shifted double Dirichlet series at the center of its critical strip. 
In our opinion, this term encapsulates the difficulty and profundity of this second moment problem.  
We conjecture that 
$$
	\tilde Z_Q(1-k/2, 1/2; f, g) \ll Q^{\theta-1/2}, 
$$
which would imply that the special value $\tilde Z_Q(1-k/2, 1/2; f, g)$ is dominated by the error term in the preceding two formulas. 

Any result of the form 
$$
	\tilde Z_Q(1-k/2, 1/2; f, g) \ll Q^{\delta-1/2}
$$ 
with $\delta< 1/2$ would turn \eqref{e:SffQ_prime} and \eqref{e:SfgQ_prime} 
into mean value estimates with a power saving error term. 
We have been unable to prove such a bound for a fixed $Q$. 
However, we have been able to prove mean value estimates for $\tilde Z_q(1-k/2, 1/2; f, g)$, as $q$ varies over a very short interval around $Q$, that are consistent with this conjecture. 
The precise upper bound for the mean value for $\tilde Z_q(1-k/2, 1/2; f, g)$ 
is given in Proposition \ref{prop:Mest}. 
For $Q, y\geq 1$, 
this estimate says roughly that, 
$$
	 \sum_{|q-Q| \leq Q/y, \atop (q, N_0)=1} \tilde Z_q(1-k/2, 1/2; f, g) 
	= 
	\cO(y^{\epsilon} Q^{\epsilon})
	.
$$
Thus, for example, taking $y=Q\log Q$, implies the bound 
$$
	\tilde Z_Q(1-k/2, 1/2; f, g) 
	\ll 
	Q^\epsilon
	,
$$
for a single $Q\geq 1$, 
while taking $y=1$ leads to the estimate 
$$
	\sum_{|q-Q| \leq  Q, \atop (q, N_0)=1} \tilde Z_q(1-k/2, 1/2; f, g)
	\ll 
	Q^{\epsilon}
	,
$$
for any $\epsilon>0$. 
One does not expect more than square root cancellation to occur, and thus these estimates are consistent with the conjectured upper bound for an individual $\tilde Z_q(1-k/2, 1/2; f, g)$.	
	
The mean value estimates for the $\tilde Z_q(1-k/2, 1/2; f, g)$ allow us 
to prove corresponding mean value estimates for $S_{f, g}(q)$ in short intervals centered at $Q$.  
Before stating our main results, it will be convenient to make several definitions. 
	For any $L$-series $L(s)$ that is convergent for $\Re(s)>1$ and has an Euler product,
	we define $L^{(Q)}(s)$ to be the corresponding series with the Euler factors 
	at primes dividing $Q$ removed.

Let $N_0$ be a square-free positive integer.
Let $f$ and $g$ be holomorphic cusp forms of even weight $k$, 
which are newforms for $\Gamma_0(N_0)$, of arbitrary nebentypus, 
with normalized Fourier coefficients $A(m)$, $B(m)$, as in \eqref{def1}, so that $A(1) = B(1)=1$.
Assume that $f$ and $g$ are eigenfunctions for all Hecke operators.

When $f=g$, for each $q\geq 1$, with $(q, N_0)=1$, define
\begin{multline}\label{e:H2Qff}
	H_{f, f}^{(2)}(q)
	:=
	\frac{1}{2}
	\sum_{d\mid q} \frac{\mu(d) A(d)^2}{d} 
	\sum_{d_1, d_2\mid d} 
	\frac{\mu(d_1) \mu(d_2) }{(d_1, d_2)A((d_1, d_2))}
	\left(\prod_{p\mid \frac{d_1d_2}{(d_1,d_2)^2}}\frac{1}{p+1}\right)
	\\
	\times
	\left\{
	\sum_{p\mid(d_1, d_2), p^\alpha\|q, \atop \alpha\geq 0} 
	\log p \left(p^{-\alpha} +\frac{1}{2} \right)
	- \sum_{p\mid \frac{d_1d_2}{(d_1, d_2)}} \log p \frac{2p-1}{2(p-1)}
	+
	\sum_{p\mid \frac{d_1d_2}{(d_1, d_2)}, p^\alpha\|q, \atop \alpha\geq 0} 
	\log p \left(p^{-\alpha} \frac{3p-1}{p-1}\right)
	\right\}
	\\
	-
	\sum_{d\mid q} \frac{\mu(d)}{d} 
	\left(\log d -2\sum_{p\mid d} \frac{\log p }{p}\right)
	\left(\prod_{p\mid d} \frac{A(p)^2 + (1-A(p)^2) p^{-1}+p^{-2}}{1+p^{-1}} \right)
\end{multline}
and
\begin{multline}\label{e:Hqff}
	H_{f, f}(q)
	:=
	2 	\log q
	\left(\prod_{p\mid q, \atop \text{ prime }} (1-p^{-1}) \right)
	\frac{L^{(q)}(1, f, \vee^2)}{\zeta^{(q)}(2)} 
	\\ 
	+
	\left\{
	2C_2(k) 
	+ \sum_{p\mid N_0} \log p \left(\frac{4p-1}{4(p-1)}\right)
	+ \sum_{j=1}^{k-1} \frac{1}{j}
	- \log (8\pi)
	- 2\sum_{p^\alpha\|q, \alpha\geq 1} \log p \frac{1-p^{-\alpha}}{p-1}
	\right\}
	\\
	\times
	\left(\prod_{p\mid q, \atop \text{ prime }} (1-p^{-1}) \right)
	\frac{L^{(q)}(1, f, \vee^2)}{\zeta^{(q)}(2)} 
	\\
	+ 
	H_{f, f}^{(2)}(q)
	\frac{L(1, f, \vee^2)}{\zeta(2)} 
	+ 2\left. \frac{d}{ds} \left(\frac{L^{(q)}(s, f, \vee^2)}{\zeta^{(q)}(2s)} \prod_{p\mid q, \atop \text{ prime }} \left(1-p^{-s}\right)\right)\right|_{s=1}
\end{multline}
where $C_2(k)$ is a constant given in \eqref{C1def}, independent of $f$ (but depending on $k$). 
Here $p$ is a prime.

For any positive integer $m$, 
\be\label{e:numberPrimes}
	r(m)
	:= 
	\sharp\left\{p \; \bigg| \; p\text{ is a prime and } p\mid m \right\}
	.
\ee

With this notation, we prove the following:
\begin{thm}\label{thm:main}
For $Q\gg 1$ and $y\geq1$, when $f=g$, we have
\be\label{e:mainf=g}
	\frac{y}{Q}
	\sum_{q\geq 1, \atop (q, N_0)=1} S_{f, g}(q) e^{-\frac{y^2 \left(\log \frac{Q}{q}\right)^2}{4\pi}}
	=
	\frac{y}{Q}
	\sum_{q\geq 1, \atop (q, N_0)=1} H_{f, f}(q) e^{-\frac{y^2 \left(\log \frac{Q}{q}\right)^2}{4\pi}}
	+
	\cO(Q^{\theta+\epsilon-1/2})
	+ \cO (yQ^{-1/2})
	.
\ee
For $Q\gg 1$ and $y\geq 1$, when $f\neq g$, we have
\begin{multline}\label{e:mainfneqg}
	\frac{y}{Q}
	\sum_{q\geq 1, \atop (q, N_0)=1} S_{f, g}(q) e^{-\frac{y^2\left(\log \frac{Q}{q}\right)^2}{4\pi}}
	=
	\frac{L(1, f\otimes g)}{\zeta^{(N_0)}(2)}e^{\frac{\pi}{y^2}} \Res_{v=1}\zeta^{(N_0)}(v)E_{f, g}^{(N_0)}(1) 
	\\
	+
	\frac{e^{\frac{\pi}{y^2}}}{2} 
	L(1, f\otimes g) 
	\prod_{p\mid N_0} 
	\left(pA(p)\overline{B(p)}+1\right)\left(1-p^{-1}\right)
	\\
	\times
	\left\{
	\frac{1}{2^{r(N_0)}}
	\prod_{p\nmid N_0}
	\left\{1-p^{-2} \left( A(p)\overline{B(p)} - \frac{(A(p)+\overline{B(p)})^2}{1+p^{-1}} +1\right)
	\right\}
	\right.
	\\
	\left.
	+
	\prod_{p\nmid N_0}
	\left\{1-p^{-2} \left(A(p)\overline{B(p)} - \frac{A(p)^2+\overline{B(p)}^2}{p+1} 
	+ p^{-1}\right)\right\}
	\right\}
	\\
	+
	\cO(Q^{\theta+\epsilon-1/2})
	+ \cO(yQ^{-1/2})
	.
\end{multline}
Here $E^{(N_0)}_{f, g}(1)$ is a positive constant depending on $f$ and $g$, which is defined in \eqref{Edef}. 
\end{thm}

When $f=g$, the main term of the average over $q$ is non-zero, of the form $c_1(f)\log Q +c_2(f)$, with $c_1(f),c_2(f)$ independent of $Q$.
The constants when $f \neq g$ are a little easier to give explicitly and hence are stated above.
For $p\mid N_0$, $A(p)=\pm \sqrt p^{-1}$ and $B(p)=\pm \sqrt{p}^{-1}$, 
so the second piece of the main term when $f \ne g$ is non-negative.  It could, however, be equal to zero if $pA(p)\overline{B(p)}+1=0$ for some  $p\mid N_0$.
The first piece of the main term, however, is positive, independent of $Q$, and dominates the error term.  

An immediate consequence is the following: for any $f,g$, there exists a character $\chi$ such that
$$
L(1/2,f,\chi) \overline{ L(1/2,g,\chi)} \ne 0.   
$$
That is, the $L$-series of $f$ and $g$, twisted by $\chi$,  are simultaneously nonzero at the center of the critical strip.  This fact can also be derived from \cite{Roh84, Roh88}, where 
it is shown that if one fixes a finite set of primes $S$, then there are only finitely many primitive Dirichlet characters unramified outside of $S$ such that $L(1/2,f,\chi)=0$.    Interestingly, this result depends upon the algebraicity of the central value (divided by a transcendental factor), whereas the methods employed in this work are purely analytic.    If $f, g$ are Maass forms, the algebraicity property is not present, but the methods of this paper should go through to establish a corresponding simultaneous non-vanishing result.    We restrict ourselves to the case of holomorphic cusp forms in this work because the analysis of the relevant special functions is considerably simpler. 

The dominance of the main term in \eqref{e:mainfneqg} allows us to show that given any $X \gg 1$ there will exist a conductor $q$ close to $X$ such that a character modulo $q$  has simultaneous non-vanishing twists.  Specifically, choosing $y= X^{1/2 - \epsilon}$ gives us the following

\begin{cor}\label{cor:main} 
Fix $X \gg 1$ and $\epsilon >0$.   There exists a conductor $q$  with $|X-q| \ll X^{1/2 + \epsilon}$, and a character $\chi$ modulo $q$, such that 
$$
L(1/2,f,\chi) \overline{ L(1/2,g,\chi)} \ne 0.   
$$
The implied constant depends at most on $\epsilon, k$ and $N_0$.
\end{cor}
A very intriguing question remains:  Will there always exist a \emph{quadratic} character $\chi$ such that  $L(1/2,f,\chi)$ and $ L(1/2,g,\chi)$ are simultaneously non-zero?
One way to answer this would be by finding an expression as main term plus error term for an analog of \eqref{e:SQorig}, where the sum is over quadratic characters with 
conductors in a fixed range.    Significant progress in this direction has been made by Soundararajan and Young, \cite{SY}, in the case $f=g$, but the case $f \ne g$ and a corresponding simultaneous non-vanishing result, appears to be very difficult.

\subsection{Shifted sums and an outline of the approach}
It is relatively easy to replace the problem of estimating the mean value by the problem of estimating a certain double shifted sum, for example, for $X \gg 1$,
\be
\sum_{m,h \ge 1 }
\frac{A(m + hQ)\overline{B(m)} e^{-2m/X-hQ/X}}{((m + hQ)m)^{1/2}}\,.
\ee
This estimate must be made quite precise, and contributes part of the main term. 
The problem of estimating this sum can be translated into the problem of determining the meromorphic continuation to $\C^2$ of a double shifted Dirichlet series:
$$
Z_Q(s,w)= \sum_{h,m \ge 1}
\frac{ a(m+hQ)\overline{b(m)}  }{m^{s+k-1}(hQ)^{w+(k-1)/2}}.
$$
Here $a(m), b(m)$ are unnormalized Fourier coefficients of $f$ and $g$, that is, $a(m) = A(m) m^{(k-1)/2}$ and similarly for $b(m)$.
In fact, the problem can be further reduced to finding a precise estimate for the special value $Z_Q(1-k/2,1/2)$.   Interestingly, the point $(1-k/2,1/2)$ is at the center of a critical strip for $Z_Q(s,w)$.   This is a critical strip in the usual sense for an automorphic $L$-function.   The function can be computed by a convergent series on either side of the critical strip, but it is very difficult to obtain estimates at points inside the strip better than can be obtained by convexity.

\subsection{A roadmap}  The paper is organized as follows.
	In Section \ref{s:review}, we collect the relevant background information about single and double shifted Dirichlet series.   
	In Section \ref{sect:Laverage}, we relate the second moment question to the problem of estimating shifted sums.   
	The shifted sums that arise naturally are summed over indices relatively prime to $Q$, 
	and in Section \ref{s:sieving}, we show how to build up such sieved shifted sums from standard shifted sums.   
	In Section \ref{s:shiftedsum_conv}, we relate the shifted sums to inverse Mellin transforms of multiple Dirichlet series and demonstrate how shifting lines of integration leads to contributions to the main terms from residues at simple and (in the case $f=g$), double poles.  
	One such contribution is a special value of $Z_Q(s,w)$ at the point $(1-k/2,1/2)$.   
	Near $(1-k/2,1/2)$, $Z_Q(s,w)$  separates into a polar piece and an analytic piece.   
The remainder of the paper is devoted to analysis of the analytic piece 
$\tilde Z_Q(s, w)$ as $Q$ varies. 
To perform this analysis, we introduce the triple shifted Dirichlet series, 
$$
	\cM(s, w, v)
	=
	\sum_{q\geq 1, \atop (q, N_0)=1} 
	\frac{Z_q(s, w)}{q^v}
	.
$$
In Section \ref{s:triple}, we obtain the meromorphic continuation of $\cM(s, w, v)$ to $\C^3$.
In Section \ref{s:svest}, we use the analytic properties of $\cM(s, w, v)$ to prove Theorem \ref{thm:main}.

\paragraph{\bf Acknowledgement} The authors would like to thank Valentin Blomer, David Hansen, Roman Holowinsky, Thomas Hulse, Mehmet Kiral, Chan Ieong Kuan, Michael Rubinstein and Matt Young for many very helpful and enlightening conversations.

\section{A review of  some properties of shifted multiple Dirichlet series}\label{s:review}

In this section, we collect background materials from \cite{HH}. 

Let $f$, $g$ be cusp forms of even weight $k$ for $\Gamma_0(N_0)$ with $N_0$ square-free,
with Fourier expansions as above.  Fix $\ell_1,\ell_2 \ge 1$  to be square-free with $(N_0, \ell_1\ell_2)=1$, set
$N=N_0\ell_1\ell_2/(\ell_1,\ell_2)$ and let 
 $V_{\ell_1, \ell_2}$ be the $\Gamma:=\Gamma_0(N)$-invariant function
\begin{equation}\label{e:Vell1ell2}
	V(z) 
	= V_{\ell_1, \ell_2}(z)
	:= y^k\overline{f(\ell_1z)}g(\ell_2z)
	,
\end{equation}
which is rapidly decreasing at the cusps of $\Gamma$.
Recall that in \cite{HH}, for $h\geq 1$ the shifted series
\begin{equation}\label{e:Dsh}
	D(s; h)
	:=
	\sum_{m_2,h \ge 1 \atop \ell_1m_1= \ell_2m_2 + \sh}\frac{a(m_1)\overline{b(m_2)}}{(\ell_2m_2)^{s+k-1}}
	,
\end{equation}
absolutely convergent for $\Re(s) >1$, was defined.
It was convenient to view it as the limit as $\delta \rightarrow 0$ of another series:
$$
	D(s;\sh,\delta)
	= 
	\sum_{m_2,h \ge 1 \atop \ell_1m_1= \ell_2m_2 + \sh} 
	\frac{a(m_1) \overline{b(m_2)}} {(\ell_2m_2)^{ s+\wt-1}}
	\left(1 + \frac{\delta \sh}{ \ell_2 m_2} \right)^{-(s+k-1)}
 	.
 $$

The meromorphic continuation and spectral expansion of $D(s; h, \delta)$ was obtained, and to give this we need a certain amount of notation.
For any $u, v\in L\left(\Gamma\bsl \bH\right)$, we define the 
Petersson inner product
$$
	\left<u, v\right>
	:= 
	\frac{1}{\vol}\iint_{\Gamma\bsl \bH}
	u(z)\overline{v(z)}\; \frac{dx\; dy}{y^2}
$$
where $\vol= {\rm Vol}(\Gamma\bsl \bH)$.

Let $\{u_j\}_{j\geq 0}$
be an orthonormal basis (with respect to the Petersson inner product) 
for the discrete part of the spectrum of the Laplace operator on $L^2\left(\Gamma\bsl \bH\right)$.
Let $s_j(1-s_j)$ for $s_j\in \C$ be an eigenvalue of the Laplace operator for $u_j$ for each $j\geq 0$.   Then we have the Fourier expansion
\begin{equation}\label{e:normalized_maassform}
	u_j(z) 
	= 
	\sum_{n\neq 0} \rho_j(n)\sqrt{y}K_{it_j}(2\pi|n|y)e^{2\pi inx}\,.
\end{equation}
Here $s_j=\frac{1}{2}+it_j$ for $t_j\in \R$ or $\frac{1}{2}< s_j< 1$. 
Also $K_{\tau}(y)$, for $\tau\in \C$ and $y>0$, is the $K$-Bessel function.

The cusps of $\Gamma_0(N)$ are uniquely represented by 
the rationals $1/a$ for $a \mid N$, 
and the Eisenstein series for $\Gamma$ are indexed by the cusps.
For each cusp $\cusp\in \Q$, let $\sigma_\cusp\in SL_2(\R)$, 
with $\sigma_\cusp \infty = \cusp$, 
be a scaling matrix for the cusp $\cusp$, i.e., 
$\sigma_\cusp$ is the unique matrix (up to right translations)
such that 
$\sigma_\cusp \infty=\cusp$
and 
$$
	\sigma_\cusp^{-1}\Gamma_\cusp \sigma_\cusp 
	=
	\Gamma_\infty = \left\{\left.
	\pm \bpm 1 & b\\ 0 & 1\ebpm \;\right|\;
	b\in \Z\right\}
$$
where
$$
	\Gamma_\cusp 
	= 
	\left\{\gamma\in \Gamma\;|\; \gamma \cusp = \cusp\right\}
	.
$$

For a cusp $\cusp$, define the Eisenstein series at the cusp $\cusp$ to be
$$
	E_\cusp(z, s)
	:=
	\sum_{\gamma\in \Gamma_\cusp\bsl \Gamma}
	\Im(\sigma_\cusp^{-1}\gamma z)^s.
$$
For $\cusp = 1/a$, with $a \mid N$, 
we have the following Fourier expansion which we quote from \cite{DI83}:
\begin{multline}\label{m:EisensteinSeries_fex}
	E_\cusp \left(z, \frac{1}{2}+\tau\right) 
	=
	\delta_{\cusp, \infty} y^{\frac{1}{2}+\tau}
	+
	\tau_\cusp\left(\frac{1}{2}+\tau, 0\right) y^{\frac{1}{2}-\tau}\\
	+
	\sum_{n\neq 0}
	\tau_\cusp\left(\frac{1}{2}+\tau, n\right)
	\sqrt y K_{\tau}(2\pi |n|y)e^{2\pi inx},
\end{multline}
where
$$
	\tau_\cusp\left(\frac{1}{2}+\tau, 0\right)
	=
	\frac{\sqrt\pi \Gamma(\tau)}
	{\Gamma\left(\frac{1}{2}+\tau\right)}
	\rho_\cusp\left(\frac{1}{2}+\tau, 0\right)
$$
and
$$
	\tau_\cusp\left(\frac{1}{2}+\tau, n\right)
	=
	\frac{2\pi^{\frac{1}{2}+\tau}}
	{\Gamma\left(\frac{1}{2}+\tau\right)}
	\rho_\cusp\left(\frac{1}{2}+\tau, n\right)
	|n|^{\tau},
$$
with 
\be\label{e:rhocusp}
	\rho_\cusp \left(\frac{1}{2}+\tau, n\right)
	=
	\left(\frac{1}{aN}\right)^{\frac{1}{2}+\tau}
	\sum_{\gamma\geq 1, \atop \left(\gamma, \frac{N}{a}\right)=1}
	\gamma^{-1-2\tau}
	\sum_{\delta\pmod{\gamma a}, (\delta, \gamma a)=1}
	e^{-2\pi i n\frac{\delta}{\gamma a}}
	\,.
\ee
Here we fix
\be\label{e:sigma1/a}
	\sigma_{1/a} 
	= \bpm  1 & 0 \\a & 1\ebpm\bpm \sqrt{\m_\cusp} & 0 \\ 0 & \sqrt{\m_\cusp}^{-1}\ebpm
	,
\ee
where $\m_\cusp = N/a$.

 In \cite{B}, it is shown that for $t\in \R$ the $\rho_\cusp\left(\frac{1}{2}+it, n\right)$ satisfy
$$
	\left|\zeta(1+it)\right|^2
	\sum_\cusp 
	\left|\rho_\cusp\left(\frac{1}{2}+it, n\right)\right|^2
	\ll_\epsilon ((1+|t| )N)^{\epsilon}
$$
for any $\epsilon>0$.

With this notation, it was demonstrated that the spectral expansion of $D(s;h,\delta)$ had the following form.  For $\Re(s)>\frac{1}{2}$, 
$$
	D(s; h, \delta) 
	= 
	D_{\rm cusp}(s; h, \delta) +D_{\rm cont}(s; h, \delta)
	,
$$
where 
\be\label{cuspdelta}
	D_{\rm cusp}(s; h, \delta)
	:= 
	\frac{(4\pi)^k 2^{s-\frac{1}{2}}}
	{2\sqrt{\pi}\Gamma(s+k-1) h^{s-\frac{1}{2}}}
	\sum_j \overline{\rho_j(-h)} M(s, t_j, \delta) \overline{\left<V, u_j\right>}
\ee
and
\begin{multline}\label{contdelta}
	D_{\rm cont}(s; h, \delta)
	:= 
	\frac{(4\pi)^k 2^{s-\frac{1}{2}} \vol }
	{2\sqrt{\pi}\Gamma(s+k-1) h^{s-\frac{1}{2}}}
	\\
	\times \sum_{\cusp}
	\frac{1}{2\pi i} 
	\int_{(0)} 
	\frac{\tau_\cusp\left(\frac{1}{2}-\tau, -h\right)}
	{\zeta^*(1+2\tau)}
	M\left(s, \tau/i, \delta\right) 
	\overline{\left< V, E_\cusp^*\left(*, \frac{1}{2}+\tau\right)\right>} \; d\tau.
\end{multline}
Here $\zeta^*(s)$ is the completed Riemann zeta function.
The sums and integrals are absolutely convergent 
for any $\delta>0$, and for any fixed $A>0$ have a meromorphic 
continuation to 
$\Re(s) >\frac{1}{2}-A$ that is absolutely convergent away from the poles
of $M(s,t,\delta)$.

For convenience, we repeat here a description of 
the relevant properties of $M(s,t,\delta)$.   
The function $M(s,t,\delta)$ is originally defined, for $1/2<\Re(s) $ and $t \in \R$, 
by
$$
	M(s, t, \delta)
	=
	\int_0^\infty y^{s-\frac{1}{2}} e^{y(1-\delta)} K_{it}(y)\; \frac{dy}{y}
	.
$$
Its meromorphic continuation is given in the following proposition. 
This is a slight strengthening of the original proposition of \cite{HH}, 
which can be found in \cite{Hulse}. 

\begin{prop}\label{prop:Mbounds} 
Fix  $\epsilon >0, \delta >0$ and  $A \gg 1$, not  an integer. 
The function $M(s,z/i,\delta)$ has a meromorphic continuation to all $s$ and $z$ in $\C$.  
In this region $M(s,z/i,\delta)$ is analytic except for simple poles at the points $s = \hf \pm z - \ell$ for $0 \le \ell < A$.   
When $z \neq 0$, the residues at   $s = \hf  \pm  z - \ell$ are given by 
\begin{multline}\label{Mres}
	\Res_{s =\hf \pm z - \ell} M(s,z/i,\delta)
	=
	\frac{(-1)^\ell \sqrt{\pi}2^{\ell \mp z} \Gamma(\hf \mp z + \ell) \Gamma (\pm 2 z - \ell) }{\ell ! \Gamma (\hf + z)\Gamma(\hf - z)}
	\\
	+ 
	\mathcal{O}_\ell((1+|\Im(z)|)^{\ell}e^{-\frac{\pi}{2}|\Im(z)|}\delta^{\qtr-\epsilon})
	.
\end{multline}
When $z =0$ there are double poles at $s=\hf-\ell$ for $\ell \geq 0$, and within an $\epsilon>0$ radius of these points, $M(s,z/i,\delta)$ has the form
\be\label{laurent}
	M(s,0,\delta) 
	= 
	\frac{c_2(\ell)+\Or_\ell(\delta)}{(s-\thf+\ell)^{2}}+\frac{c_1(\ell)+\Or_\ell(\delta^{\qtr})}{(s-\thf+\ell)}+\Or_\ell(1)+\Or_\ell(\delta^{\qtr-\epsilon})
\ee
where $c_1(\ell)$ and $c_2(\ell)$ are computable functions of $\ell$, given by
\be\label{c1}
	c_1(\ell) 
	=  
	\frac{(2\ell)!}{2^\ell (\ell!)^3}\lt( 2\Gamma'(1)-\frac{\Gamma'(\hf+\ell)}{\Gamma(\hf+\ell)}-\log(2) \rt)
\ee
and
\be\label{c2}
	c_2(\ell) =  \frac{(2\ell)!}{2^\ell(\ell!)^3}
	.
\ee

Considering the poles in terms of $z$, when $0<\left|\Re(s) -\hf +\ell \right|<\epsilon$ for small enough $\epsilon>0$ and $\ell \geq 0$ we have that
\begin{multline}\label{ires2}
	\Res_{z=\pm(s+\ell-\hf)} M(s,z/i,\delta)
	=
	\mp \frac{(-1)^\ell \sqrt{\pi}2^{\hf-s}\Gamma(1-s)\Gamma(2s+\ell-1)}{\ell!\Gamma(s+\ell)\Gamma(1-s-\ell)} 
	\\
 	+
	\cO_{A}\lt(\frac{\Gamma(2s+\ell-1)}{\Gamma(s+\ell)}(1+|s|)^{1-2\Re(s)} \delta^{\frac{1}{4}-\epsilon}\rt) 
\end{multline}
where these residues have a meromorphic continuation
to $\hf -A < \Re (s) < \hf -\ell +\epsilon$. 

When $s$ and $z$ are at least $\epsilon$ away from the poles of $M(s,z/i,\delta)$ we have the upper bound
\be\label{upper1}
	M(s,z/i,\delta) 
	\ll_{A,\epsilon} 
	(1 + |\Im(z)|)^{2 \Re(s) -2 - 2A}(1+|s|)^{4A- 3 \Re(s) +3}\delta^{-A }e^{-\frac\pi 2 |\Im(s)|}.
\ee

Also, for $\Re (s+z) \leq \hf+\max(0,2\Re(z))$, at least a distance $\epsilon$ away from poles, and $\delta (1 +|\Im(z)|)^{2} \le 1$,
\begin{multline}\label{deltasmall}
	M(s,z/i,\delta) 
	=    
	\frac{\sqrt{\pi} 2^{\hf-s}\G(s-\hf -z) \G(s-\hf + z) \G(1-s)}{\G(\hf - z)\G(\hf + z)}	
	\\  
	+
	\cO_{A,\epsilon} \left( (1 + |\Im(z)|)^{2\Re (s)-2 +  2\epsilon}(1 + |s|)^{ 1 -\Re( s) }e^{-\frac{\pi}{2}|\Im (s)|} \delta^\epsilon \right)
\end{multline}
while for $\delta (1 +|\Im(z)|)^{2} > 1$
\be\label{deltabig}
	M(s,z/i,\delta) 
	\ll_{A,\epsilon}  
	(1 + |\Im(z)|)^{2\Re(s) -2}(1 + |s|)^{ 4A-3\Re(s) +3 }e^{-\frac{\pi}{2}|\Im(s)|}.
\ee

When $\Re (z) =0$ and $|z|,|\Im (s)| \gg 1$, $|s\pm z -\hf -m| =\epsilon>0$,  for $\epsilon$ small, we have
\be\label{nearp1}
	M(s,z/i,\delta) 
	\ll_m 
	\epsilon^{-1}(1+|s|)^{1-\Re(s)}e^{-\frac{\pi}{2}|\Im(s)|}.
\ee
\end{prop}

Finally, it is convenient to summarize part of the proposition by setting
$$
	M(s, t)
	=
	\lim_{\delta \rightarrow 0} M(s, t,\delta)= \frac{\sqrt{\pi} 2^{\hf-s}\G(s-\hf -it) \G(s-\hf + it) \G(1-s)}{\G(\hf - it)\G(\hf + it)}.
$$

The use of $\delta >0$ is convenient because the spectral expansions \eqref{cuspdelta}, \eqref{contdelta} converge for $\Re (s)>1/2$ for all $\delta >0$ and have absolutely convergent meromorphic continuations back to $\Re(s) >1/2-A$.    However, the limits as $\delta$ approaches 0 of these series:
$$
	D_{\rm cusp}(s; h)
	= 
	\lim_{\delta \rightarrow 0}D_{\rm cusp}(s; h, \delta)
$$
and
$$
	D_{\rm cont}(s; h)
	= 
	\lim_{\delta \rightarrow 0}D_{\rm cont}(s; h, \delta)
$$
are only convergent in certain regions: 
$\Re (s) <1/2-k/2$ for $D_{\rm cusp}(s; h)$, 
and $\Re(s) <1-k/2$ for $D_{\rm cont}(s; h)$.

The meromorphic continuation of 
$D_{\rm cusp}(s; h, \delta)$ to $\Re(s) >1/2-A$ 
is given by the same expression \eqref{cuspdelta}.  
The meromorphic continuation of $D_{\rm cont}(s; h, \delta)$
to $\Re(s)>1/2-A$, however, is different.    
This is because as $s$ is continued to 
$\Re(s) >1/2-A$, an extra residual term is added every time 
$s$ passes the vertical line $\Re(s) = 1/2 -r$, with $r \ge 0$ 
an integer.  
Specifically, for $\Re(s)=:\sigma >1/2-A$,
\be\label{Dcontextend}
	D_{\rm cont}(s; h, \delta) 
	= 
	D_{\rm int}(s; h, \delta)+\Omega(s; h, \delta),
\ee
where
\begin{multline} \label{Dspec} 
	D_{\rm int}(s; h, \delta)
	:=   
	\frac{(4 \pi)^{k}2^{s-\hf} \vol }
	{2\sqrt{\pi}\Gamma(s+k-1) h^{s-\hf}}
	\\ 
	\times
	\sum_\cusp   
	\frac{1}{4\pi i} 
	\int_{-\ol{C_\sigma}} 
	 \frac{ \tau_\cusp\left(\frac{1}{2}-\tau, -h\right) }
	{\zeta^*(1+2\tau)}
	M(s,\tau/i,\delta)
	\ol{\langle V,E^*_\cusp(*,\thf+\tau) \rangle}  \; d\tau
 \end{multline}
and
\begin{multline}\label{Omega}
	\Omega(s; h, \delta)
	:=   
	\frac{(4 \pi)^{k}2^{s-\hf} \vol }
	{2\sqrt{\pi}\Gamma(s+k-1) h^{s-\hf}}
	\sum_\cusp 
	\lt[
	\sum_{\ell=0}^{\lfloor\hf-\sigma\rfloor} 
	\lt(
	\frac{ \tau_\cusp\left(1-s-\ell, -h\right) } { \zeta^*(2s+2\ell) }   
	\overline{\langle V, E_\cusp^*\left(*, s+\ell \right)\rangle } 
	\rt.
	\rt.
	\\
	+ 
	\lt. 
	(1-\pmb{\delta}_{\sigma, \ell} )
	\frac{\tau_\cusp\left(s+\ell, -h\right)}
	{\zeta^*(2-2s-2\ell )} 
	\ol{\langle V,E^*_\cusp(*,1-s-\ell ) \rangle}
	\rt)
	\\
	\lt. 
	\times
	\frac{(-1)^{\ell} \sqrt \pi 2^{\hf-s} \Gamma(1-s) \Gamma(2s+\ell-1)}
	{\ell! \Gamma(s+\ell) \Gamma(1-s-\ell)}
	\lt(1+\Or_A\lt((1+|s|)^{\frac94-2\sigma}\delta^{\qtr-\epsilon}\rt)\rt)
	\rt]
	\,.
\end{multline}
Here
\be\label{csigma}
	 C_\sigma 
	 = 
	 \lt\{
	 \begin{array}{ll}
	 (0) & \mbox{ when } \thf-\sigma \notin \Z_{\geq 0} \\
	 C & \mbox{ when } \thf -\sigma \in \Z_{\geq 0}
	 \end{array}
	 \rt.
\ee
and
\be\label{deltasigma}
	\pmb{\delta}_{\sigma,\ell} 
	= 
	\lt\{
	\begin{array}{ll}
	1 & \mbox{ when } \ell \neq  \thf-\sigma \\
	0 & \mbox{ when } \ell = \thf -\sigma. 
	\end{array}
	\rt. 
\ee
		
Finally, 
the function $D(s;h)$ has simple poles at  $s = \hf \pm it_j - r$ 
for each $t_j$ and each $0 \le r < A \mp \Im ( t_j)$ 
and
\be\label{Rdef}
	\Res_{s = \hf + it_j - r}D(s;h)
	= 
	c_{r,j} h^{r-it_j}
	\overline{\rho_j(-h)},
\ee
where now the index $j$ identifies 
our choice of $+it_j$ or $-it_j$ 
and
\be\label{crjdef1}
	c_{r,j} 
	=
	\frac{
	(-1)^r (4\pi)^k\overline{\<V,u_j\>}
	\Gamma(\hf - it_j +r)\Gamma(2it_j -r)
	}
	{
	2r!\Gamma(\hf + it_j)\Gamma(\hf-i t_j) \Gamma(k-\hf + it_j -r)
	}
	.
\ee
For $T \gg1$, 
the $c_{r,j} $ satisfy the average upper bounds
\be\label{crj2}
	\sum_{|t_j|\sim T}
	|c_{r,j}|^2 e^{\pi |t_j|} 
	\ll
	\log( T) (\ell_1 \ell_2)^{ -k} T^{2r + 1}.
\ee

For $Q\in \Z$, $Q\ge 1$ with $(Q, N_0)=1$,  a double Dirichlet series is defined by
\be\label{ZDef}
	Z_Q(s,w) 
	=
	Z_Q(s, w; \ell_1, \ell_2)
	:=   
	(\ell_1\ell_2)^{\frac{k-1}{2}}\sum_{h \ge 1}\frac{ D(s;hQ)}{(hQ)^{w+(k-1)/2}} 
	.
\ee
This is absolutely convergent for $\Re(s), \Re(w) >1$.     
For any $\delta >0$, 
we also define the corresponding double series
\be\label{ZdeltaDef}
	Z_Q(s,w;\delta) 
	:=  
	(\ell_1\ell_2)^{\frac{k-1}{2}} 
	\sum_{h \ge 1}\frac{ D(s;hQ; \delta)}{(hQ)^{w+(k-1)/2}}
	.
\ee
This can be rewritten as
\begin{multline}\label{ZDef}
	Z_Q(s,w) 
	=  
	\sum_{h,m \ge 1\atop \ell_1m_1= \ell_2m_2 +hQ}
	\frac{
	A(m_1) \overline{B(m_2)}
	\left( 1+ \frac{hQ}{ \ell_2m_2 }\right)^{(k-1)/2}  
	}
	{
	(\ell_2m_2)^s(hQ)^{w+(k-1)/2}
	}
	\\ 
	= 
	(\ell_1\ell_2)^{(k-1)/2}  
	\sum_{h,m \ge 1\atop \ell_1m_1= \ell_2m_2 +hQ}
	\frac{ 
	a(m_1)\overline{b(m_2) } 
	}
	{(\ell_2m_2)^{s+k-1}(hQ)^{w+(k-1)/2}}
	,
\end{multline}
and
$$
	Z_Q(s,w) 
	= 
	\lim_{\delta \rightarrow 0}Z_Q(s,w;\delta)
	.
$$
The functions $Z_Q(s,w)$ and $Z_Q(s,w;\delta)$ 
have spectral expansions and meromorphic continuations 
to $\Re(s) >1/2 -A$ given by 
\be\label{Zinit}
	Z_Q(s,w) 
	=
	(\ell_1\ell_2)^{(k-1)/2} 
	\left( S_{\text{cusp}}(s,w) + S_{\text{cont}}(s,w) \right)
\ee
and 
\be\label{Zdeltainit}
	Z_Q(s,w;\delta) 
	= 
	(\ell_1\ell_2)^{(k-1)/2} \left(S_{\text{cusp}}(s,w;\delta) 
	+ S_{\text{cont}}(s,w;\delta)\right)
	.
\ee
Here for $\Re (s)>1/2-A$, $\Re (s') >1$ and $\Re(w)>1$, 
\be\label{Scuspdeltadef}
	S_{\rm cusp}(s, w; \delta)
	:=
	\frac{(4\pi)^k 2^{s-\frac{1}{2}}}
	{2\sqrt{\pi}\Gamma(s+k-1) }
	\sum_j L_Q\left(s', \overline{u_j }\right) 
	M(s, t_j, \delta) 	
	\overline{ \left<V, u_j\right>}
\ee
and
\be\label{Scontdeltadef}
	S_{\rm cont}(s, w; \delta)
	:= 
	S_{\rm cont, int}(s, w; \delta) +\Psi(s,w;\delta) , 
\ee
with
\begin{multline}\label{ScontIntdeltadef}
	S_{\rm cont, int}(s, w; \delta)
	:=
	\frac{(4\pi)^k 2^{s-\frac{1}{2}} \vol} {2\sqrt{\pi}\Gamma(s+k-1)}
	\sum_\cusp \frac{1}{2\pi i} \int_{(0)} 
	\frac{Q^{-s'}\zeta_{\cusp, Q}(s', \tau)}
	{\zeta^*(1+2\tau)\zeta^*(1-2\tau) }
	\\
	\times 
	M(s, \tau/i, \delta)
	\overline{\left<V, E^*_\cusp\left(*, \frac{1}{2}+\tau\right)\right>} 
	\; d\tau 
\end{multline}
and
\begin{equation}\label{Psi}
	\Psi(s,w;\delta)
	:= 
	\sum_{h_0\geq 1} \frac{\Omega(s; h_0Q, \delta)}{(h_0Q)^{w+(k-1)/2}}
	,
\end{equation}
assuming $\Re(s)\notin \frac{1}{2}\Z$.
Here 
$s' = s+w +\frac{k}{2}-1$, 
\be\label{LQDef}
	L_Q(s',\overline{u_j}) 
	= 
	\sum_{h \ge 1}\frac{\overline{\rho_j(hQ)}}{(hQ)^{s'}},
\ee
and
\be\label{zetaQDef2}
	\zeta_{\cusp, Q}(s', \tau)
	:=
	\frac{\zeta^*(1-2\tau)}{2}
	\sum_{h\geq 1} 
	\frac{\tau_\cusp\left(\frac{1}{2}-\tau, -hQ\right)} {h^{s'}}
	,
\ee
then
$$
	\frac{\zeta_{\cusp, Q}(s', \tau)}
	{\zeta(1-2\tau)}
	=
	\frac{1}{Q^{\tau}} \sum_{h\geq 1}
	\frac{\rho_\cusp\left(\frac{1}{2}-\tau, -hQ\right)}
	{h^{s'+\tau}}
	.
$$

For each $\cusp=1/a$ with $a\mid N$, following \cite{HM06}, for $n\neq 0$, we have
\begin{multline}\label{e:rho_an}
	\rho_\cusp(s, n)
	=
	\left(\frac{N}{a}\right)^{-s}
	\frac{\sum_{d\mid n, \atop (d, N)=1} d^{1-2s}}
	{\zeta(2s)}
	\\
	\times
	\prod_{p\mid N} (1-p^{-2s})^{-1}
	\prod_{p\mid a, p^\alpha\|n, \atop \alpha\geq 0}
	\frac{p^{-2s}} {1-p^{-2s+1}}
	\left(p-p^{\alpha (-2s+1)+1} -1 + p^{(\alpha+1)(-2s+1)}\right)
\end{multline}
and after a simple computation, 
\begin{multline}\label{e:zeta_aQ}
	\zeta_{\cusp, Q}(s', \tau)
	=
	\zeta(s'+\tau) \zeta(s'-\tau)
	\left(\frac{N}{a}\right)^{-\frac{1}{2}+\tau}
	Q^{-\tau}
	\prod_{p\mid N} \left(1-p^{-1+2\tau}\right)^{-1}
	\prod_{p\mid\frac{N}{a}} (1-p^{-(s'-\tau)})
	\\
	\times
	\prod_{p\mid a, p^\alpha\|Q, \atop \alpha\geq 0}
	\frac{p^{-1+2\tau}}{1-p^{2\tau}}
	\left((p-1)(1-p^{-(s'-\tau)})
	+ p^{\alpha 2\tau}(p^{2\tau}-p)(1-p^{-(s'+\tau)}) 
	\right)
	\\
	\times
	\prod_{p\mid Q, p\nmid N, \atop p^\alpha\|Q}
	\left(1-p^{2\tau}\right)^{-1}
	\left((1-p^{-(s'-\tau)}) - p^{(\alpha+1)2\tau}(1-p^{-(s'+\tau)}) \right)
	\,.
\end{multline}

The corresponding expansions hold for 
$S_{\rm cusp}(s, w)$ and $S_{\rm cont}(s, w)$,
with  $M(s,t;\delta)$ replaced by $M(s,t)$. 
These expansions converge absolutely, however, only
in the range $\Re(s)  <1/2-k/2$ and $\Re (w) >1$.
Also we have 
$\Psi(s, w) = \lim_{\delta\to 0} \Psi(s, w; \delta)$
and the meromorphic continuation of $\Psi(s, w)$ to $\Re(s')> \frac{1}{2}-\epsilon$ is given by 
\begin{multline}\label{e:Psi0}
	\Psi(s,w)
	:=	
	\frac{(4\pi)^{k}\Gamma(1-s)}{2\Gamma(s+k-1)}
	\sum_{a} \bigg[ 
	\sum_{\ell=0}^{\lfloor\hf-\sigma\rfloor} 
	\lt(
	\frac{\vol Q^{-s'}\zeta_{\cusp, Q} \left(s', 1-s-\ell\right) } { \zeta^* (2-2s-2\ell) \zeta^*(2s+2\ell) }   
	\overline{\langle V, E_\cusp^*\left(*, s+\ell \right)\rangle } 
	\rt.
	\\
	+ 
	\lt. 
	(1-\pmb{\delta}_{\sigma, \ell} )
	\frac{\vol Q^{-s'} \zeta_{\cusp, Q} \left(s', s+\ell\right)}
	{\zeta^*\left(2s+2\ell \right)\zeta^*(2-2s-2\ell )} 
	\ol{\langle V,E^*_\cusp(*,1-s-\ell ) \rangle}
	\rt)
	\frac{(-1)^{\ell} \Gamma(2s+\ell-1)}
	{\ell! \Gamma(s+\ell) \Gamma(1-s-\ell)}
	\\
	+
	\pmb{\delta}'_{\sigma'}
 	\frac{\G(s + \frac12 -s')\G(s- \frac32 +s')}
	{Q^{s'}\pi^{\hf -s'}\zeta^*(3-2s')\G( s'-\frac12)^2\G( \frac32  -s')}  
	\\
	\times 
	\lt(K_{\cusp, Q}^+(s')
	\overline{\<V,E_\cusp^*(*,\tfrac 32 -s')\>}
	+(1-\pmb{\delta}'_{2\sigma'})K_{\cusp, Q}^-(s')
	\overline{\<V,E_\cusp^*(*,s' -\tfrac 12)\>}\rt) \bigg]
	,
\end{multline}
where $K_{\cusp, Q}^\pm(s')$ is a ratio of Dirichlet polynomials defined by 
$$
	\Res_{\tau=\pm(1-s')} \zeta_{\cusp, Q}(s', \tau) 
	= 
	K^\pm_{\cusp, Q}(s') \zeta(-1+2s')
	.
$$

Referring to line (7.9) of \cite{HH}, and the following line, 
we have the bound for $\Re \, {s'} \ge\hf - \epsilon$, 
$\Re \, w > 1$, and $\ell_1,\ell_2 \ll Q$
\be\label{Z2}
	Z_Q(s,w) - \Psi(s, w)
	\ll  
	Q^{\theta - {s'}+\epsilon}(1+|s|)^{1-k}(1 + |{s'}|)^{1+\epsilon}
	.
\ee
The $\theta$ can be removed if $\ell_1,\ell_2 \gg Q^{1/4} \log Q$.
The same bound applies to the individual pieces	
$$ 
	(\ell_1\ell_2)^{(k-1)/2} S_{\text{cusp}}(s,w;\delta) 
$$ 
and  
$$
	(\ell_1\ell_2)^{(k-1)/2}  S_{\text{cont,int}}(s,w;\delta)
.
	$$

Finally, $Z_Q(s,w)$ has poles at $s = 1/2-r+it_j$ and, when $t_j \ne 0$,
\be\label{Zpoles}
	\Res_{s = \hf -r \pm it_j} Z_Q(s,w)
	= 
	(\ell_1\ell_2)^{(k-1)/2} c_{r,j} L_Q(w+(k-1)/2-r+it_j,\overline{u_j})
	.
\ee

\section{Relating mean values of $L$-series to shifted sums}\label{sect:Laverage}

Let $f,g$ be modular forms of even weight $k$, square-free level $N_0$ and normalized Fourier coefficients $A(m),B(m)$, as in \eqref{def1}.
For a positive integer $Q\geq 1$ with $(Q, N_0)=1$, let $\chi$ be a character modulo $Q$. 

Multiplying the twisted $L$-series by $\G(s)$ and applying an inverse Mellin transform, one obtains, for $X \gg 1$,
$$
	I 
	= 
	\int_{(2)} L(s+1/2,f, \chi) X^s \G(s)ds = \sum_{m \ge 1} \frac{A(m) \chi(m)}{m^{1/2}}e^{-m/X}.
$$
Fix any $\epsilon >0$. Moving the line of integration to 
$\Re (s) = -1+\epsilon$,  and using the functional equation to bound the $L$-series, one obtains
$$
	I 
	= 
	L(1/2, f, \chi) + \mathcal{O}(X^{-1+\epsilon}Q^{1 - 2\epsilon})
	.
$$
 Note that the character $\chi$ does \emph{not} need to be primitive for this to be true.

Doing the same thing for $g$, multiplying, and applying the convexity bound for the $L$-series at $1/2$ gives us
\begin{multline}
	L(1/2, f, \chi)\overline{ L(1/2, g, \chi)} 
	\\
	= 
	\sum_{m_1,m_2 \ge 1} \frac{A(m_1)\overline{B(m_2)} \chi(m_1)\bar \chi(m_2)}{(m_1m_2)^{1/2}}e^{-m_1/X-m_2/X}
	+ 
	\cO (X^{-1+\epsilon}Q^{3/2 - \epsilon})
	.
\end{multline}
Averaging over all characters $\chi$ modulo $Q$ gives us
\begin{multline}\label{SQdef}
	S(Q)
	= 
	S_{f, g}(Q)
	=
	\sum_{m_1\equiv m_2 \pmod Q \atop (m_2,Q)=1}\frac{A(m_1)\overline{B(m_2)}}{(m_1m_2)^{1/2}}e^{-\frac{m_1}{X} - \frac{m_2}{X}}
	+ 
	\mathcal{O}(X^{-1+\epsilon}Q^{3/2 - \epsilon})
	\\
	= 
	\sum_{m \ge 1\atop (m,Q)=1} \frac{A(m)\overline{B(m)} }{m}e^{-\frac{2m}{X}} 
	+
	\sum_{m_2,h_0 \ge 1\atop (m_2,Q)=1}\frac{A(m_2 + h_0Q)\overline{B(m_2)} }{((m_2 + h_0Q)m_2)^{1/2}} e^{\frac{-2m_2}{X} -\frac{h_0 Q}{X}}
	\\
	+ 
	\sum_{m_2,h_0 \ge 1\atop (m_2,Q)=1}\frac{\overline{B(m_2 + h_0Q)} A(m_2) }{((m_2 + h_0Q)m_2)^{1/2}} e^{\frac{-2m_2}{X} -\frac{h_0 Q}{X}} 
	+ 
	\mathcal{O}(X^{-1+\epsilon}Q^{3/2 - \epsilon})
	,
\end{multline}
where $S_{f, g}(Q)$ is defined in \eqref{e:SQorig}. 

By a similar inverse Mellin transform to that used above,
$$
	\sum_{m \ge 1 \atop (m,Q)=1} \frac{A(m)\overline{B(m)}e^{-\frac{2m}{X}}  }{m}
	= 
	\frac{1}{2 \pi i}\int_{(2)} L^{(Q)}(s+1, f \otimes g) \G(s) (X/2)^s ds
	.
$$
Moving the line of integration to $\Re(s) = -1 + \epsilon$, 
if $f \ne g$ this becomes
\be\label{sumpiece}
	\sum_{m \ge 1 \atop (m,Q)=1} \frac{A(m)\overline{B(m)}e^{-\frac{2m}{X}}  }{m} 
	= 
	L^{(Q)}(1,f \otimes g) + \mathcal{O} (X^{-1 + \epsilon})
	,
\ee
while if $f=g$ then
\be\label{cf}
	\sum_{m \ge 1 \atop (m,Q)=1}\frac{A(m) \overline{A(m)}e^{-\frac{2m}{X}}  }{m} 
 	=
 	\left(\prod_{p\mid Q, \text{ prime}} (1-p^{-1})\right) 
 	\frac{L^{(Q)}(1,f,\vee^2)}{\zeta^{(Q)}(2)} \log\left(\frac{X}{2}\right) 
	+ c_f(Q) + \mathcal{O} (X^{-1 + \epsilon})
	.
\ee
Here, for $\Re \, s >1$,
$$
	L^{(Q)}(s,f \otimes g) 
	= 
	\sum_{m\ge 1, (m,Q)=1}\frac{A(m)\overline{B(m)}}{m^s}
$$
and
$$
	\frac{\zeta^{(Q)}(s)}{\zeta^{(Q)}(2s)} L^{(Q)}(s,f,\vee^2)
	= 
	\sum_{m\ge 1, (m,Q)=1}\frac{A(m)\bar A(m)}{m^s}.
$$
The implied constants depend, as usual, upon $f$ and $g$. 
The constant $c_f(Q)$ is given by 
\be\label{e:cfq}
	c_f(Q)
	= 
	\left.\frac{d}{d s} \left(  
	\frac{L^{(Q)}(s, f, \vee^2) }{\zeta^{(Q)}(2s)} \prod_{p\mid Q, \text{ prime}} (1-p^{-s})
	\right)\right|_{s=1} 
	.
\ee

We have shown so far that 
\be\label{e:Stotal}
	S(Q)= S_1(X)+S_2(X)+S_3(X)+S_4(X),
\ee
where
\be\label{S1}
S_1(X)= \left\{\begin{array}{ll}
	L^{(Q)}(1,f \otimes g), & \text{ if } f \ne g \\
 	\left(\prod_{p\mid Q,\atop \text{ prime}} (1-p^{-1})\right) 
	\frac{L^{(Q)}(1,f,\vee^2)}{\zeta^{(Q)}(2)}
	\log\left(\frac{X}{2}\right)  + c_f(Q), & 
	\text{ if } f = g,
 \end{array}\right.
\ee
\be\label{S2}
	S_2(X)
	= 
	\sum_{m_2,h_0 \ge 1 \atop (m_2,Q)=1}
	\frac{A(m_2 + h_0Q)\overline{B(m_2) }
	e^{-\frac{2m_2}{X} - \frac{h_0 Q}{X}}}
	{((m_2 + h_0Q)m_2)^{1/2}},
\ee
\be\label{S3}
	S_3(X)
	= 
	\sum_{m_2,h_0 \ge 1 \atop (m_2,Q)=1}
	\frac{A(m_2)\overline{ B(m_2+h_0Q)} 
	e^{-\frac{2m_2}{X} - \frac{h_0 Q}{X}}}
	{((m_2 + h_0Q)m_2)^{1/2}}
	,
\ee
and 
\be\label{S4}
	S_4(X) 
	\ll 
	X^{-1+\epsilon}Q^{3/2 - \epsilon}
	.
\ee

Note that although the left hand side is independent of $X$, the four contributions on the right hand side all depend on $X$.
Interestingly, in the case $f=g$, as $ X\rightarrow \infty$ the diagonal term,  $S_1(X)$, tends to infinity.  
In other words, 
at least  in the case $f=g$, there must be a corresponding main term contribution from $S_2(X)$ and $S_3(X)$ canceling the multiple of $\log X$ in $S_1(X)$.  
In fact, we will see that $S_2(X)$ and $S_3(X)$ contribute to the main term even when $f \ne g$.

\section{Sieving to obtain complete shifted sums}\label{s:sieving}
Let 
\begin{equation}\label{e:G}
	G\left(1/X\right) 
	:= 
	\frac{1}{2\pi i} \int_{(2)} \tG(s) X^s \; ds
\end{equation}
where $\tG(s)$ is analytic for $\Re(s)> -2$, with the exception of simple poles at $s=-1,0$ with residues $-1$ and $1$ respectively, 
and greater than polynomial decay in vertical strips. 
For example, in the preceding section, we had $\tG(s) = \Gamma(s)$ 
and $G(X) = e^{-X}$. 
Throughout the remainder of this paper, 
we will work with a general pair of functions $\tG(s)$ and $G(X)$ having these properties.  
	
Let $N_0$ and $Q\geq 1$ and assume that $N_0$ is square-free 
and $(Q, N_0)=1$. 
Let $f$ and $g$ be holomorphic cusp forms of even weight $k$, 
for $\Gamma_0(N_0)$, with normalized Fourier coefficients $A(m)$, $B(m)$, as in \eqref{def1}, so that $A(1) = B(1)=1$.
Assume that $f$ and $g$ are eigenfunctions for Hecke operators $T_p$ 
with primes $p\nmid N_0$.
 
For any $d\mid Q$, $d\geq 1$, define
$$
	S_{d}\left(X/d;Q/d\right)
	:=
	\sum_{h, m\geq 1} 
	\frac{A(d \left(m+h\frac{Q}{d} \right))\overline{B(dm)}}{(m( m+h\frac{Q}{d}))^{\frac{1}{2}}} 
	G\left(\frac{ m+h\frac{Q}{d}}{\frac{X}{d}}\right) G\left(\frac{m}{\frac{X}{d}}\right)
	.
$$
For $\ell_1, \ell_2\geq 1$, we define a ``building block" shifted sum as follows:
\be\label{e:Sbulidingblock}
	S(X,Q,\ell_1,\ell_2)
	:=
	\sum_{m_1, m_2 \geq 1\atop \ell_1m_1 -\ell_2m_2 \equiv 0\mod Q} 
	\frac{A(m_1)\overline{B(m_2)}}
	{(\ell_1m_1)^{\frac{1}{2}}(\ell_2m_2)^{\frac{1}{2}}}
	G\left(\frac{\ell_1 m_1}{X}\right) G\left(\frac{\ell_2 m_2}{X}\right)
	.
\ee

The objective of this section is to prove
\begin{prop}\label{prop:Sblockdef}

We have
\begin{multline*}
	S_2(X)
	=\sum_{d\mid Q} \frac{\mu(d)}{d} S_d\left(X/d; Q/d\right)
	\\
	=
	\sum_{d\mid Q} 
	\frac{\mu(d)}{d} 
	\sum_{d_1, d_2\mid d}
	\mu(d_1)\mu(d_2) 
	A(d/d_1)\overline{B(d/d_2)}
	S\left(\frac{X}{d}, \frac{Q}{d}, d_1, d_2\right)
	\end{multline*}
	where $S_2(X)$ is defined in (\ref{S2}). 
\end{prop}

\subsection{Proof of Proposition \ref{prop:Sblockdef}}
By applying the usual sieving method, we get 
$$
	S_2(X)
	=
	\sum_{d\mid Q} \frac{\mu(d)}{d} S_d \left(\frac{X}{d}; \frac{Q}{d}\right)
	.
$$
Our expression for $S_2(X)$ will follow from a reduction of 
$S_{d}(X;Q)$ to the $S(X;Q,\ell_1,\ell_2)$.

Recall  that for any prime $p \nmid N_0$, $f$ is an eigenfunction 
of the Hecke operator $T_p$, i.e., 
$$
	(T_pf)(z)
	:=
	\frac{1}{p}\sum_{b\mod p} f\left(\frac{z+b}{p}\right)
	+
	p^{k-1} f(pz)
	=
	A(p)p^{\frac{k-1}{2}} \cdot f(z)
	.
$$
For any square free $d\geq 1$, with $(d, N_0)=1$, define
$$
	f_d(z)
	:= 
	\sum_{n\geq 1} A(dn)n^{\frac{k-1}{2}} e^{2\pi inz}\,.
$$
\begin{lem}
	Let $d\geq 1$ be a square-free integer, with $(d, N_0)=1$.
	Then we have
	\begin{equation}\label{e:fd_formula}
	f_d(z)
	=
	\sum_{d_0\mid d} 
	\mu(d_0) 
	A(d/d_0) d_0^{\frac{k-1}{2}} f(d_0 z)
	\,.
	\end{equation}
\end{lem}
\begin{proof}
For any prime $p\nmid N_0$, we have
\begin{multline*}
	f_p(z)
	=
	p^{-\frac{k-1}{2}} \sum_{n\geq 1} A(n) n^{\frac{k-1}{2}} e^{2\pi i\frac{nz}{p}}
	\left( \frac{1}{p} \sum_{b\mod p} e^{2\pi i\frac{nb}{p}}\right)
	\\
	=
	p^{-\frac{k-1}{2}}
	\left\{ (T_p f)(z) - p^{k-1} f(pz) \right\}
	=
	A(p)\cdot f(z) -p^{\frac{k-1}{2}} f(pz)
	\,.
\end{multline*}
So, \eqref{e:fd_formula} is true for $d=p$. 

Assume that (\ref{e:fd_formula}) is true for a square-free integer $d\geq 1$, with $(d, N_0)=1$. 
For any prime $p\nmid d N_0$, we have
\begin{multline*}
	f_{dp}(z)
	= 
	A(p) f_d(z) - p^{\frac{k-1}{2}} f_d(pz)
	\\
	= 
	A(p) \sum_{d_0\mid d} \mu(d_0) A\left(d/d_0\right) d_0^{\frac{k-1}{2}} f(d_0 z)
	- \sum_{d_0\mid d} \mu(d_0) A\left(d/d_0\right) (d_0p)^{\frac{k-1}{2}} f(d_0pz)
	\\
	=
	\sum_{d_0\mid dp}\mu(d_0) A(dp/d_0) d_0^{\frac{k-1}{2}} f(d_0z)
	.
\end{multline*}
\end{proof}

%
By \eqref{e:fd_formula}, for any square-free $d\mid Q$, we have
\begin{multline*}
	A(d(m+h Q/d))\overline{B(dm)}
	\\
	=
	\sum_{d_1\mid \left(d, m+h Q/d\right)} \sum_{d_2\mid (d, m)}
	\mu(d_1) \mu(d_2) A(d/d_1) \overline{B(d/d_2)} 
	A\left((m+hQ/d)/d_1\right) \overline{B(m/d_2)}
	\,.
\end{multline*}
So, 
\begin{multline*}
	S_d(X/d; Q/d)
	=
	\sum_{h, m\geq 1} \sum_{d_1\mid (d, m+hQ/d)} \sum_{d_2\mid (d, m)}
	\mu(d_1)\mu(d_2) A(d/d_1) \overline{B(d/d_2)} 
	\\
	\times
	\frac{A\left((m+hQ/d)/d_1\right) \overline{B\left(m/d_2\right)}} {(m(m+hQ/d))^{\frac{1}{2}}}
	G\left(\frac{m+hQ/d}{X/d}\right) G\left(\frac{m}{X/d}\right)
	\\
	=
	\sum_{d_1, d_2\mid d} \mu(d_1) \mu(d_2) A(d/d_1) \overline{B(d/d_2)}
	\\
	\times
	\sum_{m_1, m_2 \geq 1, \atop d_1m_1-d_2m_2 \equiv 0\mod Q/d}
	\frac{A(m_1) \overline{B(m_2)}}{(d_1m_1)^{\frac{1}{2}} (d_2 m_2)^{\frac{1}{2}}}
	G\left(\frac{d_1m_1}{X/d}\right) G\left(\frac{d_2 m_2}{X/d}\right)
	.
\end{multline*}		
So we prove Proposition \ref{prop:Sblockdef}.

\section{Shifted sums and shifted convolutions}\label{s:shiftedsum_conv}

Let $N_0$ and $Q\geq 1$ and assume that $N_0$ is square-free 
and $(Q, N_0)=1$. 
Let $f$ and $g$ be holomorphic cusp forms of even weight $k$, 
which are newforms for $\Gamma_0(N_0)$, 
with normalized Fourier coefficients $A(m)$, $B(m)$, as in \eqref{def1}, so that $A(1) = B(1)=1$.
Assume that $f$ and $g$ are eigenfunctions for all Hecke operators.

Recall our building block shifted sum \eqref{e:Sbulidingblock} is
$$
	S(X, Q,\ell_1,\ell_2)
	=
	\sum_{h, m\geq 1\atop \ell_1m_1 -\ell_2m_2 = hQ} 
	\frac{A(m_1)\overline{B(m_2)}}{(\ell_1m_1)^{\frac{1}{2}}(\ell_2m_2)^{\frac{1}{2}}}
	G\left(\frac{\ell_1m_1}{X}\right) G\left(\frac{\ell_2m_2}{X}\right)
	.
$$
We assume that $\ell_1$ and $\ell_2$ are square-free. 
Then $N=N_0\frac{\ell_1\ell_2}{(\ell_1,\ell_2)}$ is also square-free. 
Define
$$
	(\ell_1\ell_2)^{\frac{k-1}{2}}
	S\left(1-\frac{k}{2}, \frac{1}{2}; \delta\right)
	:=
	(\ell_1\ell_2)^{\frac{k-1}{2}} \left(S_{\rm cusp}\left(1-\frac{k}{2}, 1/2; \delta\right)
+ S_{\rm cont,int}\left(1-\frac{k}{2}, 1/2; \delta\right)\right)
	,
$$
where $S_{\rm cusp}$ and $S_{\rm cont, int}$ are given in \eqref{Scuspdeltadef} 
and \eqref{ScontIntdeltadef}, respectively. 

The objective of this section and the following two is to prove
\begin{thm}\label{theorem:SQ}

When $f\neq g$, for any $Q\geq 1$ with $(N_0, Q)=1$,  we have
\begin{multline*}
	S(X, Q, \ell_1, \ell_2)
	=
	\frac{1}{4}
	\frac{(\ell_1, \ell_2)}{\ell_1\ell_2}
	L(1, f\otimes g)
	\sum_{a\mid N} 
	\left\{
	\left(\frac{1}{2^{r(N)}} + \frac{1}{Q} \prod_{p\mid a, p^\alpha\|Q , \atop \alpha\geq 0} (p^\alpha-1) \prod_{p^\alpha\|Q, p\nmid N, \atop \alpha\geq 1} p^\alpha\right)
	\right.
	\\
	\times
	\left(\frac{N_0}{(a, N_0)}\right) A\left(\frac{N_0}{(a, N_0)}\right) 
	\overline{B\left(\frac{N_0}{(a, N_0)}\right)}
	\\
	\left.
	\times
	\left(\prod_{p\mid \frac{\ell_1 (a, \ell_2)}{(\ell_1, \ell_2) (a, \ell_1)}}
	\frac{ A(p)-B(p)p^{-1}}{1-p^{-2}}\right)
	\left( \prod_{p\mid \frac{\ell_2 (a, \ell_1)}{(\ell_1, \ell_2) (a, \ell_2)}}
	\frac{ B(p)-A(p)p^{-1}}{1-p^{-2}} \right)
	\right\}
	\\
	+
	\lim_{\delta\to 0}
	(\ell_1\ell_2)^{(k-1)/2} \left(S_{\rm cusp}\left(1-\frac{k}{2}, \frac{1}{2}; \delta\right)
+ S_{\rm cont,int}\left(1-\frac{k}{2}, \frac{1}{2} ; \delta\right)\right)	
	\\
	+ \mathcal{O}\left((\ell_1\ell_2)^\epsilon Q^{\theta +\epsilon-\frac{1}{2}}\right) + \mathcal{O}\left(X^{-\epsilon}\right)
	.
\end{multline*}

When $f=g$, there exists a constant $C_2(k)$, given in \eqref{C1def},
	independent of $f$ (but depending on $k$), such that, 
	for any $Q\geq 1$ with $(N_0, Q)=1$, we have 
\begin{multline*}
	S(X, Q, \ell_1, \ell_2)
	\\
	=
	- 
	\frac{1}{2}
	\Res_{s=1} L(s,f \otimes f; \ell_1, \ell_2) 
	\log X
	+
	\Res_{s=1} L(s, f\otimes f; \ell_1, \ell_2) 
	\log Q
	\\
	+
	\left\{ C_2(k) 
	- \frac{1}{4}\sum_{p\mid N} \log p \frac{1}{2(p-1)}
	+ \frac{1}{2}\frac{\Gamma'(k)}{\Gamma(k)} - \frac{1}{2} \log (4\pi)
	\right\} 
	\Res_{s=1} L(s, f \otimes f;\ell_1, \ell_2)
	\\
	+
	\left\{
	\frac{1}{2}
	\log (N, (\ell_1, \ell_2))
	+\frac{1}{4}\sum_{p\mid(\ell_1, \ell_2), p^\alpha\|Q, \atop \alpha\geq 0} 
	\log p \left(p^{-\alpha} -\frac{3}{2}\right)
	\right\}
	\Res_{s=1} L(s, f \otimes f;\ell_1, \ell_2)
	\\
	+
	\left\{ 
	\frac{1}{4}\sum_{p\mid N, p^\alpha\|Q, \atop \alpha\geq 0} 
	\log p \left(p^{-\alpha} \frac{3p-1}{p-1}-1\right)
	- \sum_{p^\alpha\|Q, \alpha\geq 1} \log p \frac{1-p^{-\alpha}}{p-1}
	\right\}
	\Res_{s=1} L(s, f \otimes f;\ell_1, \ell_2)
	\\
	+ 
	\frac{1}{2}
	\left.\frac{d}{ds} \left( (s-1) L(s, f\otimes f; \ell_1, \ell_2)\right)\right|_{s=1}	
	\\
	+
	\lim_{\delta\to 0}
	(\ell_1\ell_2)^{(k-1)/2} \left(S_{\rm cusp}\left(1-\frac{k}{2}, \frac{1}{2}; \delta\right)
+ S_{\rm cont,int}\left(1-\frac{k}{2}, \frac{1}{2} ; \delta\right)\right)
	\\
	+ 
	\mathcal{O}\left((\ell_1\ell_2)^\epsilon Q^{\theta +\epsilon-\frac{1}{2}}\right) 
	+ 
	\mathcal{O}\left(X^{-\epsilon}\right)
	.
\end{multline*}	
Here $p$ is a prime. 
\end{thm}
\begin{proof}
In Proposition \ref{p:SQandMDS}, we define $S_Q(X; \delta)$ and show that 
$S(X, Q, \ell_1,\ell_2) = \lim_{\delta\to 0} S_Q(X; \delta)$. 
As discussed in Section \ref{ss:separation_cusp_cont}, 
$S_Q(X; \delta)$ can be separated into two pieces:
$$
	S_Q(X;\delta)
	=
	(\ell_1\ell_2)^{\frac{k-1}{2}}
	\left(S_Q^{({\rm cusp})} (X; \delta)
	+ S_Q^{({\rm cont})} (X; \delta)
	\right)
	.
$$
By Proposition \ref{prop:cusp1} and Proposition \ref{prop:cusp2}, we get the contributions coming from the cuspidal spectrum $S^{({\rm cusp})}_Q(X; \delta)$.
By Proposition \ref{prop:cont}, we get the contributions coming from the continuous spectrum $S_Q^{({\rm cont})}(X; \delta)$. 
\end{proof}

Note the presence in Theorem \ref{theorem:SQ} of a mysterious object:
$$
	\lim_{\delta \rightarrow 0} (\ell_1\ell_2)^{\frac{(k-1)}{2}} 
	\left(S_{\rm cusp}\left(1-\frac{k}{2}, \frac{1}{2} ; \delta\right) 
	+ S_{\rm cont,int}\left(1-\frac{k}{2}, \frac{1}{2} ; \delta\right)\right)
	.
$$
This is the non-singular part of the special value of $Z_Q(s,w)$ at $(s,w) = (1-k/2,1/2)$.  
We ultimately show that the special value is small on average. 
However, we are at this moment unable to say very much about this for any particular fixed $Q$.

\subsection{The relation to a shifted multiple Dirichlet series}
In this section we will prove

\begin{prop}\label{p:SQandMDS}
Define $S_Q(X; \delta)$ by suppressing the $\ell_1,\ell_2$ and writing
\begin{multline*}
	S_Q(X; \delta)
	:=
	\left(\frac{1}{2\pi i}\right)^3
	\int_{(\hf +2\epsilon)} \int_{(2)} \int _{(1+\epsilon)}
	Z_Q\left(s, u; \delta\right)
	\tG \left(s-w-1+u+\frac{k-1}{2}\right)
	\tG (w) 
	\\
	\times
	X^{s-\frac{3}{2}+u+\frac{k}{2}} 
	\frac{
	\Gamma\left(w-u+\frac{1}{2}\right) 
	\Gamma\left(u+\frac{k-1}{2}\right)}
	{\Gamma\left(\frac{k}{2}+w\right)} 
	\; du\; ds\; dw
	.
\end{multline*}
Then 
$$
	S(X,Q,\ell_1,\ell_2)= \lim_{\delta \rightarrow 0} S_Q(X; \delta).
$$
\end{prop}
\begin{proof}
A few algebraic manipulations give us
\begin{multline*}
	S(X,Q,\ell_1,\ell_2)
	=
	\sum_{h, m_2\geq 1\atop \ell_1m_1 -\ell_2m_2 = hQ} \frac{A(m_1)\overline{B(m_2)}}{(\ell_1m_1)^{\frac{1}{2}}(\ell_2m_2)^{\frac{1}{2}}}
	G\left(\frac{\ell_1m_1}{X}\right) G\left(\frac{\ell_2m_2}{X}\right)
	\\
	=\left(\ell_1\ell_2\right)^{\frac{(k-1)}{2}}
	\left(\frac{1}{2\pi i}\right)^2 \int_{(2)} \int _{(2)}\sum_{h, m_2\geq 1, \atop \ell_1m_1 - \ell_2m_2 =hQ} 
	\frac{a(m_1) b(m_2) \tG (s) \tG (w) X^{s+w}}
	{(\ell_1m_1)^{\frac{k}{2}+w}(\ell_2m_2)^{\frac{k}{2}+s}}
	\; ds\; dw
	\\
	=\left(\ell_1\ell_2\right)^{\frac{(k-1)}{2}}
	\left(\frac{1}{2\pi i}\right)^2 \int_{(k/2)} \int_{(2)} \sum_{h, m_2\geq 1, \atop \ell_1m_1 - \ell_2m_2 =hQ} 
	\frac{a(m_1) b(m_2) \tG(s) \tG(w) X^{s+w}}
	{(\ell_2m_2)^{s+w+k}\left(1+\frac{hQ}{\ell_2m_2}\right)^{\frac{k}{2}+w}}\; ds\; dw.
\end{multline*}
In \cite{GR}, 6.422(3), the following identity can be found, after changing the signs of $\gamma$ and $u$.  For $\Re(\beta) > \gamma>0$  and $\text{arg}(t) < \pi$:
$$
	\frac{1}{2 \pi i} \int_{(\gamma)}\G(u)\G(\beta -u) t^{-u}du 
	= 
	\G(\beta)(1+t)^{-\beta}.
$$
Applying this to the above, with $\beta = k/2+w$, $\gamma = (k+1)/2 + \epsilon$, and  $t = hQ/m$, we have
\begin{multline*}
	S(X,Q,\ell_1,\ell_2)
	\\
	=
	\left(\ell_1\ell_2\right)^{\frac{(k-1)}{2}}
	\left(\frac{1}{2\pi i}\right)^3 
	\int_{(k/2)} \int_{(2)} \int_{(\gamma)} \sum_{h, m_2\geq 1, \atop \ell_1m_1 - \ell_2m_2 =hQ} 
	\frac{a(m_1) b(m_2) }{(\ell_2m_2)^{s+w+k-u}(hQ)^u}
	\\
	\times
	\frac{\Gamma\left(\frac{k}{2}+w-u\right)\Gamma(u) \tG(s) \tG(w) X^{s+w}}
	{\Gamma\left(\frac{k}{2}+w\right)} 
	\; du\; ds\; dw
	\\
	=
	\left(\frac{1}{2\pi i}\right)^3 
	\int_{(k/2)} \int_{(2)} \int_{(\gamma)} 
	Z_Q \left(s+w+1-u, u+\frac{1-k}{2} ; \ell_1, \ell_2 \right)
	\\
	\times
	\frac{\Gamma\left(\frac{k}{2}+w-u\right)\Gamma(u) \tG(s) \tG(w) X^{s+w}}
	{\Gamma\left(\frac{k}{2}+w\right)} 
	\; du\; ds\; dw 
	.
\end{multline*}
Both arguments of $Z_Q$ are greater than $1$ so the sum defining $Z_Q$ converges absolutely and 
\begin{multline*}
	S(X,Q,\ell_1,\ell_2) 
	=  
	\lim_{\delta \rightarrow 0}
	\left(\frac{1}{2\pi i}\right)^3 
	\int_{(k/2)} \int_{(2)} \int_{(\gamma)} Z_Q \left(s+w+1-u, u+\frac{1-k}{2};\delta \right)
	\\
	\times
	\frac{\Gamma\left(\frac{k}{2}+w-u\right)\Gamma(u) \tG (s) \tG (w) X^{s+w}}
	{\Gamma\left(\frac{k}{2}+w\right)} 
	\; du\; ds\; dw
	.
\end{multline*}

Changing variables via $s' = s+w+1-u$,  $u'=u+(1-k)/2$
and dropping the primes brings the integrand to the desired form as stated in the Proposition, after which the lines of integration can be shifted, without passing over poles, to
$$
	\Re(u)= 1+\epsilon, \; \Re(s)=2, \; \Re(w) = \frac{1}{2}+2\epsilon
	.
$$
Note that we now have  $1< \Re(u) < \Re(w)+\frac{1}{2}$, $\Re(s)>1$ 
and $\Re\left(s-w-1+u+(k-1)/2\right)>0$. 
Thus the lines of integration of the $g$ functions are to the right of their poles and $Z_Q$ remains in the range of absolute convergence.
\end{proof}

\subsection{The separation into cuspidal and continuous parts}\label{ss:separation_cusp_cont}
	
Recalling the spectral expansions \eqref{Zdeltainit}--\eqref{Psi}, write
$$
	S_Q(X; \delta)
	= 
	(\ell_1\ell_2)^{\frac{(k-1)}{2}}\left(S_Q^{(\text{cusp})}(X; \delta)
	+
	S_Q^{(\text{cont})}(X; \delta)\right)
	,
$$
with
\begin{multline}\label{e:SQcuspXdelta}
	S_Q^{({\rm cusp})}(X; \delta)
	:=
	\left(\frac{1}{2\pi i}\right)^3
	\int_{(\hf +2\epsilon)} \int_{(2)} \int _{(1+\epsilon)}
	S_{\rm cusp}(s, u;\delta)
	\\
	\times
	\tG \left(s-w-1+u+\frac{k-1}{2}\right)
	\tG (w) 
	X^{s-\frac{3}{2}+u+\frac{k}{2}} 
	\frac{
	\Gamma\left(w-u+\frac{1}{2}\right) 
	\Gamma\left(u+\frac{k-1}{2}\right)}
	{\Gamma\left(\frac{k}{2}+w\right)} 
	\; du\; ds\; dw
\end{multline}
and
\begin{multline}\label{e:SQcontXdelta}
	S_Q^{({\rm cont})}(X; \delta)
	:=
	\left(\frac{1}{2\pi i}\right)^3
	\int_{(1/2 +2\epsilon)} \int_{(2)} \int _{(1+\epsilon)}
	S_{\rm cont}(s, u;\delta)
	\\
	\times
	\tG \left(s-w-1+u+\frac{k-1}{2}\right) 
	\tG (w) 
	X^{s-\frac{3}{2}+u+\frac{k}{2}} 
	\frac{
	\Gamma\left(w-u+\frac{1}{2}\right) 
	\Gamma\left(u+\frac{k-1}{2}\right)}
	{\Gamma\left(\frac{k}{2}+w\right)} 
	\; du\; ds\; dw
	.
\end{multline}	
	
We will treat first $S_Q^{(\text{cusp})}(X; \delta)$, 
	and then $S_Q^{(\text{cont})}(X; \delta)$.

\subsection{The case of $S_Q^{(\text{cusp})}(X; \delta)$}

We begin by moving the $s$ line of integration in \eqref{e:SQcuspXdelta} 
to $\Re(s) = (1-k)/2-2\epsilon$.  
The moved integral is now $\mathcal{O}(X^{-\epsilon})$.
In the process we pass over poles of  $S_{\rm cusp}(s, w;\delta)$ at
\begin{enumerate}
	\item $s-w-3/2+u+k/2=0, -1, \ldots$ from $\tG(s-w-3/2+u+k/2)$, 
	
	\item $s=1/2-r+it_j$, for $0\leq r\leq \frac{k}{2}$ 
	from $S_{\rm cusp}(s, w;\delta)$. 
\end{enumerate}

In fact, $s-w-3/2+u+k/2=0$ corresponds to $\Re(s) = 1+ \epsilon -k/2$, and $s-w-3/2+u+k/2=-1$ to $\Re(s) =  \epsilon -k/2 < 1/2 - k/2$,
so only the pole at 0 will be passed over.
Considering this first,  we have $s=w-u+3/2-k/2$. 
The contribution to $S_Q^{({\rm cusp})}(X; \delta)$ from this pole is
\begin{multline*}
	S_{Q,(1)}^{(\text{cusp})}(X; \delta)
	= 
	\left(\frac{1}{2\pi i}\right)^2 
	\int_{\left(\frac{1}{2}+2\epsilon\right)} \int_{(1+\epsilon)}
	S_{\rm cusp}\left(w-u+\frac{3}{2}-\frac{k}{2}, u; \delta\right)
	\\
	\times
	\tG (w) X^w 
	\frac{\Gamma\left(w-u+\frac{1}{2}\right)\Gamma\left(u+\frac{k-1}{2}\right)}
	{\Gamma\left(\frac{k}{2}+w\right)}
	\; du\; dw
	.
\end{multline*}
In this region
	$\Re\left(w-u+3/2-k/2\right)
	=1-k/2+\epsilon$, 
	$\Re\left(w-u+1/2 \right) = \epsilon$
	and 
	$\Re\left(u+(k-1)/2\right) = (1+k)/2+\epsilon$.

\subsubsection{Analysis of $S_{Q,(1)}^{(\text{cusp})}(X; \delta)$}

We will show 
\begin{prop}\label{prop:cusp1}
$$
	(\ell_1\ell_2)^{\frac{k-1}{2}}
	S_{Q,(1)}^{({\rm cusp})}(X; \delta)
	=
	(\ell_1\ell_2)^{\frac{(k-1)}{2}} 
	S_{\rm cusp} \left(1-\frac{k}{2}, \frac{1}{2} ; \delta\right)
	+ 
	\mathcal{O}(X^{-\epsilon}) 
	+\mathcal{O}\left((\ell_1\ell_2)^\epsilon Q^{\theta+\epsilon-\frac{1}{2}} \right).
$$
\end{prop}
\begin{proof}
Move $w$ to $\Re(w)=-\epsilon$.      
Poles of $S_{\rm cusp}\left(w-u+3/2-k/2, u; \delta\right)$, of $\tG(w)$,
and of $\Gamma\left(w-u+1/2\right)$ are passed over in this process.  
The shifted integral is $\cO(X^{-\epsilon})$.  
The residues at the poles are computed as follows.  
Recalling \eqref{Scuspdeltadef}, we have
\begin{multline*}
	S_{\rm cusp}\left(w-u+\frac{3}{2}-\frac{k}{2}, u; \delta\right)
	\\
	=
	\frac{(4\pi)^k 2^{w+1-u-\frac{k}{2}}}
	{2\sqrt{\pi}\Gamma\left(w+\frac{1}{2} -u+\frac{k}{2}\right) }
	\sum_j L_Q\left(w+\frac{1}{2}, \overline{u_j }\right) 
	M(w+\frac{3}{2}-u-\frac{k}{2}, t_j, \delta) 	
	\overline{ \left<V, u_j\right>}
	.
\end{multline*}

The poles of $S_{\rm cusp}\left(w-u+3/2-k/2, u; \delta\right)$ 
occur when $w-u+3/2-k/2 = 1/2-r +it_j$.
As $\Re \left(w-u+3/2-k/2 \right) = \Re(w) +1/2-k/2 -\epsilon$, 
exactly one pole is passed, when $w= u-1+it_j$, i.e when $\Re(w) = \epsilon$.  
Referring to Proposition \ref{prop:Mbounds},
$$
	\Res_{w-u+\frac{3}{2}-\frac{k}{2} = \frac{1}{2} -r +it_j} 
	S_{\rm cusp} \left(w-u+\frac{3}{2}-\frac{k}{2}, u; \delta\right) 
	= c_{r,j}(\delta)L_Q\left(u-\frac{1}{2}+it_j, \overline{u_j }\right), 
$$
and the contribution from this term is
\begin{multline*}
	\frac{1}{2\pi i} \int_{(1+\epsilon)}
	c_{r,j}(\delta)L_Q\left(u-\frac{1}{2} +it_j, \overline{u_j }\right)
	\tG(u-1+it_j) X^{u-1+it_j} 
	\\
	\times
	\frac{\Gamma\left(-\frac{1}{2}+it_j\right)
	\Gamma\left(u+\frac{k-1}{2}\right)}
	{\Gamma\left(\frac{k}{2}+u-1+it_j\right)}
	\; du,
\end{multline*}
with $\lim_{\delta \rightarrow 0}  c_{r,j}(\delta)= c_{r,j} $ and
\be\label{crjdef}
	c_{r,j} 
	=
	\frac{(-1)^r (4\pi)^k\overline{\<V,u_j\>}\Gamma(\hf - it_j +r)\Gamma(2it_j -r)}{2r!\Gamma(\hf + it_j)\Gamma(\hf-i t_j) \Gamma(k-\hf + it_j -r)}
	.
\ee

Continuing the calculation, move $u$ to $\Re(u)= 1-\epsilon$.  
The shifted integral is now $\cO (X^{-\epsilon})$ and a pole of $g(u-1+it_j)$ is passed at $u = 1-it_j$,  with residue 
$$
	c_{r,j}(\delta) L_Q\left(\frac{1}{2}, \overline{u_j }\right)
 	\frac{\Gamma\left(-\frac{1}{2}+it_j\right)\Gamma\left(\frac{k+1}{2}-it_j\right)}
	{\Gamma\left(\frac{k}{2}\right)}.
$$
By (7.16) of \cite{HH}
\be\label{targetend}
	\sum_{|t_j| \sim T}L_Q({s'},\overline{u_j})\overline{\<V,u_j\>} 
	\ll 
	Q^{-s'}(\ell_1\ell_2)^{(1-k)/2+\epsilon} (1 + |\gamma'|+ |T|)^{1+k+\epsilon}
	.
\ee
Here $s'=1/2$ and $\gamma' =0$, and it follows that $(\ell_1\ell_2)^{(k-1)/2}$ times
the sum of these residues converges absolutely and is 
$\cO\left((\ell_1\ell_2)^{\epsilon} Q^{-1/2}\right)$.
	
The function $\tG (w)$ has a pole at $w=0$, contributing the residue
\be\label{temp43}
	\frac{1}{2\pi i} \int_{(1+\epsilon)}
	S_{\rm cusp}\left(-u+\frac{3}{2}-\frac{k}{2}, u; \delta\right)
	\frac{\Gamma\left(-u+\frac{1}{2}\right)\Gamma\left(u+\frac{k-1}{2}\right)}{\Gamma\left(\frac{k}{2}\right)}du
\ee
When $\Re(u)= 1+\epsilon$, $\Re \left(-u+3/2-k/2\right) = 1/2 -\epsilon -k/2$.   
In this region the spectral expansion for 
$S_{\rm cusp}\left(-u+3/2-k/2, u; \delta\right)$ 
converges absolutely as $\delta \rightarrow 0$ and
\begin{multline*}
	S_{\rm cusp}\left(-u+\frac{3}{2}-\frac{k}{2}, u; \delta\right)\\
	=
	\frac{(4\pi)^k 2^{-u+1-\frac{k}{2}}}
	{2\sqrt{\pi}\Gamma(-u+\frac{1}{2}+\frac{k}{2}) }
	\sum_j L_Q\left(1/2, \overline{u_j }\right) 
	M\left(-u+\frac{3}{2}-\frac{k}{2}, t_j, \delta \right) 	
	\overline{ \langle V, u_j \rangle}
	.
\end{multline*}
The bound  \eqref{Z2} implies that 
$$
	(\ell_1\ell_2)^{(1-k)/2}S_{\rm cusp}\left(-u+\frac{3}{2}-\frac{k}{2}, u; \delta\right) 
	\ll 
	(1 + |u|)^{1+k+\epsilon}\left(\frac{(\ell_1\ell_2)^{\epsilon}}{\sqrt Q}\right),
$$
and $(\ell_1\ell_2)^{(1-k)/2}$ times the integral \eqref{temp43} is consequently $\mathcal{O}\left((\ell_1\ell_2)^{\epsilon} Q^{-1/2}\right)$.

The pole of $\Gamma\left(w-u+1/2 \right)$  at $w = u-1/2$ contributes
$$
	\frac{1}{2\pi i} \int_{(1+\epsilon)}
	S_{\rm cusp}\left(1-\frac{k}{2}, u; \delta\right)
	\tG(u-1/2)x^{u-1/2}du
	.
$$
Moving $u$ to $\Re(u) =1/2-\epsilon$, the moved integral contributes $\cO (X^{-\epsilon})$, and a simple pole of $\tG(u-1/2)$ at $u=1/2$ contributes the residue
$$
	S_{\rm cusp}\left(1-\frac{k}{2}, \frac{1}{2} ; \delta\right)
	.
$$
\end{proof}

\subsubsection{Analysis of $S_{Q,(2)}^{(\text{cusp})}(X; \delta)$}
Now write the contribution from the poles at $s = 1/2-r+it_j$ as
\begin{multline*}
	S_{Q,(2)}^{({\rm cusp})}(X; \delta) 
	= \sum_{j, \atop 0 \le r \le k/2}
	\left(\frac{1}{2\pi i}\right)^2 
	\int_{\left(\frac{1}{2}+2\epsilon\right)} \int_{(1+\epsilon)}
	c_{r, j}(\delta)L_Q\left(u+\frac{k-1}{2}-r +it_j, \overline{u_j}\right)
	\\
	\times 
	\tG \left(-r+it_j -w-1+u+\frac{k}{2}\right) \tG (w)  
	X^{-1-r+it_j +u+\frac{k}{2}} 
	\frac{\Gamma\left(w-u+\frac{1}{2}\right) \Gamma\left(u+\frac{k-1}{2}\right)}
	{\Gamma\left(\frac{k}{2}+w\right)}\; du \; dw
	.
\end{multline*}
We will show 
\begin{prop}\label{prop:cusp2}
$$
	(\ell_1\ell_2)^{\frac{(k-1)}{2}} S_{Q,(2)}^{({\rm cusp})}(X; \delta)
	= 
	\mathcal{O}(X^{-\epsilon}) 
	+\mathcal{O}\left((\ell_1\ell_2)^\epsilon Q^{-\frac{1}{2}+\theta+\epsilon}\right)
	.
$$
\end{prop}
\begin{proof}	
For $r$ such that $0\le r \le k/2-1$, move the $u$ line to 
$\Re \left( -1-r+it_j +u+k/2\right) = -\epsilon$.   
A pole of $g$ is passed over at $-r+it_j -w-1+u+k/2=0$, with residue
$$
	\frac{1}{2\pi i} 
	\int_{\left(\frac{1}{2}+2\epsilon\right)}
	c_{r, j}(\delta)L_Q\left(w+1/2, \overline{u_j}\right)
	 \tG (w)  
	X^{w} 
	\frac{\Gamma\left(\frac{k-1}{2}-r+it_j\right) \Gamma\left(w+\frac{1}{2}-it_j+r\right)}
	{\Gamma\left(\frac{k}{2}+w\right)} \; dw
	.
$$
Note that the argument of $\G(u + (k-1)/2)$ always has positive real part  as
$\Re (u + (k-1)/2) = 1/2 -\epsilon +r -\Re(it_j) \ge 1/2 -\epsilon  -\Re(it_j) \ge 1/4-\epsilon$. 
The last inequality follows from the original Selberg  bound $|\Re(it_j)| <1/4$.

Moving the $w$ line, to $\Re (w) = -\epsilon$, 
the shifted integral is $\cO(X^{-\epsilon})$ 
and there is one residue contributed by the pole of $\tG (w)$ at $w=0$.  
This is equal to $c_{r, j}(\delta)L_Q\left(1/2, \overline{u_j}\right)$ 
and $(\ell_1\ell_2)^{(k-1)/2}$ times the sum of these over the $j$, 
and $0\le r \le k/2-1$ is, as remarked above, 
$\cO \left((\ell_1\ell_2)^\epsilon Q^{-1/2} \right)$.

In the case $r = k/2$, the original argument of $g$ has real part 
$\break \Re \left( -k/2 +it_j -(1/2 + 2 \epsilon)-1+(1 + \epsilon)+k/2 \right) 
= \Re (it_j) -1/2 -\epsilon$, 
and $\G(w-u+1/2)$ has argument with real part equal to $\epsilon$.     
Assuming $\Re(it_j) \ge 0$, decrease $u$ to $\Re (u + it_j) = 1-\epsilon$.  
The new argument of $g$ has real part $\Re (it_j -w-1+u) = -1/2 -3\epsilon$ 
and the new argument of $\G(w-u+1/2)$ has real part greater than 0.  
Thus no additional poles have been passed over and the remaining integral is 
$\cO (X^{-\epsilon})$.
\end{proof}

This completes the analysis of $S_Q^{(\text{cusp})}(X; \delta)$.

\subsection{The case of $S_Q^{(\text{cont})}(X; \delta)$}

The main term arising in the off-diagonal piece originates in the continuous part of the spectrum. 
The Fourier coefficients of Eisenstein series at different cusps are an important part of this contribution. 
In the following section, we will review some basic properties of the Eisenstein series at various cusps. 

\subsubsection{Properties of Eisenstein series}

We use the same notation as in Section \ref{s:review}. 
Recall we're assuming that the level $N$ is square-free. 
By \cite{Iwa02}, for any cusps $\cusp, \cuspb\in \Q$, 
we have the following Fourier expansion at cusp $\cuspb$ 
of the Eisenstein series for cusp $\cusp$:
\begin{multline*}
	E_\cusp(\sigma_\cuspb z, s)
	=
	\delta_{\cusp, \cuspb} y^s
	+
	\frac{\pi^{-s+\frac{1}{2}}\Gamma\left(s-\frac{1}{2}\right)}
	{\pi^{-s}\Gamma(s)}
	\rho_{\cusp\cuspb}(s) y^{1-s}
	\\
	+
	\frac{2}{\pi^{-s}\Gamma(s)}
	\sum_{n\neq 0} \rho_{\cusp\cuspb}(s, n)|n|^{s-\frac{1}{2}}
	\sqrt{y} K_{s-\frac{1}{2}}(2\pi |n|y)
	e^{2\pi inx}
\end{multline*}
where
$\delta_{\cusp, \cuspb} = \left\{\begin{array}{lll}
	1, & \text{ if } \cusp, \cuspb \text{ are }\Gamma\text{-equivalent; }\\
	0, & \text{ otherwise.}
\end{array}\right.$

In particular, when $\cusp=\frac{1}{a}$ for $a\mid N$ and $\cuspb=\infty$, we have
\be\label{e:const_eis}
	\rho_\cusp\left(s\right)
	:=
	\rho_{\cusp\infty}(s)
	=
	\frac{\zeta(2s-1)}{\zeta(2s)}
	\varphi(a)
	\left(\frac{1}{aN}\right)^{s}
	\prod_{p\mid N} \left(1-p^{-2s}\right)^{-1}
	\prod_{p\mid\frac{N}{a}}\left(1-p^{1-2s}\right)\,,
\ee
by \cite{DI83}, 
and $\rho_{\cusp\infty}(s, n) = \rho_\cusp\left(s, n\right) $, for $n\neq 0$, 
where $\rho_\cusp(s, n)$ is defined in (\ref{e:rho_an}). 
Here $\sigma_\cusp$ is fixed as in \eqref{e:sigma1/a}.

Let 
$$
	\Phi(s)
	:=
	\frac{\pi^{-s+\frac{1}{2}} \Gamma\left(s-\frac{1}{2}\right)}
	{ \pi^{-s} \Gamma(s) } 
	\left(\rho_{\cusp \cuspb}(s)\right)_{\cusp, \cuspb}
$$
be the scattering matrix. 
Then the Eisenstein series satisfies the following functional equation:
$$
	E_\cusp(z, 1-s)
	=
	\frac{\pi^{s-\frac{1}{2}} \Gamma\left(-s+\frac{1}{2}\right)}
	{\pi^{-1+s} \Gamma(1-s) }
	\sum_{\cuspb} \rho_{\cusp \cuspb}(1-s) E_\cuspb(z, s)
	.
$$	
By (11.12) in \cite{Iwa02}, we have
$$
	\Phi(s) 
	= 
	\frac{\zeta^*(2s-1)}{\zeta^*(2s)} 
	\bigotimes_{p\mid N} 
	\bpm \frac{p-1}{p^{2s}-1} & \frac{p^s(1-p^{1-2s})}{p^{2s}-1} 
	\\ \frac{p^s(1-p^{1-2s})}{p^{2s}-1} &\frac{ p-1}{p^{2s}-1}\ebpm
	=:
	\frac{\zeta^*(2s-1)}{\zeta^*(2s)} 
	\left(\rho_{\cusp\cuspb, {\rm finite}}(s)\right)_{\cusp, \cuspb}
	.
$$
Let
$$
	E_\cusp^*(z, s)
	:=
	\zeta^*(2s) E_\cusp(z, s)
	.
$$
Then
$$
	E_\cusp^*(z, 1-s)
	=
	\zeta^*(2-2s) E_\cusp(z, 1-s)
	=
	\sum_\cuspb 
	\rho_{\cusp\cuspb, {\rm finite}}(1-s)E_\cuspb^*(z, s)
	\\
	=:
	\zeta^*(2s) \tilde E_\cusp(z, s)
	.
$$

For each fixed $\cusp$, we have
$$
	\sum_{\cuspb} \rho_{\cusp\cuspb, {\rm finite}}(1-s)
	=
	\prod_{p\mid N} 
	\frac {p-1+p^{1-s} - p^{s}}
	{p^{2(1-s)}-1}
	=
	\prod_{p\mid N}
	\frac{p^s+1}
	{p^{1-s}+1}
	,
$$
and we set
$$ 
	\mathfrak{e}(N)
	:=
	\left.\left( \sum_{\cuspb} \rho_{\cusp\cuspb, {\rm finite}}(1-s)\right)
	\right|_{s=1}
	= \prod_{p\mid N} \frac{p+1}{2}
	.
$$
For each cusp $\cusp=1/a$ with $a\mid N$, by \cite{DI83}, we have
$$
	\Res_{s=1} \tilde E_\cusp(z, s)
	=
	\mathfrak{e}(N)
	\frac{3}{\pi} \prod_{p\mid N} \left(p+1\right)^{-1}
	=
	\frac{3}{\pi} \frac{1}{2^{r(N)}}
	.
$$
	
Recalling the description of $\zeta_{\cusp, Q}(s, \tau)$ given in (\ref{e:zeta_aQ}), 
we note that 
$\zeta_{\cusp, Q}(s, \tau)$ has a meromorphic continuation to $(s, \tau)\in \C^2$. 
Moreover $\zeta_{\cusp, Q}(s, \tau)$ has simple poles at $s=1\pm \tau$. 
We define $z_{\cusp, Q}^\pm(\tau)$ by the relation 
$$
	\Res_{s=1\pm\tau} \zeta_{\cusp, Q}(s, \tau)
	=
	\zeta(1\pm 2\tau) z_{\cusp, Q}^\pm (\tau)
	.
$$
The special value of the ratio of Dirichlet polynomials, $z_{\cusp, Q}^\pm(\mp 1/2)$, is computed in the following lemma. 
\begin{lem}\label{l:dirichletpoly}
For $Q\geq 1$, for $a\mid N$, 
$$
	z_{1/a, Q}^+\left(-\frac{1}{2}\right) 
	=
	\sqrt Q \prod_{p\mid N} \frac{1}{p+1}
$$
and 
$$
	z_{1/a, Q}^-\left(\frac{1}{2}\right)
	=
	\frac{1}{\sqrt Q} \prod_{p\mid a, p^\alpha\|Q, \atop \alpha\geq 0} (p^\alpha-1) 
	\prod_{p^\alpha\|Q, p\nmid N, \atop \alpha\geq 1} p^\alpha
	.
$$
Here $p$ is a prime.
Then we have
$$
	\frac{\mathfrak{e}(N) \sum_{a\mid N} z_{1/a, Q}^+\left(-\frac{1}{2}\right) 
	+ \sum_{a\mid N} z_{1/a, Q}^-\left(\frac{1}{2}\right) }
	{\sqrt Q}
	=
	2
	.
$$
\end{lem}
\begin{proof}
At $s=1+\tau$, we have
\begin{multline*}
	\Res_{s=1+\tau}\zeta_{\cusp, Q}(s, \tau)
	=
	\zeta(1+2\tau)z_{\cusp, Q}^+(\tau)
	\\
	=
	\zeta(1+2\tau) \left(\frac{N}{a}\right)^{\tau-\frac{1}{2}} 
	Q^{-\tau}
	\prod_{p\mid N} \left(1-p^{-1+2\tau}\right)^{-1}
	\prod_{p\mid\frac{N}{a}}\left(1-p^{-1}\right)
	\\
	\times
	\prod_{p\mid a, p^\alpha\|Q, \atop \alpha\geq 0} 
	\frac{p^{-1+2\tau}}{1-p^{2\tau}}
	\left( p^{-1} (p-1)^2 + p^{\alpha 2\tau}(p^{2\tau}-p) (1-p^{-(1+2\tau)})\right)
	\\
	\times
	\prod_{p^\alpha\| Q, p\nmid N, \atop \alpha\geq 1}
	(1-p^{2\tau})^{-1} \left( (1-p^{-1}) - p^{(\alpha+1)2\tau}(1-p^{-(1+2\tau)})\right)
	.
\end{multline*}
By simple computation, we get $z_{1/a, Q}^+\left(-1/2\right)$,  
and
$$
	\sum_{a\mid N} 
	z_{1/a, Q}^+\left(-\frac{1}{2}\right)
	=
	\sqrt{Q} \prod_{p\mid N}\frac{2}{p+1}
	.
$$
	
At $s=1-\tau$, we have
\begin{multline*}
	\Res_{s=1-\tau} \zeta_{\cusp, Q}(s, \tau)
	=
	\zeta(1-2\tau) z_{\cusp, Q}^-(\tau)
	\\
	=
	\zeta(1-2\tau) 
	\left(\frac{a}{N}\right)^{\frac{1}{2}-\tau}
	Q^{-\tau}
	\prod_{p\mid a, p^\alpha\|Q, \atop \alpha\geq 0}
	\frac{p^{-1+2\tau}}{1-p^{2\tau}}
	(p-1)
	\left( 1-p^{\alpha 2\tau}\right)
	\\
	\times
	\prod_{p^\alpha\| Q, p\nmid N, \atop \alpha\geq 1}
	(1-p^{2\tau})^{-1} \left( ( 1-p^{-(1-2\tau)}) - p^{(\alpha+1)2\tau} (1-p^{-1})\right)
	.
	\end{multline*}
Again, by simple computation, we get $z_{1/a, Q}^+\left(-\frac{1}{2}\right)$, 
and
$$
	\sum_{a\mid N}
	z_{1/a, Q}^-\left(\frac{1}{2}\right)
	=
	\sqrt Q
	.
$$
\end{proof}

\subsubsection{Summary of the contribution of $S_Q^{({\rm cont})}(X; \delta)$}

In the next few sections, we will compute the contribution of the continuous part of the spectrum $S_Q^{({\rm cont})}(X; \delta)$, given in \eqref{e:SQcontXdelta}. 
In this section, we collect all the information and summarize it in Proposition \ref{prop:cont}. 


\begin{prop}\label{prop:cont}
Let $\ell_1$ and $\ell_2$ be positive square-free integers, with $(N_0, \ell_1\ell_2)=1$.
Let $N=N_0 \ell_1\ell_2/ (\ell_1, \ell_2)$. Assume that $Q\geq 1$, with $(N_0, Q)=1$.  

When $f\neq g$, we have
\begin{multline}\label{e:prop:cont_fneqg}
	\lim_{\delta\to 0}
	(\ell_1\ell_2)^{\frac{k-1}{2}}
	\left(S_Q^{(\rm cont)}(X; \delta)- \frac{1}{\sqrt Q}S_{\rm cont, int}(1-k/2, 1/2; \delta)\right)
	\\
	=
	\frac{1}{4}
	\frac{(\ell_1, \ell_2)}{\ell_1\ell_2}
	L(1, f\otimes g)
	\sum_{a\mid N} 
	\left\{
	\left(\frac{1}{2^{r(N)}} + \frac{1}{Q} \prod_{p\mid a, p^\alpha\|Q , \atop \alpha\geq 0} (p^\alpha-1) \prod_{p^\alpha\|Q, p\nmid N, \atop \alpha\geq 1} p^\alpha\right)
	\right.
	\\
	\times
	\left(\frac{N_0}{(a, N_0)}\right) A\left(\frac{N_0}{(a, N_0)}\right) 
	\overline{B\left(\frac{N_0}{(a, N_0)}\right)}
	\\
	\left.
	\times
	\left(\prod_{p\mid \frac{\ell_1 (a, \ell_2)}{(\ell_1, \ell_2) (a, \ell_1)}}
	\frac{ A(p)-B(p)p^{-1}}{1-p^{-2}}\right)
	\left( \prod_{p\mid \frac{\ell_2 (a, \ell_1)}{(\ell_1, \ell_2) (a, \ell_2)}}
	\frac{ B(p)-A(p)p^{-1}}{1-p^{-2}} \right)
	\right\}
	\\
	+ \mathcal{O}\left((\ell_1\ell_2)^\epsilon Q^{\theta +\epsilon-\frac{1}{2}}\right) 
	+ \mathcal{O}\left(X^{-\epsilon}\right)
	.
\end{multline}

When $f=g$, there exists a constant $C_2(k)$, given in \eqref{C1def},
	independent of $f$ (but depending on $k$), such that 
\begin{multline}\label{e:prop:cont_f=g}
	\lim_{\delta\to 0}
	(\ell_1\ell_2)^{\frac{k-1}{2}}
	\left( S_Q^{(\rm cont)}(X; \delta) - \frac{1}{\sqrt Q} S_{\rm cont, int}(1-k/2, 1/2; \delta)\right)
	\\
	=
	- 
	\frac{1}{2}
	\Res_{s=1} L(s,f \otimes f; \ell_1, \ell_2) 
	\log X
	+
	\Res_{s=1} L(s, f\otimes f; \ell_1, \ell_2) 
	\log Q
	\\
	+
	\left\{ C_2(k) 
	- \frac{1}{4}\sum_{p\mid N} \log p \frac{1}{2(p-1)}
	+ \frac{1}{2}\frac{\Gamma'(k)}{\Gamma(k)} - \frac{1}{2} \log (4\pi)
	\right\} 
	\Res_{s=1} L(s, f \otimes f;\ell_1, \ell_2)
	\\
	+
	\left\{
	\frac{1}{2}
	\log (N, (\ell_1, \ell_2))
	+\frac{1}{4}\sum_{p\mid(\ell_1, \ell_2), p^\alpha\|Q, \atop \alpha\geq 0} 
	\log p \left(p^{-\alpha} -\frac{3}{2}\right)
	\right\}
	\Res_{s=1} L(s, f \otimes f;\ell_1, \ell_2)
	\\
	+
	\left\{ 
	- \sum_{p^\alpha\|Q, \alpha\geq 1} \log p \frac{1-p^{-\alpha}}{p-1}
	+
	\frac{1}{4}\sum_{p\mid N, p^\alpha\|Q, \atop \alpha\geq 0} 
	\log p \left(p^{-\alpha} \frac{3p-1}{p-1}-1\right)
	\right\}
	\\
	\times
	\Res_{s=1} L(s, f \otimes f;\ell_1, \ell_2)
	\\
	+ 
	\frac{1}{2}
	\left.\frac{d}{ds} \left( (s-1) L(s, f\otimes f; \ell_1, \ell_2)\right)\right|_{s=1}	
	+ 
	\mathcal{O}\left((\ell_1\ell_2)^\epsilon Q^{\theta +\epsilon-\frac{1}{2}}\right) 
	+ \mathcal{O}\left(X^{-\epsilon}\right)
	.
\end{multline}	
Here $p$ is a prime. 
Moreover, $S_{{\rm cont, int}}(1-k/2, 1/2; \delta)$ is the special value of the function defined in \eqref{ScontIntdeltadef}.
\end{prop}

\subsubsection{Computation of the main term}

In the next section, we will move the line of integration in $S_Q^{({\rm cont})}(X; \delta)$ and pick up a number of residues, 
	which contribute to the main term.

When $f\neq g$, this contribution consists of 
\begin{multline*}
	\frac{(4\pi)^k}{4\Gamma(k)}
	\left(
	\sum_{a\mid N} 
	\frac{ z_{1/a, Q}^+\left(-\frac{1}{2}\right) }{\sqrt Q} 
	\vol \overline{\left<V, \tilde E_{1/a}(*, s)\right>}\mid_{s=1}
	+
	\sum_{a\mid N} \frac{z_{1/a, Q}^- \left(\frac{1}{2}\right)} {\sqrt Q}
	\vol \overline{\left<V, E_{1/a}(*, s)\right>}\mid_{s=1}
	\right)
	\\
	+
	\text{(\ref{e:other2})}
	\,.
\end{multline*}
By Lemma \ref{l:dirichletpoly}, 
$z_{1/a, Q}^+\left(-1/2\right)=\sqrt Q\prod_{p\mid N}\frac{1}{p+1}$ 
is fixed for any $a\mid N$. 
So we have
$$
	\sum_{a\mid N} z_{1/a, Q}^+\left(-\frac{1}{2}\right)
	\overline{\langle V, \tilde E_{1/a} \left(*, s\right) \rangle }
	=
	\sqrt Q \left(\prod_{p\mid N, \text{ prime}}\frac{1}{p+1}\right) 
	\overline{\langle V, \sum_{a\mid N} \tilde E_{1/a} \left(*, s\right)\rangle}
	.
$$
By taking the summation over $a\mid N$, we get
\begin{multline*}
	\sum_{a\mid N} \tilde E_{1/a} \left(z, s\right)
	=
	\sum_\cusp \sum_\cuspb \rho_{\cusp\cuspb, {\rm finite}} (1-s)
	E_\cuspb (z, s)
	\\
	=
	\sum_\cuspb \left(\sum_\cusp \rho_{\cusp\cuspb, {\rm finite}} (1-s) \right)
	E_\cuspb(z, s)
	=
	\left(\prod_{p\mid N, \text{prime}} \frac{p^s+1}{p^{1-s}+1} \right)
	\sum_{a\mid N} E_{1/a} (z, s)
	.
\end{multline*}
So, 
\begin{multline*}
	\sum_{a\mid N} z_{1/a, Q}^+\left(-\frac{1}{2}\right)
	\overline{\langle V, \tilde E_{1/a} \left(*, s\right) \rangle }
	\\
	=
	\sqrt Q\left(\prod_{p\mid N, \text{ prime} }\frac{1}{p+1}\right)
	\left(\prod_{p\mid N, \text{ prime}} \frac{p^s+1}{p^{1-s}+1}\right)
	\sum_{a\mid N}
	\overline{\langle V, E_{1/a} \left(*, s\right)\rangle}
	.
\end{multline*}
Then by Lemma \ref{l:dirichletpoly}, 
\begin{multline}\label{e:tildeEtoE}
	\frac{1}{\sqrt Q} \sum_{a\mid N} z_{1/a, Q}^+\left(-\frac{1}{2}\right)
	\vol \overline{\left<V, \tilde E_{1/a}(*, s)\right>}\mid_{s=1}
	+ 
	\frac{1}{\sqrt Q}\sum_{a\mid N} z_{1/a, Q}^-\left(\frac{1}{2}\right)
	\vol \overline{\left<V,  E_{1/a}(*, s)\right>}\mid_{s=1}
	\\
	= 
	\sum_{a\mid N} 
	\left( \frac{1}{2^{r(N)}} 
	+ \frac{\left(\prod_{p^\alpha\|Q, p\nmid N, \atop \alpha\geq 1} p^\alpha\right)
	\left(\prod_{p\mid a, p^\alpha\|Q, \atop \alpha\geq 0} (p^\alpha-1)\right)}{Q}
	\right)
	\vol \overline{\left<V,  E_{1/a}(*, s)\right>}\mid_{s=1}
	.
\end{multline}
So we have
\begin{multline}\label{e:S_cont_fneqg}
	\lim_{\delta\to 0}
	(\ell_1\ell_2)^{\frac{k-1}{2}}
	\left( S_Q^{(\rm cont)}(X; \delta) - \frac{1}{\sqrt Q} S_{\rm cont, int}(1-k/2, 1/2; \delta)\right)
	\\
	=
	\frac{(4\pi)^k}{4\Gamma(k)}
	(\ell_1\ell_2)^{\frac{k-1}{2}}
	\sum_{a\mid N} 
	\left( \frac{1}{2^{r(N)}} 
	+ \frac{\left(\prod_{p^\alpha\|Q, p\nmid N, \atop \alpha\geq 1} p^\alpha\right)
	\left(\prod_{p\mid a, p^\alpha\|Q, \atop \alpha\geq 0} (p^\alpha-1)\right)}{Q}
	\right)
	\\
	\times
	\vol \overline{\left<V,  E_{1/a}(*, s)\right>}\mid_{s=1}
	+ \mathcal{O}\left((\ell_1\ell_2)^\epsilon Q^{\theta +\epsilon-\frac{1}{2}}\right) 
	+ \mathcal{O}\left(X^{-\epsilon}\right).
\end{multline}

We now consider the case when $f=g$. 
Define 
\begin{multline}\label{C1def}
	C_2(k)
	:=
	-
	\lim_{\delta\to 0}
	\left\{
	\frac{(4\pi)^k }
	{8\Gamma(k) }
	\frac{1}{2\pi i}
	\int_{\left(\frac{1}{2}+2\epsilon\right)}
	\tG (-s) \tG (s) \;ds
	\right.
	\\
	+
	\frac{(4\pi)^k
	}
	{
	2\sqrt\pi 
	\Gamma\left(\frac{k}{2}\right)
	}
	\frac{1}{2\pi i}
	\int_{\left(\frac{1}{2}-\frac{k}{2}-\epsilon\right)}
	\frac{2^{s-\frac{1}{2}}
	M\left(s, -\frac{1}{2i}, \delta\right)
	\Gamma\left(s+\frac{k}{2}-1\right)
	\Gamma\left(1-s\right)
	}
	{\Gamma(s+k-1) 
	}
	\; ds
	\\
	\left.
	+
	\left.\frac{d}{ds} \left( sK(s; \delta) \right) \right|_{s=0}	
	\right\}
	,
\end{multline}
where 
\begin{multline*}
	K(s; \delta)
	:=
	-2 K_1(s; \delta)  
	+
	K_2^-\left(s-\frac{1}{2}; \frac{k}{2}, \delta\right) 
	-
	K_2^-\left( s-\frac{1}{2} ; \frac{k}{2}-1, \delta\right)
	\\
	+
	K_2^+ \left(s-\frac{1}{2}; \frac{k}{2}, \delta\right)  
	- 
	K_2^+\left(s-\frac{1}{2}; 0, \delta\right) 
	- 
	K_3\left(s-\frac{1}{2}; \delta\right)
	\\ 
	- 
	2 K_4(s; \delta) 
	+ 
	K_5\left(s-\frac{k}{2}; \delta\right)
\end{multline*}
and 
$K_1$, $K_2^\pm$, $K_3$, $K_4$ and $K_5$ are defined in 
(\ref{e:K1-}), (\ref{e:K2-+}), (\ref{e:K2--}), (\ref{e:K_3^-}), (\ref{e:K4-}) 
and (\ref{e:K5-}), respectively. 

Let
$$
	L(s, f\otimes f; \ell_1, \ell_2)
	:=
	\sum_{m, n\geq 1, \atop m\ell_1 = n\ell_2} \frac{A(m)\overline{A(n)}}{(\ell_1m)^s}
	.
$$
For any cusp $1/a$ with $a\mid N$, we have $\Res_{s=1} E_{1/a} (*, s) = \Res_{s=1} E_\infty (*, s)$. 
So we have
\begin{multline*}
	\vol \Res_{s=1} \overline{\left<V_{\ell_1, \ell_2}, E_{1/a} (*, s)\right>}
	=
	\vol \Res_{s=1} \overline{\left<V_{\ell_1, \ell_2}, E_\infty (*, s)\right>}
	\\
	=
	(\ell_1\ell_2)^{-\frac{k-1}{2}} \frac{\Gamma(k)}{(4\pi)^k} 
	\Res_{s=1} L(s, f\otimes f; \ell_1, \ell_2)
	. 
\end{multline*}

We combine (\ref{e:f=g0}), (\ref{e:f=g1}), (\ref{e:f=g2}), (\ref{e:f=g3}), (\ref{e:f=g4}), (\ref{e:f=g5}), (\ref{e:f=g6}), (\ref{e:f=g7}), (\ref{e:f=g8}) and (\ref{e:f=g9}) together, 
obtaining, the contribution to the main term, as $\delta\to 0$, 
\begin{multline*}
	- 
	\frac{(\ell_1\ell_2)^{-\frac{k-1}{2}}}{4}
	\frac{ z(N, Q) }
	{\sqrt Q}
	\Res_{s=1} L(s,f \otimes f; \ell_1, \ell_2) 
	\log \left(\frac{X}{Q}\right)
	\\
	-
	\frac{(\ell_1\ell_2)^{-\frac{k-1}{2}}}{4}
	\frac{ z'(N, Q) }
	{\sqrt Q}
	\Res_{s=1} L(s, f \otimes f;\ell_1, \ell_2)
	\\
	+
	(\ell_1\ell_2)^{-\frac{k-1}{2}} C_2(k) \cdot \Res_{s=1} L(s, f \otimes f;\ell_1, \ell_2)
	\\
	+
	\frac{(4\pi)^k}{4\Gamma(k)}
	\frac{\vol }{\sqrt Q}
	\left\{ 
	\left( \sum_\cusp z_{\cusp, Q}^+ \left(-\frac{1}{2}\right) 
	\left.\frac{d}{ds}
	\left( (s-1) \langle V, \tilde E_\cusp(*, s)\rangle \right)
	\right|_{s=1} \right) 
	\right.
	\\
	\left.
	+
	\left( \sum_\cusp z_{\cusp, Q}^- \left(+\frac{1}{2}\right) 
	\left.\frac{d}{ds}
	\left( (s-1) \langle V, E_\cusp(*, s)\rangle \right)
	\right|_{s=1} \right) 
	\right\}
	\\
	+
	\text{ (\ref{e:other2}) }
\end{multline*}
where
$$
	z(N, Q) 
	= 
	\mathfrak{e}(N)\left(\sum_\cusp z_{\cusp, Q}^+\left(-\frac{1}{2}\right)
	\right)
	+ \left(\sum_\cusp z_{\cusp, Q}^-\left(\frac{1}{2}\right)\right)
$$
	and 
$$
	z'(N, Q) 
	= 
	\mathfrak{e}(N)
	\left(\sum_\cusp z_{\cusp, Q}^{+'}\left(-\frac{1}{2}\right)\right)
	-
	\left(\sum_\cusp z_{\cusp, Q}^{-'}\left(\frac{1}{2}\right)\right)
	.
$$
Similarly as we did in \eqref{e:tildeEtoE}, 
\begin{multline*}
	\vol \sum_{a\mid N} z_{1/a, Q}^+ \left(-\frac{1}{2}\right) 
	\left.\frac{d}{ds}
	\left( (s-1) \overline{\langle V, \tilde E_{1/a} (*, s)\rangle} \right)
	\right|_{s=1}
	\\
	=
	\sum_{a\mid N} \mathfrak e(N) z_{1/a, Q}^+\left(-\frac{1}{2}\right)
	\left.\frac{d}{ds} \left((s-1) \vol \overline{\left<V, E_{1/a}(*, s)\right>}\right)\right|_{s=1}
	\\
	+
	\sqrt Q
	(\ell_1\ell_2)^{-\frac{k-1}{2}} \frac{\Gamma(k)}{(4\pi)^k} 
	\Res_{s=1} L(s, f\otimes f; \ell_1, \ell_2) 
	\sum_{p\mid N} \frac{\log p (2p+1)}{2(p+1)}
	.
\end{multline*}
Then, by Lemma \ref{l:dirichletpoly}, 
\begin{multline}\label{e:S_cont_f=g}
	\lim_{\delta\to 0}
	(\ell_1\ell_2)^{\frac{k-1}{2}} 
	\left( S_Q^{(\rm cont)}(X; \delta) - S_{\rm cont, int}(1-k/2, 1/2; \delta)\right)
	\\
	=
	- 
	\frac{1}{2}
	\Res_{s=1} L(s,f \otimes f; \ell_1, \ell_2) 
	\log \left(\frac{X}{Q}\right)
	\\
	+
	\left\{ C_2(k) 
	+ \frac{1}{4}
	\left(\sum_{p\mid N} \frac{\log p (2p+1)}{2(p+1)}\right)\right\}
	\Res_{s=1} L(s, f \otimes f;\ell_1, \ell_2)
	\\
	-
	\frac{1}{4}
	\frac{ z'(N, Q) } {\sqrt Q}
	\Res_{s=1} L(s, f \otimes f;\ell_1, \ell_2)
	\\
	+
	\sum_{a\mid N} 
	\frac{\mathfrak e(N) z_{1/a, Q}^+\left(-\frac{1}{2}\right)+z_{1/a, Q}^-\left(\frac{1}{2}\right)}{\sqrt Q}
	\frac{(4\pi)^k}{4\Gamma(k)}
	(\ell_1\ell_2)^{\frac{k-1}{2}}
	\left.\frac{d}{ds} \left((s-1) \vol \overline{\left<V, E_{1/a}(*, s)\right>}\right)\right|_{s=1}
	\\	
	+ 
	\mathcal{O}\left((\ell_1\ell_2)^\epsilon Q^{\theta +\epsilon-\frac{1}{2}}\right) 
	+ \mathcal{O}\left(X^{-\epsilon}\right)
	.
\end{multline}	

In the following lemma, we get an explicit formula for $\vol \overline{\left<V_{\ell_1, \ell_2}, E_{1/a}(*, s)\right>}$. 
\begin{lem}\label{lem:innerprod_cusp}
Recall that $V_{\ell_1, \ell_2}(z):= \overline{f(\ell_1 z)} g(\ell_2 z) y^k$. 
For each $a\mid N=\frac{N_0 \ell_1\ell_2}{(\ell_1, \ell_2)}$, we have
\begin{multline}\label{e:inner_ell1ell2_1/a}
	\vol \overline{\left<V_{\ell_1, \ell_2}, E_{1/a}(*,  s)\right>}
	\\
	=
	\left(\frac{N_0}{(a, N_0)}\right)
	A\left(\frac{N_0}{(a, N_0)}\right) 
	\overline{B\left(\frac{N_0}{(a, N_0)}\right)}
	(\ell_1\ell_2)^{-s-\frac{k-1}{2}} (\ell_1, \ell_2)^{2s-1} (a, (\ell_1, \ell_2))^{1-s}
	\\
	\times
	\left(\prod_{p\mid \frac{\ell_1 (a, \ell_2)}{(\ell_1, \ell_2) (a, \ell_1)}}
	\frac{ A(p)-B(p)p^{-s}}{1-p^{-2s}}\right)
	\left( \prod_{p\mid \frac{\ell_2 (a, \ell_1)}{(\ell_1, \ell_2) (a, \ell_2)}}
	\frac{ B(p)-A(p)p^{-s}}{1-p^{-2s}} \right)
	\\
	\times
	\frac{\Gamma(s+k-1)}{(4\pi)^{s+k-1}}
	L(s, f\otimes g).
\end{multline}
Here $L(s, f\otimes g) = \sum_{m\geq 1} \frac{A(m)\overline{B(m)}}{m^s}$. 
\end{lem}
\begin{proof}
Let $\fund= \Gamma\bsl \bH$ be the fundamental domain for $\Gamma$. 
For $\cusp=1/a$ with $a\mid \frac{N_0\ell_1\ell_2}{(\ell_1, \ell_2)}$,  
compute the inner product
\begin{multline*}
	\overline{\left<V_{\ell_1, \ell_2} , E_\cusp(*, s)\right>}
	=
	\frac{1}{{\rm Vol}(\fund)}
	\int_{\fund} \overline{V_{\ell_1, \ell_2} (z) }
	\sum_{\gamma\in \Gamma_\cusp \bsl \Gamma} 
	\left(\Im(\sigma_\cusp^{-1}\gamma z)\right)^s
	\; \frac{dx \; dy}{y^2}
	\\
	=
	\frac{1}{{\rm Vol}(\fund)}
	\sum_{\gamma\in \Gamma_\cusp \bsl \Gamma}
	\int_{\sigma_\cusp^{-1}\gamma \fund}
	\overline{V_{\ell_1, \ell_2} (\sigma_\cusp z) } y^s \; \frac{dx\; dy}{y^2}
	.
\end{multline*}
As in \cite{Iwa02}, as $\gamma$ runs over $\Gamma_\cusp \bsl \Gamma$, 
the sets $\sigma_\cusp^{-1} \gamma \fund$ cover the strip $\left\{z\in\bH\;|\; 0< x< 1\right\}$ 
once (for an appropriate choice of representatives), giving
\be\label{e:inner_ell1ell2}
	\frac{1}{{\rm vol}(\fund)}
	\int_0^\infty \int_0^1 V_{\ell_1, \ell_2}(\sigma_\cusp z) y^s\; \frac{dx\; dy}{y^2}
	.
\ee
Here
$$
	V_{\ell_1, \ell_2}(\sigma_\cusp z)
	=
	f(\ell_1 \sigma_\cusp z) \overline{g(\ell_2 \sigma_\cusp z)}
	\Im(\sigma_\cusp z)^k
	.
$$

Now, we will follow Asai's method \cite{Asa76}. 
For a matrix $\gamma=\sm a & b\\ c& d\esm\in GL_2(\R)$ of positive determinant, write 
$$
	\left(f\mid_k \gamma\right)(z)
	=
	\det(\gamma)^{\frac{k}{2}} 
	(cz+d)^{-k}
	f(\gamma z)
	.
$$	
Then, for each $\cusp = 1/a$, we have
\be\label{e:Vell1ell2_exp}
	V_{\ell_1, \ell_2}(\sigma_{\cusp} z) 
	=
	(\ell_1\ell_2)^{-\frac{k}{2}} 
	\left(f\left|_k \bpm \ell_1 & \\ & 1\ebpm \sigma_\cusp \right.\right)(z)
	\overline{
	\left(g\left|_k \bpm \ell_1 & \\ & 1\ebpm \sigma_\cusp \right.\right)(z)}
	y^k
	.
\ee

Let 
$$
	M_\cusp := \frac{N_0}{(a, N_0)}
	.
$$
Then $M_\cusp \mid N_0$, $\left(\frac{a}{(a, \ell_i)}, M_\cusp\right)=1$, for $i=1, 2$, and 
$\frac{a}{(a, \ell_i)} M_\cusp \equiv 0 \pmod{N_0}$.
Since $\left(\frac{\ell_i}{(a, \ell_i)} M_\cusp, \frac{a}{(a, \ell_i)}\right)=1$, 
we can take $c_i, d_i\in \Z$, such that 
$$
	M_\cusp \frac{\ell_i}{(a, \ell_i)} d_i - \frac{a}{(a, \ell_i)} c_i =1
	.
$$
Define
$$
	W_{\cusp, \ell_i}
	:=
	\bpm \frac{\ell_i}{(a, \ell_i)} & c_i \\ \frac{a}{(a, \ell_i)} & d_i M_{\cusp}\ebpm 
	\bpm M_{\cusp} & \\ & 1\ebpm
	=
	\bpm M_{\cusp} \frac{\ell_i}{(a, \ell_i)}  & c_i \\ \frac{a}{(a, \ell_i)} M_{\cusp} & d_i M_{\cusp} \ebpm 
$$
and
this matrix $W_{\cusp, \ell_i}$ satisfies conditions described in (2) in \cite{Asa76}.
Since $f$ and $g$ are newforms of level $N_0$, $f\mid_k W_{\cusp, \ell_1}$ and $g\mid_k W_{\cusp, \ell_2}$ are again newforms of level $N_0$.

By Theorem 2 in \cite{Asa76}, we have
$$
	f\mid_k W_{\cusp, \ell_1}
	=
	\mu\left(M_{\cusp}\right)
	\sqrt{M_{\cusp}}
	A\left(M_{\cusp}\right)
	\cdot f
$$
and
$$
	g\mid_k W_{\cusp, \ell_2}
	=
	\mu\left(M_{\cusp}\right)
	\sqrt{M_{\cusp}}
	B\left(M_{\cusp}\right)
	\cdot g
	.
$$

Since
$$
	W_{\cusp, \ell_i}
	=
	\bpm \frac{\ell_i}{(a, \ell_i)} & c_i \\ \frac{a}{(a, \ell_i)}  & d_i M_{\cusp}\ebpm
	\bpm M_{\cusp} & \\ & 1\ebpm
	=
	\bpm \ell_i & \\ & 1\ebpm
	\sigma_\cusp 
	\bpm \frac{M_{\cusp}}{\sqrt{\m_\cusp}(\ell_i, a)} & \frac{c_i}{\ell_i\sqrt{\m_\cusp}}
	\\
	0 & \frac{\sqrt{\m_\cusp} (a, \ell_i)}{\ell_i} \ebpm
	,
$$
where $\sigma_{\cusp}$ is defined in \eqref{e:sigma1/a} and $\m_\cusp = \frac{N_0\ell_1\ell_2}{(\ell_1, \ell_2) a}$, 
we have
$$
	\bpm \ell_i & \\ & 1\ebpm \sigma_\cusp 
	=
	W_{\cusp, \ell_i} 
	\bpm \frac{\sqrt{\m_\cusp} (\ell_i, a)}{M_{\cusp}}  & -\frac{c_i}{\sqrt{\m_\cusp}M_{\cusp}}\\
	0 & \frac{\ell_i}{(\ell_i, a)\sqrt{\m_\cusp}}
	\ebpm
	.
$$
Thus,
$$
	\left( f\left|_k \bpm \ell_1 & \\ & 1\ebpm \sigma_\cusp \right.\right)(z)
	=
	\mu\left(M_{\cusp}\right)
	\sqrt{M_{\cusp}}
	A\left(M_{\cusp}\right) 
	\left(\frac{\m_\cusp (\ell_1, a)^2}{M_\cusp \ell_1}\right)^{\frac{k}{2}}
	\cdot
	f \left(\frac{\m_\cusp(\ell_1, a)^2}{\ell_1 M_{\cusp}} z -\frac{c_1 (\ell_1, a)}{\ell_1 M_{\cusp}}\right)
$$
and
$$
	\left(g\left|_k \bpm \ell_2 & \\ & 1\ebpm \sigma_\cusp \right.\right)(z)
	=
	\mu\left(M_{\cusp}\right)
	\sqrt{M_{\cusp}}
	B\left(M_{\cusp}\right)
	\left(\frac{\m_\cusp (\ell_2, a)^2}{M_\cusp \ell_2}\right)^{\frac{k}{2}}
	\cdot 
	g \left(\frac{\m_\cusp(\ell_2, a)^2}{\ell_2 M_{\cusp}} z -\frac{c_2 (\ell_2, a)}{\ell_2 M_{\cusp}}\right)
	.
$$
Substituting into \eqref{e:Vell1ell2_exp}, and then \eqref{e:inner_ell1ell2} gives us \eqref{e:inner_ell1ell2_1/a}.

\end{proof}

When $f\neq g$, by Lemma \ref{lem:innerprod_cusp}, we have
\begin{multline*}
	(\ell_1\ell_2)^{\frac{k-1}{2}} \vol \overline{\left<V_{\ell_1, \ell_2}, E_{1/a}(*, s)\right>}\mid_{s=1}
	\\
	=
	\left(\frac{N_0}{(a, N_0)}\right) A\left(\frac{N_0}{(a, N_0)}\right) 
	\overline{B\left(\frac{N_0}{(a, N_0)}\right)}
	\frac{(\ell_1, \ell_2)}{\ell_1\ell_2}
	\\
	\times
	\left(\prod_{p\mid \frac{\ell_1 (a, \ell_2)}{(\ell_1, \ell_2) (a, \ell_1)}}
	\frac{ A(p)-B(p)p^{-1}}{1-p^{-2}}\right)
	\left( \prod_{p\mid \frac{\ell_2 (a, \ell_1)}{(\ell_1, \ell_2) (a, \ell_2)}}
	\frac{ B(p)-A(p)p^{-1}}{1-p^{-2}} \right)
	\frac{\Gamma(k)}{(4\pi)^{k}}
	L(1, f\otimes g)
	.
\end{multline*}
Substituting into \eqref{e:S_cont_fneqg}, we get \eqref{e:prop:cont_fneqg}.


When $f=g$, by Lemma \ref{lem:innerprod_cusp}, we have
\begin{multline*}
	(\ell_1\ell_2)^{\frac{k-1}{2}} \vol \left.\frac{d}{ds}\left((s-1)\overline{\left<V_{\ell_1, \ell_2}, E_{1/a} (*, s)\right>} \right)\right|_{s=1}
	\\
	=
	\left\{
	-
	\log \left(\frac{(a, (\ell_1, \ell_2))}{(N, (\ell_1, \ell_2))} \right) 
	+
	\frac{\Gamma'(k)}{\Gamma(k)} - \log(4\pi)\right\}
	\frac{\Gamma(k)}{(4\pi)^k}
	\Res_{s=1} L(s, f\otimes f; \ell_1, \ell_2)
	\\
	+ 
	\frac{\Gamma(k)}{(4\pi)^k}
	\left.\frac{d}{ds} \left( (s-1) L(s, f\otimes f; \ell_1, \ell_2)\right)\right|_{s=1}
	.
\end{multline*}
Substituting into \eqref{e:S_cont_f=g}, 
\begin{multline}\label{e:S_cont_f=g2}
	\lim_{\delta\to 0}
	(\ell_1\ell_2)^{\frac{k-1}{2}} 
	\left( S_Q^{(\rm cont)}(X; \delta) - S_{\rm cont, int}(1-k/2, 1/2; \delta)\right)
	\\
	\\
	=
	- 
	\frac{1}{2}
	\Res_{s=1} L(s,f \otimes f; \ell_1, \ell_2) 
	\log \left(\frac{X}{Q}\right)
	\\
	+
	\left\{ C_2(k) 
	+ \frac{1}{4}
	\left(\sum_{p\mid N} \frac{\log p (2p+1)}{2(p+1)}\right)
	+
	\frac{1}{2}
	\log (N, (\ell_1, \ell_2))
	+
	\frac{1}{2}
	\frac{\Gamma'(k)}{\Gamma(k)} - 
	\frac{1}{2} \log(4\pi)
	\right\}
	\\
	\times
	\Res_{s=1} L(s, f \otimes f;\ell_1, \ell_2)
	\\
	-
	\frac{1}{4}
	\frac{ z'(N, Q) } {\sqrt Q}
	\Res_{s=1} L(s, f \otimes f;\ell_1, \ell_2)
	\\
	-
	\frac{1}{4}
	\left(\sum_{a\mid N} 
	\frac{\mathfrak e(N) z_{1/a, Q}^+\left(-\frac{1}{2}\right)+z_{1/a, Q}^-\left(\frac{1}{2}\right)}{\sqrt Q} \log (a, (\ell_1, \ell_2)) \right)
	\Res_{s=1} L(s, f\otimes f; \ell_1, \ell_2)
	\\
	+ 
	\frac{1}{2}
	\left.\frac{d}{ds} \left( (s-1) L(s, f\otimes f; \ell_1, \ell_2)\right)\right|_{s=1}
	\\	
	+ 
	\mathcal{O}\left((\ell_1\ell_2)^\epsilon Q^{\theta +\epsilon-\frac{1}{2}}\right) 
	+ \mathcal{O}\left(X^{-\epsilon}\right)
	.
\end{multline}
Applying Lemma \ref{l:dirichletpoly}, a simple computation gives us
\begin{multline*}
	\sum_{a\mid N} 
	\frac{ \mathfrak{e}(N) z_{1/a, Q}^+\left(-\frac{1}{2}\right) +z_{1/a, Q}^-\left(\frac{1}{2}\right) }{\sqrt Q} 
	\log (a, (\ell_1, \ell_2)) 
	\\
	=
	\log N - 2\log Q 
	+ 
	\sum_{p\mid N} 2\log p \frac{1}{p^2-1}
	+
	\sum_{p^\alpha\|Q, \alpha\geq 1} 
	4\log p \frac{1-p^{-\alpha}}{p-1}
	\\
	-
	\sum_{p\mid N, p^\alpha\|Q, \atop \alpha \geq 0} 
	\log p \left(p^{-\alpha} \frac{3p-1}{p-1} -1 \right) 
	-
	\sum_{p\mid (\ell_1, \ell_2), p^\alpha\|Q, \atop \alpha\geq 0}
	\log p \left(p^{-\alpha}-\frac{3}{2}\right)
	- 
	\frac{z'(N, Q)}{\sqrt Q}
	.
\end{multline*}
Substituting into \eqref{e:S_cont_f=g2}, we get \eqref{e:prop:cont_f=g}.
\subsubsection{Computation of the continuous part}

At the starting point, in $S_Q^{({\rm cont})}(X; \delta)$ given in \eqref{e:SQcontXdelta}, 
we have $\Re(s') = 2+1+\epsilon +k/2-1 = 2+k/2+\epsilon$
and $u=s'-s -k/2+1$.
Here $s' = s+u+k/2-1$. 
Change $u$ to $s'$ and interchange the order of integration, 
bringing the $\tau$ integral to the outside, and the $s'$ integral to the inside. 
Then
\begin{multline*}
	S_Q^{\rm (cont)} (X; \delta)
	=
	\sum_\cusp 
	\vol 
	\left(\frac{1}{2\pi i}\right)^4
	\int_{\left(\frac{1}{2}+2\epsilon\right)}
	\int_{(0)} 
	\int_{(2)}
	\int_{\left(2+\frac{k}{2}+\epsilon\right)}
	\frac{(4\pi)^k 2^{s-\frac{1}{2}}}
	{2\sqrt\pi \Gamma(s+k-1)}
	\\
	\times
	\frac{\zeta_{\cusp, Q}(s',\tau) M(s, \tau/i, \delta) 
	\overline{\left<V, E_\cusp^*\left(*, \frac{1}{2}+\tau\right)\right>}}
	{Q^{s'}\zeta^*(1+2\tau)\zeta^*(1-2\tau)}
	\\
	\times 
	\frac{\tG \left(s'-w-\frac{1}{2}\right) \tG (w) 
	\Gamma\left(w-s'+s+\frac{k}{2}-\frac{1}{2}\right) 
	\Gamma\left(s'-s+\frac{1}{2}\right)}
	{\Gamma\left(w+\frac{k}{2}\right)}
	\\
	\times
	X^{s'-\frac{1}{2}} 
	\; ds'\; ds \;d\tau \; dw \,.
\end{multline*}

We now move the $s'$ line of integration 
from $\Re(s') = 2+k/2+\epsilon$ to $\Re(s') = 3/2+\epsilon/2$. 
As the $s'$ line moves, no poles are passed over, so we have
\begin{multline}\label{e:original_cont}
	S_Q^{\rm (cont)} (X; \delta)
	=
	\frac{\vol }{\sqrt Q}
	\sum_\cusp \left(\frac{1}{2\pi i}\right)^4
	\int_{\left(\frac{1}{2}+2\epsilon\right)}
	\int_{(0)} 
	\int_{(2)}
	\int_{\left(\frac{3}{2}+\frac{\epsilon}{2}\right)}
	\frac{(4\pi)^k 2^{s-\frac{1}{2}}}
	{2\sqrt\pi \Gamma(s+k-1)}
	\\
	\times
	\frac{\zeta_{\cusp, Q}(s', \tau)
	M(s, \tau/i, \delta) 
	\overline{\left<V, E_\cusp^*\left(*, \frac{1}{2}+\tau\right)\right>}}
	{\zeta^*(1+2\tau)\zeta^*(1-2\tau)}
	\\
	\times 
	\frac{\tG \left(s'-w-\frac{1}{2}\right) \tG (w)
	\Gamma\left(w-s'+s+\frac{k}{2}-\frac{1}{2}\right) 
	\Gamma\left(s'-s+\frac{1}{2}\right) }
	{\Gamma\left(w+\frac{k}{2}\right)}
	\\
	\times
	\left(\frac{X}{Q}\right)^{s'-\frac{1}{2}} 
	\; ds'\; ds \;d\tau \; dw
	.
\end{multline}
	
We now move the $s$ line of integration 
from $\Re(s)=2$ to $\Re(s) = 1/2-k/2-\epsilon$. 
As the $s$ line moves, poles are passed over at poles of
\begin{enumerate}
	\item $M(s, \tau/i, \delta)$, 
	at $s=1/2 -r\pm \tau$ for $0\leq r\leq k/2$, 
	with a residue $c_r(\pm \tau, \delta)$
	
	\item $\Gamma\left(w-s'+s+k/2-1/2\right)$, 
	at $w-s'+s+k/2-1/2=0$, i.e., 
	$s = -w+s'-k/2+1/2$, 
	with a residue $1$
\end{enumerate}

\subsubsection{The contribution from $s=1/2-r\pm \tau$} 
	
The residues of the $M$ function at $s=1/2-r \pm \tau$ have the form 
\begin{multline}\label{e:define_c_r}
	{\rm Res}_{s=\frac{1}{2}-r\pm \tau} M(s, \tau/i, \delta)
	=
	\frac{\sqrt{\pi} 2^{r\mp \tau} 
	(-1)^r 
	\Gamma\left(\frac{1}{2}+r\mp \tau\right)
	\Gamma\left(-r\pm 2\tau\right)}
	{r! \Gamma\left(\frac{1}{2}+\tau\right)
	\Gamma\left(\frac{1}{2}-\tau\right)}
	\\
	+ 
	\cO_r \left((1+|\Re(\tau)|)^r e^{-\frac{\pi}{2}|\Re(\tau)|} \delta^{\frac{1}{4}-\epsilon}\right)
	=: 
	c_r(\pm \tau, \delta)
	.
\end{multline}
A simple computation verifies that 
$$
	\lim_{\delta\to 0} c_0\left(\pm \frac{1}{2}, \delta\right)
	=
	\pm \sqrt{\frac{\pi}{2}}
	.
$$
For $1\leq r\leq \frac{k}{2}$, we have
$$
	\lim_{\delta\to 0} c_r\left(\frac{1}{2}, \delta\right)
	=
	\frac{2^{r-\frac{1}{2}}\sqrt\pi}{2\cdot r!}
$$
and 
$$
	\lim_{\delta\to 0} c_r\left(-\frac{1}{2}, \delta\right)
	=
	-\frac{2^{r+\frac{1}{2}}\sqrt\pi }{2\cdot (r+1)!}
	.
$$

At $s=1/2-r\pm \tau$, for each $0\leq r\leq k/2$, we have
\begin{multline*}
	\frac{\vol }{\sqrt Q}
	\sum_\cusp 
	\left(\frac{1}{2\pi i}\right)^3 
	\int_{\left(\frac{1}{2}+2\epsilon\right)}
	\int_{(0)}
	\int_{\left(\frac{3}{2}+\frac{\epsilon}{2}\right)}
	\frac
	{(4\pi)^k 2^{-r\pm \tau} }
	{2\sqrt\pi 
	\Gamma\left(k-r-\frac{1}{2} \pm \tau\right) }
	\\
	\times
	\frac{ \zeta_{\cusp, Q}(s', \tau) c_r(\pm \tau, \delta) 
	\overline{\left<V, E^*_\cusp\left(*, \frac{1}{2}+\tau\right)\right>}}
	{\zeta^*(1-2\tau) \zeta^*(1+2\tau)}
	\tG \left(s'-w-\frac{1}{2}\right)
	\tG (w)
	\\
	\times
	\frac{
	\Gamma\left(w-s' \pm \tau-r+\frac{k}{2}\right)
	\Gamma\left(s'\mp \tau +r\right)}
	{\Gamma\left(w+\frac{k}{2}\right)}
	\left(\frac{X}{Q}\right)^{s'-\frac{1}{2}}
	\; ds'\; d\tau\; dw
	.
\end{multline*}
Move the $s'$ line of integration 
from $\Re(s') =3/2 + \epsilon/2$ to $\Re(s') = 1/2 -\epsilon$. 
As the $s'$ line moves, poles are passed over at poles of
\begin{itemize}
	\item $\zeta_{\cusp, Q}(s', \tau)$ 
	at $s'=1\pm \tau$
	with a residue 
	$\zeta(1\pm 2\tau) z_{\cusp, Q}^{\pm}(\tau)$
	
	\item $\tG \left(s'-w-1/2 \right)$ 
	at $s'-w-1/2=0$, i.e., 
	$s'=w+1/2$, 
	with a residue $1$
	
	\item $\Gamma\left(w-s'\pm \tau-r+k/2 \right)$
	at $w-s'\pm \tau-r+k/2=0$ 
	when $r=k/2$, i.e., 
	$s' = w\pm \tau$ 
	when $r=k/2$, 
	with a residue $-1$
\end{itemize}

\subsubsection{Contribution from the pole at $s'=1\pm \tau$}

We have the following contributions to the main term.
	
When $f\neq g$, 
\be\label{e:fn=g1}
	\frac{(4\pi)^k 2^{-\frac{k}{2}+\frac{1}{2}}
	c_{\frac{k}{2}}\left(\frac{1}{2}, \delta\right) \vol }
	{4\sqrt\pi \Gamma\left(\frac{k}{2}\right)}
	\sum_{\cusp}
	\frac{ z_{\cusp, Q}^+\left(-\frac{1}{2}\right) \overline{\left<V, \tilde E_{\cusp}(*, s)\right>}\mid_{s=1}
	+
	z_{\cusp, Q}^-\left(\frac{1}{2}\right) \overline{\left<V, E_{\cusp}(*, s)\right>}\mid_{s=1} }
	{\sqrt Q}
	.
\ee

Assume that $f=g$. 
Define
\begin{equation}\label{e:K1-}
	K_1(s; \delta)
	:=
	-
	\frac{(4\pi)^k 2^{-\frac{k}{2}+ \frac{-s+1}{2}} 
	c_{\frac{k}{2}}\left(\frac{-s+1}{2}, \delta\right)
	\tG \left(-\frac{s}{2}\right) 
	\tG (s)}
	{4\sqrt\pi \Gamma\left(\frac{k}{2}-\frac{s}{2}\right)
	\pi^{-\frac{s}{2}}
	\Gamma\left(\frac{s}{2}\right)}
	.
\end{equation}
The contributions to the main term in this case are 
\be\label{e:f=g0}
	-
	\frac{(4\pi)^k 2^{-\frac{1}{2}} 
	c_0\left(\frac{1}{2}, \delta\right)
	}
	{4\sqrt\pi \Gamma(k) }
	\frac{1}{2\pi i}
	\int_{\left(\frac{1}{2}+2\epsilon\right)}
	\tG (-w) \tG (w) \;dw
	\frac{ z(N, Q) \vol } {\sqrt Q}
	\Res_{s=1}\overline{\langle V, E_\infty (*, s)\rangle }
\ee
and 
\begin{multline}\label{e:f=g1}
	\frac{(4\pi)^k2^{\frac{-k+1}{2}}
	c_{\frac{k}{2}}\left(\frac{1}{2}, \delta\right)}
	{4\sqrt\pi \Gamma\left(\frac{k}{2}\right)}
	\frac{ z(N, Q) \vol }
	{\sqrt Q}
	\Res_{s=1}\overline{ \langle V, E_\infty (*, s)\rangle }
	\log \left(\frac{X}{Q}\right)
	\\
	+
	\frac{(4\pi)^k2^{\frac{-k+1}{2}}
	c_{\frac{k}{2}}\left(\frac{1}{2}, \delta\right)}
	{4\sqrt\pi \Gamma\left(\frac{k}{2}\right)}
	\frac{ z'(N,Q)\vol }
	{\sqrt Q}
	\Res_{s=1} \overline{\langle V, E_\infty(*, s)\rangle}
	\\
	-2
	\frac{ \left.\frac{d}{dw}
	\left( s K_{1}(s; \delta) \right) \right|_{s=0} z(N, Q) \vol}
	{\sqrt Q}
	\Res_{s=1}\overline{ \langle V, E_\infty (*, s)\rangle }
	\\
	-
	\frac{(4\pi)^k 2^{-\frac{k}{2}+\frac{1}{2}}
	c_{\frac{k}{2}}\left(\frac{1}{2},\delta\right)}
	{ 4\sqrt\pi \Gamma\left(\frac{k}{2}\right) }
	\frac{\vol}{\sqrt Q}
	\left\{ 
	\left( \sum_\cusp z_{\cusp, Q}^+ \left(-\frac{1}{2}\right) 
	\left.\frac{d}{ds}
	\left( (s-1) \overline{\langle V, \tilde E_\cusp(*, s)\rangle} \right)
	\right|_{s=1} \right) 
	\right.
	\\
	\left.
	+
	\left( \sum_\cusp z_{\cusp, Q}^- \left(+\frac{1}{2}\right) 
	\left.\frac{d}{ds}
	\left( (s-1) \overline{\langle V, E_\cusp(*, s)\rangle} \right)
	\right|_{s=1} \right) 
	\right\}
	.
\end{multline}

\subsubsection{Contribution from the pole at $s'=w+1/2$}

We have the following contributions to the main term.
	
When $f\neq g$, we have 
\begin{multline}\label{e:fn=g2}
	-\frac{k(4\pi)^k 2^{-\frac{k}{2}-\frac{1}{2}}
	c_{\frac{k}{2}}\left(-\frac{1}{2}, \delta\right)
	\vol }
	{4\sqrt\pi \Gamma\left(\frac{k}{2}-1\right)}
	\\
	\times
	\sum_\cusp
	\frac{
	\left(
	z_{\cusp, Q}^+\left(-\frac{1}{2}\right)
	\overline{\left<V, \tilde E_\cusp (*, s)\right>}\mid_{s=1}
	+
	z_{\cusp, Q}^-\left(\frac{1}{2}\right)
	\overline{\left<V, E_\cusp (*, s)\right>}\mid_{s=1}	\right)}
	{\sqrt Q}
	,
\end{multline}
\begin{multline}\label{e:fn=g3}
	\frac{(4\pi)^k 2^{-\frac{k}{2}+\frac{1}{2}}
	c_{\frac{k}{2}-1}\left(-\frac{1}{2}, \delta\right) \vol  
	}
	{2\sqrt\pi \Gamma\left(\frac{k}{2}\right)
	}
	\\
	\times
	\sum_\cusp
	\frac{
	\left(
	z_{\cusp, Q}^+\left(-\frac{1}{2}\right)
	\overline{\left<V, \tilde E_\cusp (*, s)\right>}\mid_{s=1}
	+
	z_{\cusp, Q}^-\left(\frac{1}{2}\right)
	\overline{\left<V,  E_\cusp (*, s)\right>}\mid_{s=1}
	\right)}
	{\sqrt Q}
	,
\end{multline}
\begin{multline}\label{e:fn=g4}
	-\frac{(4\pi)^k 2^{-\frac{k}{2}+\frac{1}{2}} 
	c_{\frac{k}{2}}\left(\frac{1}{2}, \delta\right) \vol
	}
	{
	2\sqrt\pi \Gamma\left(\frac{k}{2}\right)
	}
	\\
	\times
	\sum_{\cusp}
	\frac{
	\left(
	z_{\cusp, Q}^+\left(-\frac{1}{2}\right)
	\overline{\left<V, \tilde E_\cusp (*, s)\right>}\mid_{s=1}
	+
	z_{\cusp, Q}^-\left(\frac{1}{2}\right)
	\overline{\left<V, E_\cusp (*, s)\right>}\mid_{s=1}
	\right)
	}
	{\sqrt Q}
\end{multline}
and
\begin{multline}\label{e:fn=g5}
	\frac{(4\pi)^k 2^{\frac{1}{2}} 
	c_0\left(\frac{1}{2}, \delta\right) \vol
	}
	{
	4\sqrt\pi \Gamma\left(k\right)
	}
	\\
	\times
	\sum_\cusp
	\frac{
	\left(
	z_{\cusp, Q}^+\left(-\frac{1}{2}\right)
	\overline{\left<V, \tilde E_\cusp (*, s)\right>}\mid_{s=1}
	+
	z_{\cusp, Q}^-\left(\frac{1}{2}\right)
	\overline{\left<V, E_\cusp (*, s)\right>}\mid_{s=1}
	\right)}
	{\sqrt Q}
	.
\end{multline}

Assume that $f=g$.
For $r=0$ and $r=k/2$, define
\be\label{e:K2-+}
	K_2^+(s; r, \delta)
	:=
	(4\pi)^k  \pi^s 2^{-r-s-1} c_r(- s, \delta)
	\frac
	{\Gamma\left(\frac{k}{2}-r-\frac{1}{2}-s\right)
	\Gamma\left(2 s+r+1\right) 
	\tG \left(\frac{1}{2}+s \right) }
	{ \Gamma\left(k-r-\frac{1}{2} - s \right)
	\Gamma\left( s+ \frac{1}{2}+\frac{k}{2}\right)
	\Gamma\left(\frac{1}{2}+ s\right)}
\ee
and, for $r=k/2-1$ and $r=k/2$, 
\be\label{e:K2--}
	K_2^-(s; r, \delta)
	:=
	(4\pi)^k \pi^s 2^{-r+s -1}  c_r(s, \delta)
	\\
	\times
	\frac
	{\Gamma\left(s +\frac{k}{2}-r -\frac{1}{2}\right)
	\Gamma\left(r+1\right) \tG \left(\frac{1}{2}+s \right)}
	{ \Gamma\left(s+ k-r-\frac{1}{2}\right)
	\Gamma\left(s +\frac{k}{2} +\frac{1}{2} \right)
	\Gamma\left(\frac{1}{2}+ s\right)}
	.
\ee
The contributions are 
\begin{multline}\label{e:f=g2}
	\frac{k(4\pi)^k 2^{-\frac{k}{2}-\frac{1}{2}}
	c_{\frac{k}{2}}\left(-\frac{1}{2}, \delta\right)
	}
	{
	4\sqrt\pi \Gamma\left(\frac{k}{2}-1\right)
	}
	\frac{ z(N, Q) \vol }
	{\sqrt Q}
	\Res_{s=1}\overline{\langle V, E_\infty(*, s)\rangle }
	\log \left(\frac{X}{Q}\right)
	\\
	+
	\frac{k(4\pi)^k 2^{-\frac{k}{2}-\frac{1}{2}}
	c_{\frac{k}{2}}\left(-\frac{1}{2}, \delta\right)
	}
	{
	4\sqrt\pi \Gamma\left(\frac{k}{2}-1\right)
	}
	\frac{z'(N, Q)\vol }
	{\sqrt Q}
	\Res_{s=1}\overline{\langle V, E_\infty(*, s)\rangle }
	\\
	+
	\frac{
	\left.\frac{d}{ds}
	\left(
	\left(s+\frac{1}{2}\right)
	K_{2}^-\left(s; \frac{k}{2}, \delta\right)
	\right) \right|_{s=-\frac{1}{2}}
	z(N, Q)\vol }
	{\sqrt Q}
	\Res_{s=1}\overline{ \langle V, E_\infty \left(*, s \right)\rangle }
	\\
	-
	\frac{k(4\pi)^k 
	2^{-\frac{k}{2}-\frac{1}{2}}
	c_{\frac{k}{2}}\left(-\frac{1}{2}, \delta\right)
	}
	{
	4\sqrt\pi \Gamma\left(\frac{k}{2}-1\right)
	}
	\frac{\vol }{\sqrt Q}
	\left\{ 
	\left( \sum_\cusp z_{\cusp, Q}^+ \left(-\frac{1}{2}\right) 
	\left.\frac{d}{ds}
	\left( (s-1) \overline{\langle V, \tilde E_\cusp(*, s)\rangle} \right)
	\right|_{s=1} \right) 
	\right.
	\\
	\left.
	+
	\left( \sum_\cusp z_{\cusp, Q}^- \left(+\frac{1}{2}\right) 
	\left.\frac{d}{ds}
	\left( (s-1) \overline{\langle V, E_\cusp(*, s)\rangle} \right)
	\right|_{s=1} \right) 
	\right\}
	,
\end{multline}
\begin{multline}\label{e:f=g3}
	-
	\frac{(4\pi)^k 2^{-\frac{k}{2}+\frac{1}{2}} 
	c_{\frac{k}{2}-1}\left(-\frac{1}{2}, \delta\right)
	}
	{
	2\sqrt\pi \Gamma\left(\frac{k}{2}\right)
	}
	\frac{z(N, Q) \vol } {\sqrt Q} 
	\Res_{s=1} \overline{\langle V, E_\infty (*, s)\rangle }
	\log\left(\frac{X}{Q}\right)
	\\
	-
	\frac{(4\pi)^k 2^{-\frac{k}{2}+\frac{1}{2}} 
	c_{\frac{k}{2}-1}\left(-\frac{1}{2}, \delta\right) }
	{ 2\sqrt\pi \Gamma\left(\frac{k}{2}\right) }
	\frac{ z'(N, Q) \vol } {\sqrt Q}
	\Res_{s=1}\overline{ \langle V, E_\infty (*, s)\rangle }
	\\
	-
	\frac{
	\left.\frac{d}{ds} 
	\left( \left(s+\frac{1}{2}\right) K_{2}^-\left(s; \frac{k}{2}-1, \delta\right)
	\right)\right|_{s=-\frac{1}{2}}
	z(N, Q) \vol }
	{\sqrt Q}
	\Res_{s=1}\overline{ \langle V, E_\infty (*, s)\rangle }
	\\
	+
	\frac{(4\pi)^k 2^{-\frac{k}{2}+\frac{1}{2}} 
	c_{\frac{k}{2}-1}\left(-\frac{1}{2}, \delta\right) }
	{ 2\sqrt\pi \Gamma\left(\frac{k}{2}\right) }
	\frac{\vol}{\sqrt Q}
	\left\{ 
	\left( \sum_\cusp z_{\cusp, Q}^+ \left(-\frac{1}{2}\right) 
	\left.\frac{d}{ds}
	\left( (s-1) \overline{\langle V, \tilde E_\cusp(*, s)\rangle} \right)
	\right|_{s=1} \right) 
	\right.
	\\
	\left.
	+
	\left( \sum_\cusp z_{\cusp, Q}^- \left(+\frac{1}{2}\right) 
	\left.\frac{d}{ds}
	\left( (s-1) \overline{\langle V, E_\cusp(*, s)\rangle} \right)
	\right|_{s=1} \right) 
	\right\}
	,
\end{multline}
\begin{multline}\label{e:f=g4}
	\frac{(4\pi)^k 2^{-\frac{k}{2}+\frac{1}{2}}
	c_{\frac{k}{2}}\left(\frac{1}{2}, \delta\right) }
	{ 2\sqrt\pi \Gamma\left(\frac{k}{2}\right) }
	\frac{ z(N, Q) \vol}
	{\sqrt Q}
	\Res_{s=1}\overline{ \langle V, E_\infty (*, s)\rangle }
	\log\left(\frac{X}{Q}\right)
	\\
	+
	\frac{(4\pi)^k 2^{-\frac{k}{2}+\frac{1}{2}}
	c_{\frac{k}{2}}\left(\frac{1}{2}, \delta\right) }
	{ 2\sqrt\pi \Gamma\left(\frac{k}{2}\right) }
	\frac{z'(N, Q) \vol } {\sqrt Q}
	\Res_{s=1}\overline{\langle V, E_\infty (*, s)\rangle }
	\\
	+
	\frac{
	\left.\frac{d}{ds}
	\left(\left(s+\frac{1}{2}\right) K_{2}^+\left(s; \frac{k}{2}, \delta\right)
	\right)\right|_{s=-\frac{1}{2}}
	z(N, Q) \vol} {\sqrt Q}
	\Res_{s=1}\overline{\langle V, E_\infty (*, s)\rangle}
	\\
	-
	\frac{(4\pi)^k 2^{-\frac{k}{2}+\frac{1}{2}}
	c_{\frac{k}{2}}\left(\frac{1}{2}, \delta\right)
	}
	{
	2\sqrt\pi \Gamma\left(\frac{k}{2}\right)
	}
	\frac{\vol}{\sqrt Q}
	\left\{ 
	\left( \sum_\cusp z_{\cusp, Q}^+ \left(-\frac{1}{2}\right) 
	\left.\frac{d}{ds}
	\left( (s-1) \overline{\langle V, \tilde E_\cusp(*, s)\rangle} \right)
	\right|_{s=1} \right) 
	\right.
	\\
	\left.
	+
	\left( \sum_\cusp z_{\cusp, Q}^- \left(+\frac{1}{2}\right) 
	\left.\frac{d}{ds}
	\left( (s-1) \overline{\langle V, E_\cusp(*, s)\rangle} \right)
	\right|_{s=1} \right) 
	\right\}
\end{multline}
and
\begin{multline}\label{e:f=g5}
	-
	\frac{(4\pi)^k 2^{\frac{1}{2}} c_0\left(\frac{1}{2}, \delta\right) }
	{4\sqrt\pi \Gamma\left(k\right) }
	\frac{ z(N, Q) \vol} {\sqrt Q}
	\Res_{s=1}\overline{ \langle V, E_\infty (*, s)\rangle }
	\log \left(\frac{X}{Q}\right)
	\\
	-
	\frac{(4\pi)^k 2^{\frac{1}{2}} c_0\left(\frac{1}{2}, \delta\right) }
	{4\sqrt\pi \Gamma\left(k\right) }
	\frac{ z'(N, Q) \vol} {\sqrt Q}
	\Res_{s=1}\overline{\langle V, E_\infty (*, s)\rangle }
	\\
	-
	\frac{
	\left.\frac{d}{ds} \left(\left(s+\frac{1}{2}\right)
	K_{2}^+(s; 0, \delta) \right)\right|_{s=-\frac{1}{2}}
	z(N, Q) \vol } {\sqrt Q}
	\Res_{s=1}\overline{ \langle V, E_\infty (*, s)\rangle }
	\\
	+
	\frac{(4\pi)^k 2^{\frac{1}{2}} c_0\left(\frac{1}{2}, \delta\right) }
	{4\sqrt\pi \Gamma\left(k\right) }
	\frac{\vol}{\sqrt Q}
	\left\{ 
	\left( \sum_\cusp z_{\cusp, Q}^+ \left(-\frac{1}{2}\right) 
	\left.\frac{d}{ds}
	\left( (s-1) \overline{\langle V, \tilde E_\cusp(*, s)\rangle} \right)
	\right|_{s=1} \right) 
	\right.
	\\
	\left.
	+
	\left( \sum_\cusp z_{\cusp, Q}^- \left(+\frac{1}{2}\right) 
	\left.\frac{d}{ds}
	\left( (s-1) \overline{\langle V, E_\cusp(*, s)\rangle} \right)
	\right|_{s=1} \right) 
	\right\}
	.
\end{multline}

\subsubsection{Contribution of the pole at $s'=w\pm \tau$, when $r=k/2$}

There are no main term contributions.

\subsubsection{The contribution from $w-s'+s+k/2-1/2=0$}

For \eqref{e:original_cont}, at $s = -w+s'-k/2+1/2$, we have
\begin{multline*}
	\sum_\cusp 
	\left(\frac{1}{2\pi i}\right)^3 
	\int_{\left(\frac{1}{2}+2\epsilon\right)} \int_{(0)} 
	\int_{\left(\frac{3}{2}+\frac{\epsilon}{2}\right)}
	\frac{
	(4\pi)^k 
	2^{-w+s'-\frac{k}{2}}
	}
	{
	2\sqrt\pi 
	\Gamma\left(-w+s'+\frac{k}{2}-\frac{1}{2}\right)
	}
	\\
	\times
	\frac{
	\zeta_{\cusp, Q}(s', \tau)
	M\left(-w+s'-\frac{k}{2}+\frac{1}{2}, \tau/i, \delta\right) 
	\overline{\left<V, E_\cusp\left(*, \frac{1}{2}+\tau\right)\right>}
	}
	{
	\sqrt Q
	\zeta^*(1-2\tau)}
	\\
	\times 
	\tG \left(s'-w-\frac{1}{2}\right)
	\tG (w)
	\left(\frac{X}{Q}\right)^{s'-\frac{1}{2}} 
	\; ds' \; d\tau \; dw
	.
\end{multline*}

Move the $s'$ line of the integration 
from $\Re(s') = 3/2+\epsilon$ to $\Re(s') = 1/2 -\epsilon$.
As the $s'$ line moves, poles are passed over at poles of
\begin{itemize}
	\item $\tG \left(s'-w-1/2 \right)$ 
	at $s'=w+1/2$, 
	with a residue $1$
	
	\item $\zeta_{\cusp, Q}(s', \tau)$
	at $s'=1\pm \tau$, 
	with a residue 
	$\zeta(1\pm 2\tau)z_{\cusp, Q}^\pm\left(\tau\right)$
	
	\item $M\left(-w+s'-k/2+1/2, \tau/i, \delta\right)$ 
	at $s'=w\pm \tau$, 
	with a residue $c_{k/2}(\pm \tau, \delta)$
\end{itemize}

\subsubsection{Contribution from the poles at $s'=w+1/2$}

At $s'=w+1/2$, we have
\begin{multline*}
	\sum_\cusp 
	\left(\frac{1}{2\pi i}\right)^2 
	\int_{\left(\frac{1}{2}+2\epsilon\right)} 
	\int_{(0)} 
	\frac{
	(4\pi)^k 
	2^{\frac{1}{2}-\frac{k}{2}}
	}
	{
	2\sqrt\pi 
	\Gamma\left(\frac{k}{2}\right)
	}
	\frac{
	\zeta_{\cusp, Q}\left(w+\frac{1}{2}, \tau\right)
	M\left(1-\frac{k}{2}, \tau/i, \delta\right) 
	\overline{\left<V, E^*_\cusp\left(*, \frac{1}{2}+\tau\right)\right>}
	}
	{\sqrt Q \zeta^*(1-2\tau) \zeta^*(1+2\tau) }
	\\
	\times
	\tG (w)
	\left(\frac{X}{Q}\right)^{w} 
	\; d\tau \; dw
\end{multline*}
	where
$$
	M\left(1-\frac{k}{2}, \tau/i, \delta\right)
	=
	\frac{\sqrt\pi 2^{\frac{k-1}{2}} 
	\Gamma\left(\frac{1-k}{2}-\tau\right) \Gamma\left(\frac{1-k}{2}+\tau\right)
	\Gamma\left(\frac{k}{2}\right) }
	{\Gamma\left(\frac{1}{2}-\tau\right) 
	\Gamma\left(\frac{1}{2}+\tau\right)}
	+
	\cO_{A, \epsilon} 
	\left(\left(1+|t|\right)^{-k+2\epsilon} \delta^\epsilon\right)
	.
$$
We move $w$ to $\Re(w) = -\epsilon$. 
As the $w$ line moves, poles are passed over at poles of
\begin{itemize}
	\item $\zeta_{\cusp, Q}\left(w+1/2, \tau\right)$ 
	at $w=1/2\pm \tau$, with a residue
	$\zeta(1\pm 2\tau) z_{\cusp, Q}^\pm (\tau)$
	
	\item $\tG (w)$ at $w=0$, with a residue $1$
\end{itemize}

When $f\neq g$, the pole at $w=1/2\pm \tau$  contributes 
\be\label{e:fn=g6}
	\frac{(4\pi)^k 2^{\frac{1}{2}-\frac{k}{2}}
	c^{(1)}\left(\frac{k}{2}, \delta\right)
	\vol }
	{2\sqrt\pi \Gamma\left(\frac{k}{2}\right)}
	\sum_\cusp
	\frac{ z_{\cusp, Q}^+\left(-\frac{1}{2}\right) 
	\overline{\left<V, \tilde E_\cusp (*, s)\right>}\mid_{s=1} 
	+
	z_{\cusp, Q}^-\left(\frac{1}{2}\right)
	\overline{\left<V, E_\cusp (*, s)\right>}\mid_{s=1}}
	{\sqrt Q}
	.
\ee
	
	
Assume that $f=g$. 
Define	
\be\label{e:K_3^-}
	K_{3}(s;\delta)
	:=
	\frac{ (4\pi)^k 2^{-\frac{1}{2}-\frac{k}{2} \pi^s}
	M\left(1-\frac{k}{2}, s/i, \delta\right) 
	\tG \left(\frac{1}{2}+s\right) }
	{ \Gamma\left(\frac{k}{2}\right)
	\Gamma\left(\frac{1}{2}+s\right)}
	.
\ee
In this case, the contribution from the pole at $w=1/2\pm \tau$ is 
\begin{multline}\label{e:f=g6}
	-
	\frac{(4\pi)^k 2^{\frac{1}{2}-\frac{k}{2}}
	c^{(1)}\left(\frac{k}{2}, \delta\right) }
	{ 2\sqrt\pi \Gamma\left(\frac{k}{2}\right) }
	\frac{ z(N, Q)\vol } {\sqrt Q}
	\Res_{s=1}\overline{ \langle V, E_\infty (*, s)\rangle }
	\log \left(\frac{X}{Q}\right)
	\\
	\\
	-
	\frac{(4\pi)^k 2^{\frac{1}{2}-\frac{k}{2}}
	c^{(1)}\left(\frac{k}{2}, \delta\right) }
	{ 2\sqrt\pi \Gamma\left(\frac{k}{2}\right) }
	\frac{ z'(N, Q) \vol } {\sqrt Q}
	\Res_{s=1}\overline{ \langle V, E_\infty (*, s)\rangle }
	\\
	-
	\frac{
	\left.\frac{d}{ds}\left( \left(s+\frac{1}{2}\right)
	K_{3}(s; \delta) \right) \right|_{s=-\frac{1}{2}}
	z(N, Q) \vol } {\sqrt Q}
	\Res_{s=1}\overline{ V, E_\infty (*, s)\rangle }
	\\
	+
	\frac{(4\pi)^k 2^{\frac{1}{2}-\frac{k}{2}}
	c^{(1)}\left(\frac{k}{2}, \delta\right)
	}
	{
	2\sqrt\pi \Gamma\left(\frac{k}{2}\right)
	}
	\frac{\vol}{\sqrt Q}
	\left\{ 
	\left( \sum_\cusp z_{\cusp, Q}^+ \left(-\frac{1}{2}\right) 
	\left.\frac{d}{ds}
	\left( (s-1) \overline{\langle V, \tilde E_\cusp(*, s)\rangle} \right)
	\right|_{s=1} \right) 
	\right.
	\\
	\left.
	+
	\left( \sum_\cusp z_{\cusp, Q}^- \left(+\frac{1}{2}\right) 
	\left.\frac{d}{ds}
	\left( (s-1) \overline{\langle V, E_\cusp(*, s)\rangle} \right)
	\right|_{s=1} \right) 
	\right\}
	.
\end{multline}

At $w=0$, we have
\be\label{e:other2}
	\sum_\cusp 
	\frac{1}{2\pi i} 
	\int_{(0)} 
	\frac{(4\pi)^k 2^{\frac{1}{2}-\frac{k}{2}} 
	\zeta_{\cusp, Q}\left(\frac{1}{2}, \tau\right)
	M\left(1-\frac{k}{2}, \tau/i, \delta\right)
	\overline{\left<V, E_\cusp\left(*, \frac{1}{2}+\tau\right)\right>}
	}
	{2\sqrt\pi \sqrt Q
	\Gamma\left(\frac{k}{2}\right)
	\zeta^*(1-2\tau)}
	\; d\tau
	\,.
\ee
Note that the integral over $\tau$ converges absolutely.

\subsubsection{Contribution from the poles at $s'=1\pm \tau$}

We have the following contributions to the main term. 

When $f\neq g$: 
\be\label{e:fn=g7}
	-
	\frac{(4\pi)^k 2^{\frac{1}{2}-\frac{k}{2}}
	c_{\frac{k}{2}}\left(\frac{1}{2};\delta\right)
	\vol }
	{ 4\sqrt\pi \Gamma\left(\frac{k}{2}\right) }
	\sum_\cusp
	\frac{ z_{\cusp, Q}^+\left(-\frac{1}{2}\right)
	\overline{\left<V, \tilde E_\cusp(*, s)\right>}\mid_{s=1}
	+
	z_{\cusp, Q}^-\left(\frac{1}{2}\right)
	\overline{\left<V, E_\cusp(*, s)\right>}\mid_{s=1}}
	{\sqrt Q}
	.
\ee

Assume that $f=g$.
Define
\be\label{e:K4-}
	K_{4}(s; \delta)
	:=
	\frac{
	(4\pi)^k 
	2^{\frac{-s-3}{2}-\frac{k}{2}}
	\pi^{\frac{s-1}{2}}
	c_{\frac{k}{2}}\left(\frac{-s+1}{2};\delta\right)
	\tG \left(-\frac{s}{2}\right)
	\tG (s)
	}
	{
	\Gamma\left(-\frac{s}{2}+\frac{k}{2}\right)
	\Gamma\left(\frac{s}{2}\right)
	} 
	.
\ee
Then we have
\begin{multline}\label{e:f=g7}
	-
	\frac{(4\pi)^k 2^{\frac{1}{2}-\frac{k}{2}}
	c_{\frac{k}{2}}\left(\frac{1}{2}; \delta\right)
	}
	{
	4\sqrt\pi\Gamma\left(\frac{k}{2}\right)
	}
	\frac{ z(N, Q)\vol } {\sqrt Q}
	\Res_{s=1}\overline{ \langle V, E_\infty (*, s)\rangle }
	\log\left(\frac{X}{Q}\right)
	\\
	-2
	\frac{(4\pi)^k 2^{\frac{1}{2}-\frac{k}{2}}
	c_{\frac{k}{2}}\left(\frac{1}{2}; \delta\right) }
	{ 4\sqrt\pi\Gamma\left(\frac{k}{2}\right) }
	\frac{ z'(N, Q)\vol } {\sqrt Q}
	\Res_{s=1}\overline{ \langle V, E_\infty (*, s)\rangle }
	\\
	-
	2
	\frac{ \left.\frac{d}{ds}\left(s K_4 (s; \delta) \right)\right|_{s=0}
	z(N, Q) \vol} {\sqrt Q}
	\Res_{s=1}\overline{ \langle V, E_\infty (*, s)\rangle }
	\\
	+
	\frac{(4\pi)^k 2^{\frac{1}{2}-\frac{k}{2}}
	c_{\frac{k}{2}}\left(\frac{1}{2}; \delta\right)
	}
	{
	4\sqrt\pi\Gamma\left(\frac{k}{2}\right)
	}
	\frac{\vol}{\sqrt Q}
	\left\{ 
	\left( \sum_\cusp z_{\cusp, Q}^+ \left(-\frac{1}{2}\right) 
	\left.\frac{d}{ds}
	\left( (s-1) \overline{\langle V, \tilde E_\cusp(*, s)\rangle} \right)
	\right|_{s=1} \right) 
	\right.
	\\
	\left.
	+
	\left( \sum_\cusp z_{\cusp, Q}^- \left(+\frac{1}{2}\right) 
	\left.\frac{d}{ds}
	\left( (s-1) \overline{\langle V, E_\cusp(*, s)\rangle} \right)
	\right|_{s=1} \right) 
	\right\}
	.
\end{multline}	
	
\subsubsection{Contribution from the poles at $s'=w\pm \tau$}

There are no main term contributions.

\subsubsection{The contribution from the shifted integral $\Re(s) = 1/2-k/2-\epsilon$}

After moving $\Re(s)$ to $1/2-k/2-\epsilon$, we have
\begin{multline}\label{e:shifted_cont}
	\frac{\vol }{\sqrt Q}
	\sum_\cusp \left(\frac{1}{2\pi i}\right)^4
	\int_{\left(\frac{1}{2}+2\epsilon\right)}
	\int_{(0)} 
	\int_{\left(\frac{1}{2}-\frac{k}{2}-\epsilon\right)}
	\int_{\left(\frac{3}{2}+\frac{\epsilon}{2}\right)}
	\frac{(4\pi)^k 2^{s-\frac{1}{2}} }
	{2\sqrt\pi \Gamma(s+k-1)}
	\\
	\times
	\frac{ \zeta_{\cusp, Q}(s', \tau) 
	M(s, \tau/i, \delta) 
	\overline{\left<V, E_\cusp^*\left(*, \frac{1}{2}+\tau\right)\right>}
	}
	{\zeta^*(1+2\tau)\zeta^*(1-2\tau)}
	\\
	\times 
	\frac{\Gamma\left(w-s'+s+\frac{k}{2}-\frac{1}{2}\right) 
	\Gamma\left(s'-s+\frac{1}{2}\right)
	\tG \left(s'-w-\frac{1}{2}\right) \tG(w)}
	{\Gamma\left(w+\frac{k}{2}\right)}
	\left(\frac{X}{Q}\right)^{s'-\frac{1}{2}} 
	\; ds'\; ds \;d\tau \; dw
	.
\end{multline}
We move the $s'$ line of integration 
from $3/2+\epsilon/2$ to $1/2-\epsilon$. 
As the $s'$ line moves, poles are passed over at poles of
\begin{itemize}
	\item $\zeta_{\cusp, Q}(s', \tau)$ 
	at $s'=1\pm \tau$, with a residue 
	$\zeta(1\pm 2\tau) z_{\cusp, Q}^\pm (\tau)$
	
	\item $\tG \left(s'-w-1/2 \right)$ 
	at $s' = w+ 1/2 $, with a residue $1$
		
	\item $\Gamma\left(w-s'+s+k/2- 1/2 \right)$
	at $w-s'+s+k/2- 1/2 =0$, i.e., 
	$s' = w+s+k/2 - 1/2$, 
	with a residue $-1$
\end{itemize}

\subsubsection{Contribution from the poles at $s'=1\pm \tau$}

There are no main term contributions.
	
\subsubsection{Contribution from the pole of $\tG \left(s'-w-1/2\right)$ at $s'=w+1/2$}

Let $c^{(3)}(s, \delta)$ be a residue of $M(s, \tau/i, \delta)$ at $\tau = s+k/2- 1/2$, as \eqref{ires2}. 
Then 
$$
	\lim_{\delta\to 0} c^{(3)}(s, \delta)
	= 
	\frac{(-1)^{\frac{k}{2}+1} \sqrt \pi 2^{\frac{1}{2}-s} 
	\Gamma(1-s) \Gamma\left(2s+\frac{k}{2}-1\right)}
	{\Gamma\left(1-s-\frac{k}{2}\right) \Gamma\left(s+\frac{k}{2}\right)
	\left(\frac{k}{2}\right)!}
	.
$$
	
We have the following contributions to the main term. 

When $f\neq g$: 
\begin{multline}\label{e:fn=g8}
	-
	\frac{k(4\pi)^k 2^{-\frac{k}{2}-\frac{1}{2}}
	c^{(3)}\left(-\frac{k}{2};\delta\right) \vol}
	{ 4\sqrt\pi
	\Gamma\left(\frac{k}{2}-1\right) }
	\\
	\times
	\sum_\cusp
	\frac{ z_{\cusp, Q}^+\left(-\frac{1}{2}\right) 
	\overline{ \left< V, \tilde E_\cusp (*, s)\right> } \mid_{s=1}
	+
	z_{\cusp, Q}^-\left(\frac{1}{2}\right)
	\overline{ \left< V, E_\cusp (*, s)\right> } \mid_{s=1}}	
	{\sqrt Q}
	.
\end{multline}

Assume that $f=g$. 
Define
\be\label{e:K5-}
	K_{5}(s;\delta)
	:=
	\frac{(4\pi)^k 2^{s-\frac{3}{2}} \pi^{-s-\frac{k}{2}+\frac{1}{2}}
	c^{(3)}\left(s;\delta\right)	
	\Gamma\left(s+\frac{k}{2}-1\right) 
	\Gamma\left(\frac{k}{2}+1\right)
	\tG \left(s+\frac{k}{2}\right) }
	{\Gamma(s+k-1)
	\Gamma\left(s+\frac{k}{2}\right)
	\Gamma\left(s+k\right)}
	.
\ee
Then we have
\begin{multline}\label{e:f=g8}
	\frac{k(4\pi)^k 2^{-\frac{k}{2}-\frac{1}{2}}
	c^{(3)}\left(-\frac{k}{2};\delta\right)
	}
	{
	4\sqrt\pi \Gamma\left(\frac{k}{2}-1\right)
	}
	\frac{ z(N, Q) \vol }
	{\sqrt Q}
	\Res_{s=1}\overline{ \langle V, E_\infty (*, s)\rangle }
	\log 
	\left(\frac{X}{Q}\right)
	\\
	+
	\frac{k(4\pi)^k 2^{-\frac{k}{2}-\frac{1}{2}}
	c^{(3)}\left(-\frac{k}{2};\delta\right)
	}
	{
	4\sqrt\pi \Gamma\left(\frac{k}{2}-1\right)
	}
	\frac{ z'(N, Q) \vol } {\sqrt Q}
	\Res_{s=1}\overline{ \langle V, E_\infty (*, s)\rangle }
	\\
	+
	\frac{ \left.\frac{d}{ds}
	\left( \left(s+\frac{k}{2}\right) K_{5}(s;\delta) \right) \right|_{s=-\frac{k}{2}}
	z(N, Q) \vol }
	{\sqrt Q}
	\Res_{s=1}\overline{ \langle V, E_\infty (*, s)\rangle }
	\\
	-
	\frac{k(4\pi)^k 2^{-\frac{k}{2}-\frac{1}{2}}
	c^{(3)}\left(-\frac{k}{2};\delta\right)
	}
	{
	4\sqrt\pi \Gamma\left(\frac{k}{2}-1\right)
	}
	\frac{\vol}{\sqrt Q}
	\left\{ 
	\left( \sum_\cusp z_{\cusp, Q}^+ \left(-\frac{1}{2}\right) 
	\left.\frac{d}{ds}
	\left( (s-1) \overline{\langle V, \tilde E_\cusp(*, s)\rangle} \right)
	\right|_{s=1} \right) 
	\right.
	\\
	\left.
	+
	\left( \sum_\cusp z_{\cusp, Q}^- \left(+\frac{1}{2}\right) 
	\left.\frac{d}{ds}
	\left( (s-1) \overline{\langle V, E_\cusp(*, s)\rangle} \right)
	\right|_{s=1} \right) 
	\right\}
\end{multline}
and
\begin{multline}\label{e:f=g9}
	-
	\frac{(4\pi)^k
	}
	{
	2\sqrt\pi 
	\Gamma\left(\frac{k}{2}\right)
	}
	\frac{1}{2\pi i}
	\int_{\left(\frac{1}{2}-\frac{k}{2}-\epsilon\right)}
	\frac{2^{s-\frac{1}{2}}
	M\left(s, -\frac{1}{2i}, \delta\right)
	\Gamma\left(s+\frac{k}{2}-1\right)
	\Gamma\left(1-s\right)
	}
	{\Gamma(s+k-1) 
	}
	\; ds
	\\
	\times
	\frac{ z(N, Q)\vol } {\sqrt Q}
	\Res_{s=1} \overline{ \langle V, E_\infty(*, s)\rangle }
	.
\end{multline}

\subsubsection{Contribution from the pole of $\Gamma\left(w-s'+s+k/2-1/2\right)$}

There are no main term contributions.
	

\section{Triple Dirichlet Series}\label{s:triple}

Let $N_0$ and $Q\geq 1$ and assume that $N_0$ is square-free 
and $(Q, N_0)=1$. 
Let $f$ and $g$ be holomorphic cusp forms of even weight $k$, 
for $\Gamma_0(N_0)$, with normalized Fourier coefficients $A(m)$, $B(m)$, as in \eqref{def1}, so that $A(1) = B(1)=1$.
Assume that $f$ and $g$ are eigenfunctions for Hecke operators $T_p$ 
with primes $p\nmid N_0$.

Recall that 
	$$D(s; h) 
	= 
	\sum_{m_1\ell_1 = m_2\ell_2 +h} 
	\frac{a(m_1) \overline{b(m_2)}} {(m_2\ell_2)^{s+k-1}}
	$$
	is absolutely convergent when $\Re(s)>1$. 
Here $a(m) = A(m)m^{(k-1)/2}$. 
Fix positive integers $\ell_1,\ell_2$, with $(\ell_1\ell_2, N_0)=1$. 
For $q \geq 1$, for $s, w$ with real parts greater than $1$, the sum
	$$
	Z_q(s, w)
	=
	Z_q(s, w; \ell_1,\ell_2)
	=
	(\ell_1\ell_2)^{\frac{k-1}{2}} 
	\sum_{h\geq 1} \frac{D(s; h q)}{(h q)^{w+\frac{k-1}{2}}}
	$$
	is absolutely convergent, and can be written in the form 
	$$
	Z_q(s, w)
	=
	(\ell_1\ell_2)^{\frac{k-1}{2}} 
	\sum_{h\geq 1, \atop \ell_1 m_1 = \ell_2m_2 +hq}
	\frac{a(m_1) \overline{b(m_2)}}{(m_2\ell_2)^{s+k-1} (hq)^{w+\frac{k-1}{2}}}
	\,.
	$$

We define
\begin{equation}\label{e:triple}
	\cM(s, w, v)
	=
	\cM(s, w, v; \ell_1, \ell_2)
	:= 
	(\ell_1\ell_2)^{\frac{k-1}{2}} 
	\sum_{q\geq 1, \atop (q, N_0)=1} \frac{Z_{q}(s, w)}{q^v}
	\,.
\end{equation}
By the polynomial bounds in $q$ for $Z_q(s, w)$, this series will converges for any fixed value of $s, w$ when the real part of $v$ sufficiently large. 

For $n\neq 0$, define
\be\label{e:adiv}
	\tilde\sigma_{-v}(n)
	:=
	\sum_{d\mid n, (d, N_0)=1} d^{-v}
	.
\ee
In addition, for each integer $m_2\geq 1$ and $\Re(w)$ and $\Re(v)>1$, let
\begin{equation}\label{e:D'}
	D'(w, v; m_2\ell_2)
	:=
	\sum_{n\geq 1, \atop \ell_1m_1 = n+\ell_2m_2} 
	\frac{a(m_1) \tilde\sigma_{-v}(n)}{n^{w+\frac{k-1}{2}}}
	.
\end{equation}

For $\Re(s), \Re(w)$ and $\Re(v)>1$, this gives us
\begin{multline}\label{e:Mswv}
	(\ell_1\ell_2)^{-\frac{k-1}{2}} \cM(s, w, v)
	=
	\sum_{q, h\geq 1, (q, N_0)=1, \atop \ell_1m_1 = \ell_2m_2+hq}
	\frac{a(m_1)\overline{b(m_2)}}{(m_2\ell_2)^{s+k-1} (h q)^{w+\frac{k-1}{2}} q^v}
	\\
	=
	\sum_{n, m_2\geq 1, \atop \ell_1m_1= \ell_2m_2+n}
	\frac{a(m_1)\overline{b(m_2)} \sum_{q\mid n, (q, N_0)=1}q^{-v}}
	{(m_2\ell_2)^{s+k-1} n^{w+\frac{k-1}{2}}}
	=
	\sum_{n\geq 1} \frac{D(s; n) \tilde\sigma_{-v}(n)}{n^{w+\frac{k-1}{2}}}
	\\
	=
	\sum_{n, m_2\geq 1, \atop \ell_1m_1= \ell_2m_2+n}
	\frac{a(m_1)\overline{b(m_2)} \tilde\sigma_{-v}(n)}
	{(m_2\ell_2)^{s+k-1} n^{w+\frac{k-1}{2}}}
	=
	\sum_{m_2\geq 1} \frac{D'(w, v; m_2\ell_2) \overline{b(m_2)}}{(m_2\ell_2)^{s+k-1}}
	.
\end{multline}

With this notation, our objective is to prove the following:
\begin{thm}\label{t:triple}
The multiple Dirichlet series $\cM(s, w, v)$ 
converges absolutely when $\Re (s)$, $\Re(w) > 1$, $\Re (v) > 0$ and has a meromorphic continuation to the region $\left\{ (s,w,v) | \Re \, v> 0\right\}$.  
	It is analytic in this region, except (possibly) for poles of the same order of the zeros of the zeta functions appearing in the denominators of either $\Phi^{(1)}$ or $\Phi^{(2)}$, 
	and the following simple polar planes:
\begin{itemize}
\item $s-3/2+s'=-r$, for $r\ge 0$, 
\item $s-3/2+s'+v=-r$ for $r\ge 0$, 
\item $2s=1-r$ for $r\geq 0$
\item $s = 1/2 + it_j -r$ for $r \ge 0$,  
\item $w+v/2 = 1/2 + it_j -r$ for $r \ge 0$,
\item $w+v/2 = -r$, for $r \ge 0$,  
\item
$2s'+v=2$
.
\end{itemize}
Here $\Phi^{(1)}(s, w,v)$ and $\Phi^{(2)}(s, w, v)$ are given in Proposition \ref{prop:M1} and Proposition \ref{prop:M2}, respectively. 

For $\Re(s)< 1/2-k/2$, $\Re(w)>1$ and $\Re(v)>0$, we have the spectral expansion for $\cM(s, w, v)$ given in Proposition \ref{prop:M1}.
And, for $\Re(v)>0$, $\Re(w+v/2) < 1/2-k/2$ and $2\Re(s)+\Re(w)+\Re(v)/4 > -3k/4 + 5/2$, 
we have the spectral expansion for $\cM(s, w, v)$ given in Proposition \ref{prop:M2}.

\end{thm}

The remainder of this section is devoted to the proof of Theorem \ref{t:triple}.

\subsection{The analytic continuation of $D'(w,v; m_2 \ell_2) $}

Let $\Gamma:= \Gamma_0(N_0)$ and
$\Gamma_1 = \left\{\gamma\in \Gamma\;|\; \gamma 1 = 1\right\}$.
Then we have an Eisenstein series of $\Gamma$ extended at the cusp $1$: 
$$
	E_1(z, (1+v)/2)
	=
	\sum_{\gamma\in \Gamma_1 \bsl \Gamma} \Im(\sigma_1^{-1} \gamma z)^{\frac{v+1}{2}}
	.
$$
Here $\sigma_1$ is the matrix defined in \eqref{e:sigma1/a}.
The Fourier expansion of $E_1(z, (1+v)/2)$ is 
\begin{multline*}
	E_1(z, (1+v)/2)
	=
	\frac{\sqrt\pi\Gamma\left(\frac{v}{2}\right)}{\Gamma\left(\frac{1+v}{2}\right)}
	\rho_1\left(\frac{1+v}{2}, 0\right)
	\\
	+
	\frac{ 2\pi^{\frac{1+v}{2}} }{ \Gamma\left(\frac{1+v}{2}\right) }
	\sum_{n\neq 0} \rho_1\left(\frac{1+v}{2}, n\right) |n|^{\frac{v}{2}} 
	\sqrt y K_{\frac{v}{2}} (2\pi |n|y) e^{2\pi inx}
\end{multline*}
where
$$
	\rho_1\left(\frac{1+v}{2}, 0\right)
	=
	\frac{ N_0^{-\frac{1+v}{2}} \prod_{p\mid N_0} \left(1-p^{-1-v}\right)^{-1} \zeta(v)} 
	{ \zeta(1+v) }
$$
and
$$
	\rho_1\left(\frac{1+v}{2}, n\right)
	=
	\frac{ N_0^{-\frac{1+v}{2}}\prod_{p\mid N_0} (1-p^{-1-v})^{-1} }{ \zeta(1+v) }  
	\tilde\sigma_{-v}(n) 
$$
for $n\neq 0$.
For $\ell\in \Z$, the Maass raising operator $R_\ell$ is defined to be the differential operator:
$$
	R_\ell := iy \frac{\partial}{\partial x} + y\frac{\partial}{\partial y} + \frac{\ell}{2}
	. 
$$
Define
\begin{multline*}
	E_1^{(k)}(z, (1+v)/2)
	:=
	\frac{(-1)^{\frac{k}{2}} \Gamma\left(\frac{1+v}{2}\right)} {\Gamma\left(\frac{1+v}{2}+\frac{k}{2}\right) N_0^{-\frac{1+v}{2}} \prod_{p\mid N_0} \left(1-p^{-1-v}\right)^{-1}}
	\\
	\times
	\left(R_{k-2} \circ \cdots \circ R_0 E_1\right) (z, (1+v)/2) 
	. 
\end{multline*}
By applying a Maass raising operator, we get an Eisenstein series of even weight $k$ on $\Gamma$ at cusp $1$ as follows:
$$
	E_1^{(k)}(z, (1+v)/2)
	=
	\sum_{m\in \Z}
	c^{(k)}(m, y, (1+v)/2) e^{2\pi imx}
	,
$$
with 
\be\label{c0def}
	c^{(k)}(0, y, (1+v)/2)
	\\
	=
	\frac{ (-1)^{\frac{k}{2}}   
	2^{1-v} \pi
	\Gamma(v) \zeta(v) \Gamma\left(\frac{1-v}{2}+\frac{k}{2}\right)}
	{\Gamma\left(\frac{1+v}{2}\right) \zeta(1+v) \Gamma\left(\frac{1-v}{2}\right) 
	\Gamma\left(\frac{1+v}{2}+\frac{k}{2}\right)}
	y^{\frac{1-v}{2}}
\ee
and for $m\neq 0$, 
\be\label{cdef}
	c^{(k)}(m, y, (1+v)/2)
	=
	\frac{
	\pi^{\frac{1+v}{2}} \tilde\sigma_{-v}(n) |n|^{\frac{-1+v}{2}}}
	{\zeta(1+v) \Gamma\left(\frac{1+v}{2}+({\rm sgn})\frac{k}{2}\right)} 
	W_{{\rm sgn}(m)\frac{k}{2}, \frac{v}{2}} (4\pi|n|y) e^{2\pi inx}
	.
\ee

Our objective is to prove:
\begin{prop}\label{prop:Icontinuation}
Fix $\epsilon >0$, $v$ with  $\Re \, v>0$ and $B \gg 1$.  
	For  $\Re \, w  >1+\epsilon$, the Dirichlet series $D'(w,v; m_2 \ell_2)$ defined in \eqref{e:D'} converges absolutely.  
It has an analytic continuation to all $w$ with $\Re \, w >1/2-B$,  and $v$ with  $\Re \, v \ge 0$, 
except for simple poles when $w+v/2 =1/2 +it_j - r$, for $r \ge 0$,  
and possible simple poles when $w+v/2= -r$.  
It also has poles as the same order as the zeros of $\zeta(2w+v+2b)$ for $b\geq 0$. 
The residues at the points $w+v/2 =1/2 +it_j - r$ are given by
\be\label{e:R2def}
	\Res_{w+v/2 =  1/2 +it_j - r} D'(w,v; m_2\ell_2)
	= 
	d_{r,j} (\ell_2m_2)^{r-it_j}\overline{\lambda_j(\ell_2m_2)},
\ee
where
\begin{multline}\label{e:drjdef}
	d_{r,j} 
	= \frac{2^{-1+k/2-v/2}\pi^{-1+k/2-v/2}\G((k+1+v)/2)}
	{\G(k/2+it_j -r +v/2)\G(k/2+it_j -r -v/2 )}
	\\
	\times  
	\frac{\G(1/2 + it_j  -r)\G(1/2-it_j+r)\G( 2it_j-r)(-1)^r 
	\overline{\rho_j(-1)\<U,u_j\>}}
	{r ! \G(1/2 + i t_j) \G(1/2 - i t_j)}
\end{multline}
and $U(z) = y^{k/2}\overline{f(\ell_1z)}E_1^{(k)}(z, (1+\bar v)/2)$.
For $T \gg1$ the $d_{r,j}$  satisfy the average upper bounds
$$
 	\sum_{|t_j|\sim T}|d_{r,j}|^2 \ll T^{2r + 1} \ell_1^{ -k}.
$$ 

When $\Re \left( w+v/2\right) < 1/2 -k/2$ and $w+v/2$ is at least $\epsilon$ away from the poles,
$D'(w,v,m_2)$ is given by the following absolutely convergent sum and integral:
\begin{multline}\label{e:D'exp}
	D'(w,v; \ell_2 m_2)
	= 
	\pi^{(k-v-1)/2}2^k (\ell_2 m_2)^{1/2 -w -v/2} 
	\\ 
	\times
	\frac{\G(1-w-v/2)\G(w+v/2)\G((k+1+v)/2)}{\G(w +(k-1)/2)\G(w +(k-1)/2 +v)}	
	\left(S'_{\rm{cusp}}(w, v; \ell_2m_2)+
	S'_{\rm{cont}}(w, v;\ell_2 m_2\right) \\
	-
	D'_{\rm finite}(w, v; \ell_2m_2)
\end{multline}
where
\begin{multline}\label{e:D'finite}
	D'_{\rm finite}(w, v; \ell_2 m_2)
	\\
	:=
	\frac{\G(1-w-v/2)\G(w+v/2)}{\G((1-v+k)/2)\G((1+v-k)/2}
	\sum_{1 \le m' <\ell_2 m_2}
	\frac{a((\ell_2 m_2 - m')/\ell_1 )\tilde\sigma_{-v}(m')}{(m')^{w+ (k-1)/2}}
	.
\end{multline}
Here
\begin{multline}\label{e:S'cusp}
	S'_{\rm{cusp}}(w, v; \ell_2m_2) 
	\\
	= 
	\sum_j 
	\frac{\overline{\rho_j(-\ell_2 m_2)}
	\G(w+v/2-1/2+it_j) \G(w+v/2-1/2-it_j)
	\overline{\left<U,u_j \right>}}{\G(1/2+it_j)\G(1/2-it_j)}
\end{multline}
and
\begin{equation}\label{e:S'cont}
	S'_{\rm{cont}}(w, v; \ell_2m_2)
	=
	S'_{\rm{int}}(w, v; \ell_2 m_2)
	+
	\Omega'(w, v; \ell_2 m_2)
\end{equation}
where
\begin{multline}\label{e:S'int}
	S'_{\rm int}(w, v; \ell_2m_2)
	=
	\sum_\cusp 
	\frac{1}{2\pi i}
	\int_{\overline{-C_\sigma}} 
	\frac{\vol \pi^{\frac{1}{2}-z} (\ell_2m_2)^{-z} \rho_\cusp\left(\frac{1}{2}-z, -\ell_2m_2\right)}
	{\Gamma(1/2 - z)} 
	\\
	\times 
	\frac{\G(w + v/2-1/2+it)\G(w + v/2-1/2-it)}
	{\Gamma\left(\frac{1}{2}+it\right)\Gamma\left(\frac{1}{2}-it\right) \zeta^*(1+2z)}
	\overline{\<U,E^*_\cusp(*,1/2 + z)\>}dz
\end{multline}
and
\begin{multline}\label{e:Omega'}
	\Omega'(w, v; \ell_2 m_2)
	\\
	=
	\sum_\cusp 
	\sum_{b=0}^{\lfloor\frac{1}{2}-\sigma\rfloor}
	\left( \frac{\vol \pi^{1-(w+v/2)-b}(\ell_2m_2)^{\frac{1}{2}-(w+v/2)-b} 
	\rho_\cusp\left(1-(w+v/2)-b, -\ell_2m_2\right)}
	{\Gamma(1-(w+v/2)-b) \zeta^*(2w+v+2b)} 
	\right.
	\\
	\left.
	\times
	\overline{\left<U, E^*_\cusp\left(*, w+v/2+b \right)\right>} 
	\right.
	\\
	+ 
	\left. 
	(1-\pmb{\delta}_{\sigma, b}) 
	\frac{\vol \pi^{w+v/2+b} (\ell_2 m_2)^{-\frac{1}{2}+w+v/2+b} 
	\rho_\cusp \left(w+v/2+b, -\ell_2 m_2\right)}
	{\Gamma\left(w+v/2+b\right) \zeta^*(2-2w-v-2b)} 
	\right.
	\\
	\times
	\overline{\left<U, E^*_\cusp\left(*, 1-(w+v/2)-b\right)\right>} 
	\bigg)
	\\
	\times 
	\frac{(-1)^b \sqrt\pi 2^{\frac{1}{2}-w-\frac{v}{2}} \Gamma\left(1-w-\frac{v}{2}\right) 
	\Gamma\left(2w+v+b-1\right) }
	{b! \Gamma\left(w-\frac{v}{2}+b\right) \Gamma\left(1-w-\frac{v}{2}-b\right)}
	\,.
\end{multline}
Here $\sigma = \Re(w+v/2)$.	
In this region  $D'(w,v; \ell_2m_2)$ satisfies the upper bound 
\begin{multline}\label{e:I0upper}
	D'(w,v; \ell_2m_2) + D'_{\rm finite}(w, v; \ell_2 m_2)
	\ll   
	\ell_1^{(1-k)/2}
	(\ell_2 m_2)^{1/2 -w -v/2 + \theta}
	\\ 
	\times
	\frac{(1+|v|)^{k/2}}{(1+|w|)^{w + k/2 -1}(1+|w+v|)^{w+k/2 -1+v}}.
\end{multline}

For $w+v/2$ at least $\epsilon$ away from the polar points, 
	and in the vertical strip 
	$c <\Re \left(w +v/2\right) < 2$,  for any $c$ with $1/2 - B + \epsilon <c <1/2 -k/2$, 
it satisfies the upper bound
\begin{multline}\label{e:I0upper2}
	D'(w,v; \ell_2 m_2) + D'_{\rm finite}(w, v; \ell_2 m_2)
	\ll  
	\ell_1^{(1-k)/2} (\ell_2m_2)^{1/2 -c + \theta} 
	\\ 
	\times
	\frac{(1+|v|)^{k/2}}{(1+|w|)^{c -1}(1+|w+v|)^{c+k/2 -1+v/2}}.
\end{multline}

\end{prop}

\begin{proof}


For reference, we record here the behavior of $W_{\lambda, \mu}( y)$ at large and small values of $y >0$.    
When $\Re(\mu) \neq 0$, 
$$
	W_{\lambda, \mu}( y) \ll y^{1/2 - | \Re \, \mu|}, 
$$
as $y\to 0$. 
When $\Re \mu =0$, 
$$
	W_{\lambda, \mu}( y) \ll y^{1/2} | \log y|, 
$$
as $y\to 0$.
For $y \gg 1$, 
$$
	W_{\lambda, \mu}( y) \ll e^{-y/2}y^\lambda.
$$

Recall that, for any fixed $Y\gg 1$ and $\delta>0$, we defined the modified Poincare series 
	$P_{h, Y}(z, s; \delta)$ in \cite{HH}:
$$
	P_{h, Y}(z, s; \delta)
	=
	\sum_{\gamma\in \Gamma_\infty \bsl \Gamma} \psi_Y(\Im(\gamma z)) (\Im(\gamma z))^s 
	e^{-2\pi i h \Re(\gamma z) + (2\pi h \Im(\gamma z))(1-\delta)}
$$
where $\psi_Y$ is the characteristic function of the interval $[Y^{-1}, Y]$. 
As in \cite{HH}, we get
\begin{multline*}
	\frac{1}{\vol} \iint_{\Gamma\bsl \bH} 
	P_{\ell_2 m_2, Y}(z, w+v/2; \delta) y^{k/2} f(\ell_1 z) \overline{E_1^{(k)}(z, (1+\bar v)/2)}
	\; \frac{dx \; dy}{y^2}
	\\
	=
	\frac{1}{\vol} \int_0^\infty
	\sum_{m\geq 1, \atop \ell_1m_1 = \ell_2m_2+m} a(m_1) \overline{c^{(k)}(m, y, (1+\bar v)/2) }
	e^{2\pi(\ell_2m_2(1-\delta) - \ell_1 m_1) y}
	\psi_Y (y) y^{w+v/2+k/2-1} \; \frac{dy}{y}
	\\
	+
	\frac{1}{\vol} \int_0^\infty
	\sum_{-\ell_2m_2 < m\leq 0} a(m_1) \overline{c^{(k)}(m, y, (1+\bar v)/2) }
	\\
	\times
	e^{2\pi(\ell_2m_2(1-\delta) - \ell_1 m_1) y}
	\psi_Y (y) y^{w+v/2+k/2-1} \; \frac{dy}{y}
	.
\end{multline*}

In analogy to $\cI_{\ell_1,\ell_2,Y,\gd}(s;{\sh})$, 
	we now define, for $\Re(v) > 0$, and for $w$ with $\Re \left(w +v/2\right) >1$, 
\begin{multline}\label{Itwodef}
\cI (w,v; m_2,\ell_1,\ell_2, \delta,Y)
	\\
	:= 
	\frac{1}{\vol} \iint_{\G\bk\boldH} 
	P_{\ell_2 m_2,Y} (z,w+v/2;\delta)y^{k/2}f(\ell_1 z) 
	\overline{E_1^{(k)}(z,(1+\bar v)/2)}{dx\, dy\over y^2}
	\\
	 -  
	\frac{1}{\vol} \int_0^\infty
	\sum_{-\ell_2m_2 < m\leq 0} a(m_1) \overline{c^{(k)}(m, y, (1+\bar v)/2) }
	\\
	\times
	e^{2\pi(\ell_2m_2(1-\delta) - \ell_1 m_1) y}
	\psi_Y (y) y^{w+v/2+k/2-1} \; \frac{dy}{y}
	.
\end{multline}
Although originally defined for $\Re \left(w +v/2\right) >1$, the meromorphic continuation of the inner product is obtainable.  The subtracted term is defined and convergent for $\Re \left(w +v/2 \right) +k/2>1$, and consequently has no poles in this region.  

Write
\be\label{ItoD}
	\cI (w,v; m_2,\ell_1,\ell_2,\delta,Y) 
	=
	\cI_0 (w,v; m_2,\ell_1,\ell_2,\delta,Y)
	-
	\sum_{0\le m'  <\ell_2m_2} d_{\delta,Y}(w,v; m',m_2),
\ee
where
\be\label{cI0def}
	\cI_0 (w,v; m_2,\ell_1,\ell_2, \delta,Y)
	=
	\\ 
	\<  P_{\ell_2m_2,Y}(*,w+v/2;\delta),y^{k/2}\overline{f(\ell_1 *)}
	E_1^{(k)}(*,(1+v)/2) \>
\ee
and
\begin{multline}\label{Dm2def}
	d_{\delta,Y}(w,v; m', \ell_2 m_2)
	=  
	\frac{1}{\vol} \int_0^\infty
	a(m_1) \overline{c^{(k)}(-m', y, (1+\bar v)/2) }
	\\
	\times
	e^{2\pi(m'-\ell_2m_2\delta)) y}
	\psi_Y (y) y^{w+v/2+k/2-1} \; \frac{dy}{y}
	.
\end{multline}

First we will interpret the integral $\cI_0$ as a Dirichlet series, 
	and take the difference in \eqref{ItoD}, removing the coefficients in $\cI_0$ with 
	negative indices. 
	We will then take the limit as first $Y\rightarrow \infty$, 
	and then $\delta \rightarrow 0$.

Taking the difference yields, 	
\begin{multline}\label{Iopen}
	\cI (w,v,m_2,\ell_1,\ell_2, \delta,Y) 
	\\
	=
	\frac{1}{\vol}  \int_{Y^{-1}}^Y y^{w+k/2+v/2 -1}  
	\sum_{ m  \ge 1}
	a((\ell_2m_2+ m)/\ell_1) e^{-2 \pi y (m+ \delta \ell_2 m_2)}
	\overline{c^{(k)}(m,y,(1+\bar v)/2)} \; \frac{dy}{y}.
\end{multline}

For $m \ge 1$, \eqref{cdef} becomes
$$
	\overline{c^{(k)}(m,y,(1+\bar v)/2)}
	= 
 	\frac{\sigma_{-v}(|m|)\pi^{(v+1)/2} |m|^{(v-1)/2}}{\zeta(v+1)\G( k/2 +(1+v)/2)}	
	W_{ k/2,-v/2}(4 \pi |m| y).
$$
Also, by \cite{GR}, 7.621.3, 
\begin{multline*}
	\int_0^\infty y^{w+k/2+v/2 -1} 
	e^{-y(1+\delta)} W_{ k/2,-v/2}(2  y)\; \frac{dy}{y}
	\\
	=
	\frac{\G(w+k/2 - 1/2)\G(w+k/2 - 1/2 +v)2^{-w-k/2 +1 - v/2}}
	{\G(w+v/2)}\left(1 + \mathcal{O}(\delta)\right).
\end{multline*}

Substituting these two facts into \eqref{Iopen} 
	we first take the limit as $Y \rightarrow \infty$, obtaining
\begin{multline*}
	\lim_{Y \rightarrow \infty}
	\cI (w,v; m_2,\ell_1,\ell_2,\delta,Y)
	\\ 
	= 
	\frac{\pi^{-w-k/2+3/2}2^{-2w-v-k+2} \G(w +(k-1)/2)\G(w+(k-1)/2 +v)}
	{\vol \zeta(1+v)\G((k+1 +v)/2)\G(w + v/2)}
	\\
	\times 
	\sum_{m \ge 1} \frac{a((\ell_2m_2+ m)/\ell_1) \sigma_{-v}(m)m^{(v-1)/2}}
	{(m+ \delta \ell_2 m_2)^{w+v/2 +k/2 -1}}
	\left(1 + \mathcal{O}(\delta)\right).
\end{multline*}
Then, taking the limit as $\delta \rightarrow 0$, we define
\be\label{Iseries}
	\cI (w,v; m_2,\ell_1,\ell_2) 
	= 
	\lim_{\delta \rightarrow 0} \lim_{Y \rightarrow \infty}
	\cI (w,v; m_2,\ell_1,\ell_2,\delta,Y)
	\\ 
	= G(w,v) D'(w,v; \ell_2 m_2), 
\ee
where
\be\label{Gdef}
	G(w,v) 
	\\
	=
	\frac{\pi^{-w-k/2+3/2}2^{-2w-v-k+2}\G(w +(k-1)/2)\G(w+(k-1)/2 +v)}
	{\vol \zeta(1+v)\G((k+1 +v)/2)\G(w + v/2)}
	.
\ee

This representation is valid for $\Re \left(w+v/2\right) >1$.
	Our objective is to now use the representation of 
	$\cI (w,v; m_2,\ell_1,\ell_2,\delta,Y)$ given in \eqref{ItoD} 
	to find a spectral decomposition of $\cI_0(w,v; m_2,\ell_1,\ell_2,\delta,Y)$.
	We will then take a limit again as $Y \rightarrow \infty$ 
	and $\delta \rightarrow 0$ in the region $\Re \left(w+ v/2\right) <1/2-k/2$, 
	connecting this region, in the process, to the region $\Re \left(w+v/2\right) >1$.  
	We will thus obtain a meromorphic continuation of  
	$\cI (w,v; m_2,\ell_1,\ell_2)$, and hence $D'(w,v; \ell_2 m_2)$, to all $w,v$ 
	with $\Re \left(w+ v/2\right) > 1/2 -B$, for a sufficiently large $B \gg 1$, 
	as was done in the case of $D(s;h)$.

For $m' \ge 1$, 
\begin{multline}\label{P1}
	d_{\delta,Y}(w,v; m',\ell_2 m_2) 
	\\
	=
	\frac{2^{-w-v/2-k/2+1} \pi^{3/2-w-k/2}}
	{\vol \zeta(1+v)\G( (1+v)/2 - k/2)}
 	\frac{ a((\ell_2 m_2 - m')/\ell_1) \sigma_{-v}(m')}
	{(m')^{w +(k-1)/2}}\\
	\times  
	\int_{Y^{-1}2 \pi m'}^{Y2 \pi m'} 
	y^{w+k/2+v/2 -1} W_{ -k/2,-v/2}(2 y) e^{ y (1- \delta \ell_2 m_2/m')}
	\; \frac{dy}{y}.
\end{multline}
The integral here is very similar to that studied in Section 3, \cite{HH}.  
	It has been analyzed and its analytic continuation obtained, 
	by Tom Hulse in \cite{Hulse}.  
	Using his notation, upon taking the limit as $Y \rightarrow \infty$, we have
\begin{multline*}
	\lim_{Y \rightarrow \infty} 
	\int_{Y^{-1}2 \pi m'}^{Y2 \pi m'} 
	y^{w+k/2+v/2 -1} W_{ -k/2,-v/2}(2 y) e^{ y (1- \delta \ell_2 m_2/m')}
	\; \frac{dy}{y}
	\\
	=M_2(w+k/2 + v/2, iv/2, \delta  \ell_2 m_2 /m'),
\end{multline*}
In \cite{Hulse}, it is shown in Proposition 3.1 that
	for each $\delta >0$ the function $M_2(w+k/2 + v/2, iv/2, \delta  \ell_2m_2 /m')$ 
	has analytic continuation to $\C$, 
	with at most polynomial growth in vertical strips in the variables $w,v,\delta^{-1}$, 
	except for simple poles at $w+k/2 + v/2 = 1/2 \pm v/2  -r$, for $r\ge 0$.  
	Note that as in Proposition 3.1 in \cite{HH}, there is no pole at $w+ v/2=1$.
	It is also shown that for $ \Re \left( w+k/2 + v/2\right) \le 1/2$,
	the limit as $\delta \rightarrow 0$  
	of $M_2(w+k/2 + v/2, iv/2, \delta  \ell_2m_2/m')$, 
	exists and that we have, for  $ \Re \left( w+ v/2\right)+k/2 \le 1/2$,
\begin{multline*}
	M_2(w+k/2 + v/2, iv/2)
	= 
	\lim_{\delta \rightarrow 0} 
	M_2(w+k/2 + v/2, iv/2, \delta \ell_2m_2 /|m|)
	\\
	=
	\frac{2^{1-w-k/2 - v/2}\G(w+k/2 -1/2  ) \G(w+k/2 + v-1/2  ) \G(1-w - v/2)}
	{\G((1+k+v)/2)\G((1+k-v)/2)}.
\end{multline*}
With this notation, after taking the limit as $Y \rightarrow \infty$ and  
$\delta\rightarrow 0$ we set
\begin{multline}\label{P1exp}
	d(w,v; m',\ell_2 m_2)  
	= 
	\lim_{\delta \rightarrow 0 \atop Y \rightarrow \infty}d_{\delta,Y}(w,v; m', \ell_2 m_2)  
	\\ 
	=
	\frac{2^{-w-v/2-k/2+1}\pi^{3/2-w-k/2}}
	{\vol \zeta(1+v)\G( (1+v)/2 - k/2)} 
	M_2(w+k/2 + v/2, iv/2)
	\frac{ a((\ell_2m_2 -m')/\ell_1) \tilde\sigma_{-v}(m')} 
	{(m')^{w +(k-1)/2}}
	\\
	= 
	R(w,v) F(w,v; m',\ell_2m_2),
\end{multline}
where
\begin{multline}\label{Rexp}
	R(w,v) 
	= 
	\frac{4^{-w-v/2-k/2+1}\pi^{3/2-w-k/2}\G(w + (k-1)/2)}
	{\vol \zeta(1+v)\G( (1+v-k)/2 ) \G( (1+v+k)/2) }
	\\ 
	\times 
	\frac{\G(w + (k-1)/2 +v)\G(1-w-v/2)}{\G( (1-v+k)/2 )}
\end{multline}
and
\be\label{Fexp}
	F(w,v; m',m_2)
	=
	\frac{ a((\ell_2m_2 -m')/\ell_1)\tilde\sigma_{-v}(m')}{(m')^{w +(k-1)/2}}.
\ee

Turning to the case $m'=0$, 
	we have
\be\label{ceekay}
	\overline{c^{(k)}(0,y,(1+\bar v)/2) }
	= 
	c(v) y^{(1-v)/2}, 
\ee
where $c(v)$ is given as in \eqref{c0def}. 
Substituting the above into  \eqref{Dm2def} gives
$$
	d_{\delta,Y}(w,v,0,m_2)
	=
	\frac{1}{\vol}  
	\int_{Y^{-1}}^Y y^{w+k/2+v/2 -1}  
	a((\ell_2m_2)/\ell_1) e^{-2 \pi y ( \delta \ell_2m_2)}
	\overline{c^{(k)}(0, y,(1+v)/2)} \; \frac{dy}{y},
$$
and then expanding $c^{(k)}(0,y,(1+v)/2)$ and letting $y \rightarrow y/(2\pi \delta \ell_2 m_2)$ yields
$$
	d_{\delta,Y}(w,v; 0,\ell_2 m_2)
	= 
	\frac{a((\ell_2 m_2)/  \ell_1)}{\vol} 
	\\ 
	\times 
	\frac{c(v)}{(2\pi \delta \ell_2 m_2)^{w+k/2-1/2}}
	\int_{2\pi \delta \ell_2 m_2 Y^{-1}}^{2\pi \delta \ell_2 m_2  Y} 
	y^{w+k/2-1/2} e^{-y} \; \frac{dy}{y} 
	.
$$
For $\Re \left(w +(v\pm v)/2\right)+k/2 > 1/2$, 
	this is absolutely convergent as $Y \rightarrow \infty$, and converges to
\be\label{twogammas}
	d_{\delta}(w,v;0,\ell_2 m_2)
	:= 
	\lim_{Y \rightarrow \infty} d_{\delta,Y} (w,v,0,m_2) 
	=  
	\\ 
	\frac{a((\ell_2 m_2)/  \ell_1)} {\vol}
	\frac{c(v)\Gamma(w+k/2-1/2)}{(2\pi \delta \ell_2 m_2)^{w+k/2-1/2}} 
\ee
which has a continuation to all $w, v \in \C$ with poles 
	at $w+k/2+(v \pm v)/2-1/2=-r$ for $r \geq 0$. 
	Furthermore, we note that for $\Re \left(w +(v\pm v)/2\right) + k/2 < 1/2$, 
	the limit of $d_{\delta}(w,v; 0,m_2)$ as $\delta \rightarrow 0 $ is $0$.

Recalling \eqref{ItoD},
	the next step in the meromorphic continuation of 
	$\cI (w,v; m_2,\ell_1,\ell_2)$, and hence also of
	$D(w,v; m_2)$ will be achieved once the analytic continuation of
\be\label{I0limitdef}
	\cI_0 (w,v; m_2,\ell_1,\ell_2) 
	= 
	\lim_{\delta \rightarrow 0 \atop Y \rightarrow \infty}
	\cI_0 (w,v; m_2,\ell_1,\ell_2,\delta,Y) 
\ee
is obtained.
	To accomplish this, the same analysis 
	as in Section 4 in \cite{HH}, applies, as we are using 
	the same Poincare series (with $1$ in place of $\ell_2$) 
	and only changing the automorphic function $V$ to
$$
	U(z; v) 
	= U_{\ell_1}(z; v) 
	:= 
	y^{k/2}\overline{f(\ell_1 z)}E_1^{(k)}(z,(1+\bar v)/2).
$$
The function $U$ still vanishes at the cusps and is automorphic with respect to 
	$\G_0(N_0\ell_1)$.   
	The resulting spectral expansion of the triple product is
\begin{multline*}
	\cI_0 (w,v; m_2,\ell_1,\ell_2, \delta,Y) 
	=
	\< P_{\ell_2m_2,Y}(*,w+v/2;\delta),U \>
	\\
	=
	\sum_{j\ge1}\overline{\<U,u_j\>}  
	\<P_{\ell_2m_2,Y}(    *    ,w+v/2;\delta),u_j\>
	\\
	+
	\frac{1}{4 \pi} \sum_\cusp \int_{-\infty}^\infty 
	\vol \overline{\left<U,E_\cusp\left(*,\frac{1}{2} + it\right)\right>}
	\left<P_{\ell_2m_2,Y}\left(*, w+\frac{v}{2};\delta\right), E_\cusp \left(*,\frac{1}{2}+ it\right)\right>dt 
 \,.
 \end{multline*}
 So
 \begin{multline}\label{prespec}
 	\cI_0(w, v; m_2, \ell_1, \ell_2,  \delta)
	:=
 	\lim_{Y\to \infty} \cI_0(w, v; m_2, \ell_1, \ell_2, \delta, Y)
	\\
	=
	\cI_{0, {\rm cusp}}(w, v; m_2, \ell_1, \ell_2, \delta)
	+
	\cI_{0, {\rm cont}}(w, v; m_2, \ell_1, \ell_2, \delta)
\end{multline}
where
\be\label{e:I0cusp}
	\cI_{0, {\rm cusp}}(w, v; m_2, \ell_1, \ell_2, \delta)
	=  
	\frac{(2 \pi\ell_2 m_2)^{\frac{1}{2} -w - v/2}}{\vol}
	\sum_j \overline{\rho_j(-\ell_2m_2)}  M(w+v/2,t_j,\delta) \overline{\<U,u_j\>}
\ee
and
\begin{multline}\label{e:I0cont}
	\cI_{0, {\rm cont}}(w, v, m_2, \ell_1, \ell_2, \delta)
	\\
	=
	\frac{(2 \pi\ell_2 m_2)^{\frac{1}{2} -w - v/2}}{\vol}
	\frac{1}{4\pi}
	\sum_\cusp 
	\int_{-\infty}^\infty
	\frac{\vol 2\pi^{\frac{1}{2}-it} (\ell_2m_2)^{-it} \rho_\cusp\left(\frac{1}{2}-it, -\ell_2m_2\right)}
	{\Gamma(1/2 - it) \zeta^*(1+2it)} 
	\\
	\times 
	M(w+v/2,t,\delta) \overline{\<U,E^*_\cusp (*,1/2 + it)\>}dt 
	.
\end{multline}

As was the case with $D(s;h)$, this spectral expansion has poles coming from the cuspidal contribution at $w+v/2 = 1/2+it_j -r$, with $r \ge 0$, and no poles originating in the continuous contribution when $\Re(w+v/2) > 1/2$. 
By the upper bound for $M(w+v/2,t,\delta)$ given in \cite{HH}, 
	this converges absolutely for any fixed $\delta>0$.  
	To determine the region of absolute convergence as $\delta \rightarrow 0$, 
	we note first that in \cite{HH}, 
	(and Stirling's formula), when $\delta$ is small relative to $t_j$
\be\label{Mwv}
	M(w+v/2,t,\delta) \ll_{w,v} (1 + |t_j|)^{2 \Re \, (w + v/2) -2},
\ee
where the implied constant depends on $w,v$.  

By applying Proposition 4.1 in \cite{HH}, for $T\gg 1$, 
\be\label{e:spec_bound}
	\sum_{j, |t_j|\sim T} e^{\frac{\pi|t_j|}{2}} \left<U, u_j\right> 
	\ll
	\ell_1^{\frac{1-k}{2}} T^{1+k+\frac{\Re(v)}{2}}
	.
\ee

Substituting \eqref{Mwv} and \eqref{e:spec_bound} into \eqref{prespec}, after applying Cauchy-Schwarz  and making a dyadic subdivision of the interval $[0,T]$ we see that the exponent of $T$ will be bounded above by $  2\Re (w+v/2) +k-1$ and thus will be negative when 
$\Re \, w+v/2  < 1/2-k/2$.  It follows that for such $w,v$ the spectral expansion \eqref{prespec} will converge absolutely as $\delta \rightarrow 0$.  The upper bound for $D'$ claimed in the proposition follows from an examination of of the different components of $\cI_0$ and the above discussion.

Let $\sigma=\Re(w+v/2)$. 
As was the case with $D(s; h)$, 
we get 
\begin{multline*}
	\cI_{0, {\rm cont}}(w, v;  m_2, \ell_1, \ell_2, \delta)
	\\
	=
	\frac{(2\pi \ell_2 m_2)^{\frac{1}{2} -w - v/2}}{\vol}
	\sum_\cusp 
	\left[
	\frac{1}{2\pi}
	\int_{\overline{-C_\sigma}} 
	\frac{\vol \pi^{\frac{1}{2}-z} (\ell_2 m_2)^{-z} 
	\rho_\cusp\left(\frac{1}{2}-z, -\ell_2m_2\right)}
	{\Gamma(1/2 - z) \zeta^*(1+2z)} 
	\right. 
	\\
	\times 
	M(w+v/2,z/i,\delta) \overline{\<U,E^*_\cusp(*,1/2 + z)\>}dz
	\\
	+
	\sum_{b=0}^{\lfloor\frac{1}{2}-\sigma\rfloor}
	\left( \frac{\vol \pi^{1-(w+v/2)-b}(\ell_2 m_2)^{\frac{1}{2}-(w+v/2)-b} 
	\rho_\cusp\left(1-(w+v/2)-b, -\ell_2 m_2\right)}
	{\Gamma(1-(w+v/2)-b) \zeta^*(2w+v+2b)}
	\right.
	\\
	\times
	\overline{\left<U, E^*_\cusp\left(*, (w+v/2)+b \right)\right>} 
	\\
	+ 
	(1-\delta_{\sigma, \ell}) 
	\frac{\vol \pi^{(w+v/2)+b} (\ell_2 m_2)^{-\frac{1}{2}+(w+v/2)+b} 
	\rho_\cusp \left((w+v/2)+b, -\ell_2 m_2\right)}
	{\Gamma\left((w+v/2)+b\right) \zeta^*(2-2w-v-2b)}
	\\
	\times
	\overline{\left<U, E^*_\cusp\left(*, 1-(w+v/2)-b\right)\right>} 
	\bigg)
	\\
	\times 
	\frac{(-1)^b \sqrt\pi 2^{\frac{1}{2}-w-\frac{v}{2}} 
	\Gamma\left(1-w-\frac{v}{2}\right) 
	\Gamma\left(2w+v+b-1\right) }
	{b! \Gamma\left(w-\frac{v}{2}+b\right) \Gamma\left(1-w-\frac{v}{2}-b\right)}
	\\
	\times
	\left.
	\left(1+ \cO_A\left(\left(1+|w+v/2|\right)^{\frac{9}{4}-2\sigma} \delta^{\frac{1}{4}-\epsilon}\right)\right)\right]
\end{multline*}

The full meromorphic continuation of  $\cI (w,v; m_2,\ell_1,\ell_2)$ 
	is then achieved just as in the case of $D(s; h)$, 
	by using the analysis of  $M_2(w+v/2,t,\delta)$ in \cite{Hulse}, 
	together with Cauchy's theorem, 
	to fill in the gap between the regions  $\Re \left(w+v\right) < 1/2 - k/2 $ where the spectral side  converges, 
	and the region  $\Re \left(w+v/2 \right) >1$, where the Dirichlet series converges.  
\end{proof}

\subsection{The analytic continuation of $\cM(s, w, v)$ I}
As defined in \eqref{e:Mswv}, we have
\be\label{e:M1exp}
	\cM(s, w, v)
	=
	(\ell_1\ell_2)^{\frac{k-1}{2}} 
	\sum_{n\geq 1} \frac{D(s; n) \tilde\sigma_{-v}(n)} {n^{w+\frac{k-1}{2}}}
\ee
and this series converges absolutely for $\Re(s)>1$, $\Re(w)>1$ and $\Re(v)>0$. 

Let
\begin{multline}\label{R1def}
	R_1 
	=  
	\left\{ (s,w,v)| \Re \, v > 0 \right\} 
	\cap 
	\left[
 	\left\{ (s,w,v) | \Re \, s \ge 1/2 -k/2,  \; \Re \, w >1 \right\}  
	\cup \right. 
	\\
	\left. 
	\left\{ (s,w,v) | 1/2 -A < \Re \, s < 1/2 -k/2, \;  \Re \, w > -\Re \,s/2 -k/4 + 5/4 \right\}
	\right].
\end{multline}

\begin{prop}\label{prop:M1}
The triple Dirichlet series $\cM(s, w, v)$ has a meromorphic continuation in $R_1$. 

For $1/2-A < \Re(s) < 1/2-k/2$, $\Re(w) > 1$ and $\Re(v)>0$, 
	we have the following spectral expansion for $\cM(s, w, v)$:
$$
	\cM(s, w, v) 
	=
	(\ell_1\ell_2)^{\frac{k-1}{2}} 
	\left(
	\cM^{(1)}_{\rm cusp}(s, w, v)
	+
	\cM^{(1)}_{\rm cont}(s, w, v)
	\right)
$$
and
$$
	\cM^{(1)}_{\rm cont}(s, w, v)
	=
	\cM^{(1)}_{\rm int}(s, w, v)
	+
	\Phi^{(1)}(s, w, v)
	,
$$
where
\begin{multline}\label{e:M1cusp}
	\cM^{(1)}_{\rm cusp}(s, w, v)
	=
	\frac{(4\pi)^k \Gamma(1-s)}{2\Gamma(s+k-1)}
	\sum_j 
	L\left(s', \overline{u_j} \otimes E_1\left(*, (v+1)/2\right)\right)
	\\
	\times
	\frac{\Gamma\left(s-\frac{1}{2}-it_j\right) \Gamma\left(s-\frac{1}{2}+it_j\right)}
	{\Gamma\left(\frac{1}{2}-it_j\right)\Gamma\left(\frac{1}{2}+it_j\right) }
	\overline{\left<V, u_j\right>} 
	,
\end{multline}
\begin{multline}\label{e:M1int}
	\cM^{(1)}_{\rm int}(s, w, v)
	=
	\frac{(4\pi)^k  \Gamma(1-s) }{2\Gamma(s+k-1)}
	\sum_{\cusp=\frac{1}{a}, a\mid N} 
	\frac{1}{2\pi i} 
	\int_{\overline{-C_{\Re(s), \Re(s'), \Re(v)}}} 
	\frac{\vol 2\pi^{\frac{1}{2}-z} } {\Gamma\left(\frac{1}{2}-z\right) \zeta^*(1+2z)}
	\\
	\times
	\frac{\Gamma\left(s-\frac{1}{2}-z\right) \Gamma\left(s-\frac{1}{2}+z\right)}
	{\Gamma\left(\frac{1}{2}-z\right) \Gamma\left(\frac{1}{2}+z\right)}
	\\
	\times
	\frac{\zeta(s'+z)\zeta(s'-z) \zeta(s'+v+z)\zeta(s'+v-z) \tilde\cP_\cusp(s'; z, v)}
	{\zeta(1-2z) \zeta(2s'+v)}
	\overline{\left<V, E_\cusp^*\left(*, \frac{1}{2}+z\right)\right>} 
	\; dz
\end{multline}
and
\begin{multline}\label{e:Phi1}
	\Phi^{(1)} (s, w, v)
	=
	\frac{(4\pi)^k  \Gamma(1-s) }{2\Gamma(s+k-1)}
	\\
	\times
	\sum_{\cusp=\frac{1}{a}, a\mid N} 
	\left[
	\pmb{\delta}'_{\Re(s')} 
	\frac{\vol 2 \Gamma\left(s+s'-\frac{3}{2}\right) \Gamma\left(s-s'+\frac{1}{2}\right) 
	\zeta(2s'-1)\zeta(v+1) \zeta(2s'+v-1) }
	{\Gamma\left(s'-\frac{1}{2}\right) \Gamma\left(\frac{3}{2}-s'\right) \zeta(2s'+v)}
	\right.
	\\
	\times
	\left(
	\frac{\tilde\cP_\cusp(s'; 1-s', v) \overline{\left<V, E_\cusp\left(*, \frac{3}{2}-s'\right)\right>}} {\zeta^*(2s'-1)} 
	\right.
	\\
	\left.
	+
	(1-\pmb{\delta}_{\Re(s'), -\frac{1}{2}}) 
	\frac{\tilde\cP_\cusp(s'; s'-1, v) \overline{\left<V, E_\cusp\left(*, s'-\frac{1}{2}\right)\right>}} {\zeta^*(3-2s')}
	\right)
	\\
	+
	\pmb{\delta}'_{\Re(s'+v)} 
	\frac{\vol 2 \Gamma\left(s+\frac{1}{2}-s'-v\right)\Gamma\left(s-\frac{3}{2}+s'+v\right) 
	\zeta(1-v) \zeta(-1+2s'+2v) \zeta(-1+2s'+v) }
	{\Gamma\left(-\frac{1}{2}+s'+v\right) \Gamma\left(\frac{3}{2}-s'-v\right) \zeta(2s'+v) }
	\\
	\times
	\left(
	\frac{\tilde\cP_\cusp(s'; 1-s'-v, v) \overline{\left<V, E_\cusp\left(*, \frac{3}{2}-s'-v\right)\right>}}
	{\zeta^*(-1+2s'+2v)} 
	\right.
	\\
	+
	\left.
	\left.
	(1-\pmb{\delta}_{\Re(s'+v), -\frac{1}{2}}) 
	\frac{\tilde\cP_\cusp(s'; -1+s'+v, v) \overline{\left<V, E_\cusp\left(*, -\frac{1}{2}+s'+v\right)\right>}}
	{\zeta^*(3-2s'-2v)}
	\right)
	\right]
	\\
	+
	\frac{(4\pi)^k}{2\Gamma(s+k-1) \zeta(1-2s)\zeta(2s'+v) }
	\sum_{\cusp=\frac{1}{a}, a\mid N}
	\sum_{\ell=0}^{\lfloor \frac{1}{2}-\Re(s)\rfloor}
	\frac{(-1)^\ell \Gamma(1-s) \Gamma(2s+\ell-1)}
	{\ell! \Gamma(s+\ell) \Gamma(1-s-\ell)}
	\\
	\times
	\zeta\left(s'-\frac{1}{2}+s+\ell\right) \zeta\left(s'+\frac{1}{2}-s-\ell\right)
	\zeta\left(s'+v-\frac{1}{2}+s+\ell\right) \zeta\left(s'+v+\frac{1}{2}-s-\ell\right) 
	\\
	\times
	\left( 
	\frac{\pi^{1-s-\ell}\vol 
	\tilde\cP_\cusp\left(s'; -\frac{1}{2}+s+\ell, v\right) 
	}
	{\Gamma(1-s-\ell) \zeta^*(2s+2\ell) }
	\overline{\left<V, E_\cusp^*\left(*, s+\ell\right)\right>}
	\right.
	\\
	+
	\left.
	(1-\pmb{\delta}_{\sigma, \ell}) 
	\frac{\pi^{s+\ell} \vol
	\tilde\cP_\cusp\left(s'; \frac{1}{2}-s-\ell, v\right) }
	{\Gamma(s+\ell) \zeta^*(2-2s-2\ell) } 
	\overline{\left<V, E_\cusp^*(*, 1-s-\ell)\right>}
	\right)
	\,.
\end{multline}
Here, for each $j$, the Rankin-Selberg $L$-function is defined as
\be\label{e:RSuj}
	L\left(s', \overline{u_j}\otimes E_1(*, (v+1)/2)\right)
	:=
	\sum_{n\geq 1} \frac{\overline{\rho_j(-n)} \tilde\sigma_{-v}(n)}{n^{s'}}
\ee
and for each cusp $\cusp=a/N$, for $a\mid N$, $\tilde\cP_\cusp$ is defined by
\be\label{e:tildePcusp}
	\frac{\zeta(s'+z) \zeta(s'-z) \zeta(s'+v+z) \zeta(s'+v-z)} {\zeta(1-2z) \zeta(2s'+v)} 
	\tilde\cP_\cusp(s'; z, v)
	=
	\sum_{n\geq 1} \frac{\rho_\cusp \left(\frac{1}{2}-z, -n\right) \tilde\sigma_{-v}(n)}{n^{s'+z}} 
	.
\ee

Then $\cM(s, w, v)- \Phi^{(1)}(s, w, v)$ has poles only at $s=1/2-r+it_j$ for $r\geq 0$, for $(s, w, v)\in R^1$.  
The function $\Phi^{(1)}(s, w, v)$ has at most polar planes at
 \begin{itemize} \label{itemize:phi1}
	\item $2s'+v=2$, 
	\item $s+1/2 -s' =-r$, 
	\item $s+1/2 -s' -v=-r$, 
	\item $s-3/2+s' = -r$, 
	\item $s-3/2+s' +v= -r$, 
	\item $2s+\ell-1 = -r$, for $0\leq \ell \leq \lfloor 1/2-\Re(s)\rfloor$,
\end{itemize}
where $r\ge 0$. 
In addition to the above polar planes, $\Phi^{(1)}(s, w, v)$ may have poles of the same order as the zeros of zeta functions in the denominator of $\Phi^{(1)}(s, w, v)$.  

In particular, in the part of $R_1$ corresponding to $\Re \, s <1/2 -k/2$,
  $\Re \, s'  \ge1/2$,  and $\Re(s'+v) \geq1$ 
and at points at least $\epsilon$ away from the poles, 
	it follows that  $\cM(s,w,v)$ obeys the upper bound 
\be\label{supper}
	\cM(s,w,v)- \Phi^{(1)}(s, w, v) \ll (\ell_1\ell_2 )^{\epsilon} (1+|s|+|w|)^{c}
\ee
for some constant $c>0$.
\end{prop}

\begin{proof}

By Proposition 5.1 in \cite{HH}, $D(s; n)$ has a meromorphic continuation to all $s\in \C$, with $\Re(s)> 1/2-A$. 

For $\Re(s) > 1/2-k/2$, $D(s; n) \ll n^{k/2}$, so the series $\cM(s, w, v)$ given in \eqref{e:M1exp}, 
converges absolutely for $\Re(v)>0$, $\Re(s)> 1/2-k/2$ and $\Re(w)>3/2$. 
	
For $s$ in the strip $1/2-A < \Re(s) < 1/2-k/2$, $D(s; n)$ can be expressed by the following 
absolutely convergent expression:
	\begin{multline*}
	D(s; n)
	=
	\frac{(4\pi)^k}{2\Gamma(s+k-1) n^{s-1/2}}
	\\
	\times
	\left\{ 
	\sum_j \overline{\rho_j(-n)}
	\frac{\Gamma\left(s-\frac{1}{2}-it_j\right) \Gamma\left(s-\frac{1}{2}+it_j\right) \Gamma(1-s)}
	{\Gamma\left(\frac{1}{2}-it_j\right)\Gamma\left(\frac{1}{2}+it_j\right) }
	\overline{\left<V, u_j\right>} 
	\right.
	\\
	+
	\sum_\cusp 
	\frac{1}{2\pi i} \int_{\overline{-C_\sigma} } 
	\frac{\vol 2\pi^{\frac{1}{2}-z} n^{-z} \rho_\cusp\left(\frac{1}{2}-z, -n\right)}
	{\Gamma\left(\frac{1}{2}-z\right) \zeta^*(1+2z)}
	\\
	\left.
	\times
	\frac{\Gamma\left(s-\frac{1}{2}-z\right) \Gamma\left(s-\frac{1}{2}+z\right)\Gamma(1-s)}
	{\Gamma\left(\frac{1}{2}-z\right) \Gamma\left(\frac{1}{2}+z\right)}
	\overline{\left<V,E_\cusp^*\left(*, \frac{1}{2}+z\right)\right>} \; dz
	\right\}
	\\
	+
	\Omega(s, n)
	\end{multline*}
	where
\begin{multline*}
	\Omega(s; n)
	=
	\frac{(4\pi)^k}{2\Gamma(s+k-1)}
	\sum_\cusp
	\sum_{\ell=0}^{\lfloor\frac{1}{2}-\Re(s)\rfloor}
	\frac{(-1)^\ell \Gamma(1-s) \Gamma(2s+\ell-1)}{\ell! \Gamma(s+\ell) \Gamma(1-s-\ell)}
	\\
	\times
	\left( 
	\frac{\pi^{1-s-\ell} n^{-2s-\ell+1} \vol \rho_\cusp(1-s-\ell, -n)}
	{\Gamma(1-s-\ell)\zeta^*(2s+2\ell)}
	\overline{\left<V, E_\cusp^*(*, s+\ell)\right>}
	\right.
	\\
	+ 
	\left. (1-\pmb{\delta}_{\sigma, \ell})
	\frac{\pi^{s+\ell} n^\ell \vol \rho_\cusp(s+\ell, -n)}
	{\Gamma(s+\ell) \zeta^*(2-2s-2\ell)}
	\overline{\left<V, E_\cusp^*(*, 1-s-\ell)\right>}
	\right)
	.
\end{multline*}
	Here, for $\cusp = 1/a$, $a\mid N$, 
	\begin{multline*}
	\rho_\cusp(s, n)
	=
	\left(\frac{a}{N}\right)^s 
	\frac{\sum_{d\mid n, \atop (d, N)=1} d^{1-2s}}
	{\zeta(2s)}
	\\
	\times
	\prod_{p\mid N} \left(1-p^{-2s}\right)^{-1} 
	\prod_{p\mid a, \atop p^k\|n, k\geq 0} 
	\frac{p^{-2s}}{1-p^{-2s+1}} 
	\left(p - p^{k(-2s+1)+1} -1 +p^{(k+1) (-2s+1)}\right)
	\end{multline*}
	for $n\neq 0$. 
	
Let
\begin{multline*}
	D_{\rm cusp}(s; n)
	:=
	\frac{(4\pi)^k}{2\Gamma(s+k-1) n^{s-1/2}}
	\\
	\times
	\sum_j \overline{\rho_j(-n)}
	\frac{\Gamma\left(s-\frac{1}{2}-it_j\right) \Gamma\left(s-\frac{1}{2}+it_j\right) \Gamma(1-s)}
	{\Gamma\left(\frac{1}{2}-it_j\right)\Gamma\left(\frac{1}{2}+it_j\right) }
	\overline{\left<V, u_j\right>} 
	.
\end{multline*}
Then we have
\begin{multline*}
	\cM^{(1)}_{\rm cusp}(s, w, v)
	:=
	\sum_{n\geq 1} \frac{D_{\rm cusp}(s; n) \tilde\sigma_{-v}(n) } {n^{w+\frac{k-1}{2}}}
	\\
	=
	\frac{(4\pi)^k}{2\Gamma(s+k-1)}
	\sum_j 
	L\left(s', \overline{u_j} \otimes E_1\left(*, (v+1)/2\right)\right)
	\\
	\times
	\frac{\Gamma\left(s-\frac{1}{2}-it_j\right) \Gamma\left(s-\frac{1}{2}+it_j\right) \Gamma(1-s)}
	{\Gamma\left(\frac{1}{2}-it_j\right)\Gamma\left(\frac{1}{2}+it_j\right) }
	\overline{\left<V, u_j\right>} 
	,
\end{multline*}
where $L(s', \overline{u_j}\otimes E_1(*, (v+1)/2))$ is defined in \eqref{e:RSuj}, 
and $s' = s+w+k/2-1$.
This is absolutely convergent when $\Re(v)>0$, $\Re(w) > -s/2 +5/4-k/4$.

Let
\begin{multline*}
	D_{\rm int}(s; n)
	:=
	\frac{(4\pi)^k}{2\Gamma(s+k-1) n^{s-1/2}}
	\sum_\cusp 
	\frac{1}{2\pi i} \int_{\overline{-C_\sigma} } 
	\frac{\vol 2\pi^{\frac{1}{2}-z} n^{-z} \rho_\cusp\left(\frac{1}{2}-z, -n\right)}
	{\Gamma\left(\frac{1}{2}-z\right) \zeta^*(1+2z)}
	\\
	\times
	\frac{\Gamma\left(s-\frac{1}{2}-z\right) \Gamma\left(s-\frac{1}{2}+z\right)\Gamma(1-s)}
	{\Gamma\left(\frac{1}{2}-z\right) \Gamma\left(\frac{1}{2}+z\right)}
	\overline{\left<V,E_\cusp^*\left(*, \frac{1}{2}+z\right)\right>} \; dz
	.
\end{multline*}
Our intent now is to understand the sum 
$$
	\sum_{n=1}^\infty \frac{D_{\rm int}(s; n) \tilde{\sigma}_{-v}(n)}{n^{w+\frac{k-1}{2}}}
	.
$$ 
To this end we consider first a sub-sum:
$$
	\sum_{n\geq 1} \frac{\rho_\cusp\left(\frac{1}{2}-z, -n\right) \tilde\sigma_{-v}(n)} {n^{s'+z}}
	,
$$
for each cusp $\cusp = 1/a$.
A computation gives us
$$
	\sum_{n\geq 1} 
	\frac{\rho_\cusp\left(\frac{1}{2}-z, -n\right) \tilde\sigma_{-v}(n)} {n^{s'+z}}
	=
	\frac{\zeta(s'+z) \zeta(s'-z) \zeta(s'+v+z) \zeta(s'+v-z)} {\zeta\left(1-2z\right) \zeta(2s'+v)}
	\tilde\cP_\cusp(s'; z, v)
$$
where 
	\begin{multline*}
	\tilde{\cP}_\cusp(s'; z, v)
	=
	\left(\frac{a}{N}\right)^{\frac{1}{2}-z} 
	\prod_{p\mid N}
	\left( (1-p^{-(1-2z)} )^{-1} (1-p^{-(2s'+v)})^{-1} \right)
	\\
	\times
	\left( \prod_{p\mid a} 
	p^{-(1-2z)}
	\left( (p-1) p^{-(s'+z)} -p^{-(3s'+v+z)+1} -p^{-(s'-z)} +p^{-2s'}(1+p^{-v})\right)\right)
	\\
	\times
	\left( \prod_{p\mid N, p\nmid a} (1-p^{-(s'-z)}) (1-p^{-(s'+v-z)})
	\right)
	\,.
\end{multline*}
As in \cite{HH}, we have
$$
	\cM^{(1)}_{\rm cont}(s, w, v)
	:=
	\sum_{n=1}^\infty \frac{D_{\rm int}(s; n) \tilde\sigma_{-v}(n)} {n^{w+\frac{k-1}{2}}}
	+
	\sum_{n=1}^\infty \frac{\Omega(s; n) \tilde\sigma_{-v}(n)} {n^{w+\frac{k-1}{2}}}
	=
	\cM^{(1)}_{\rm int}(s, w, v)
	+
	\Phi^{(1)}(s, w, v)
$$
where
$\cM_{\rm int}^{(1)}(s, w, v)$ is defined in \eqref{e:M1int}
and 
$\Phi^{(1)}(s, w, v)$ is defined in \eqref{e:Phi1}.
The integral in $\cM_{\rm int}^{(1)}(s, w, v)$ is absolutely convergent for $\Re(v)>0$ and $\Re(w) > s/2 +5/4 -k/4$. 

The poles of $\cM(s, w, v)$ in $R_1$ can occur at $s=1/2-r+it_j$ or at poles of $\Phi^{(1)}(s, w, v)$ in $R_1$. 
The function $\Phi^{(1)}(s, w, v)$ has potential polar lines at 
\begin{itemize} \label{itemize:phi1}
	\item $2s'+v=2$, 
	\item $s+1/2 -s' =-r$, 
	\item $s+1/2 -s' -v=-r$, 
	\item $s-3/2+s' = -r$, 
	\item $s-3/2+s' +v= -r$, 
	\item $2s+\ell-1 = -r$, for $0\leq \ell \leq \lfloor 1/2-\Re(s)\rfloor$,
\end{itemize}
where $r\ge 0$. 
In addition to the above polar lines, $\Phi^{(1)}(s, w, v)$ may have poles of the same order as the zeros of zeta functions in the denominator of $\Phi^{(1)}(s, w, v)$.  

The lines $s+1/2 -s' =-r$ and $s+1/2-s'-v=-r$ are in fact canceled. 
	The polar lines can exist at $s-3/2 +s'=-r$ and $s-3/2+s'+v=-r$. 

In particular, in the part of $R_1$ corresponding to $\Re \, s <1/2 -k/2$,
  $\Re \, s'  \ge1/2$,  and $\Re(s'+v) \geq1$ 
and at points at least $\epsilon$ away from the poles, 
	after applying (7.9) in \cite{HH}, 
	together with a trivial bound for the $L$-series $L(s'+v, u_j)$, 
	it follows that  $\cM(s,w,v)$ obeys the upper bound in \eqref{supper}.

\end{proof}

\subsection{The analytic continuation of $\cM(s, w, v)$ II}
We will now reverse the roles of $s$ and $w$ in  $\cM(s,w,v)$ and perform a very similar analysis in the region $R_2$, defined by
\begin{multline}\label{R2def}
	R_2
	:=
	\left\{(s, w, v)\;|\; \Re(v)>0, \; \Re(w) >1, \; \Re(s)> 1\right\}
	\\
	\cup
	\left\{(s, w, v) \; |\; \Re(v) >0,\; 1/2-B< \Re(w+v/2) < 1/2-k/2,\right.
	\\
	\left. 2\Re(s) + \Re(w) +\Re(v)/4 > -3k/4 + 5/2\right\}
	\\
	\cup
	\left\{(s, w, v)\;|\; \Re(v)>0,\; 1/2-k/2 \leq \Re(w+v/2)  < 2, \; \Re(s) > 3/2+\theta\right\}
	.
\end{multline}  

\begin{prop}\label{prop:M2}

The triple Dirichlet series $\cM(s, w, v)$ has a meromorphic continuation to $R_2$. 
	For $\Re(v) >0$, $\Re(w+v/2) < 1/2-k/2$ and $2\Re(s) + \Re(w) + \Re(v)/4 > -3k/4 +5/2$, 
	it has the following spectral expansion:
\begin{multline}\label{e:M2}
	\cM(s,w,v)+\cM_3(s, w, v)
	=
	\ell_1^{(k-1)/2}\ell_2^{1-s- w -k/2-v/2} 2^{k}\pi^{k/2-1/2-v/2} 
	\\
	\times
	\frac{\G((k+1+v)/2)\G(w+v/2) \G(1-w-v/2)} {\G(w+(k-1)/2)\G(w+(k-1)/2+v)}
	\\ 
	\times
	\left(\cM^{(2)}_{\rm cusp} (s,w,v) +\cM^{(2)}_{\rm cont} (s,w,v) \right) 
	,
\end{multline}
where
\begin{multline}\label{e:M2cusp}
	\cM^{(2)}_{\rm cusp} (s,w,v)
	=
	\sum_j
	L_{\ell_2}(s'+v/2, \bar g \otimes \bar u_j)\overline{\<U,u_j\>} 
	\\ 
	\times
	\frac{\G(w+v/2-1/2 +it_j)\G(w+v/2-1/2 - it_j)}{\G(1/2+it_j)\G(1/2-it_j)}
	,
\end{multline}
\be\label{e:M2cont}
	\cM^{(2)}_{\rm cont}(s, w, v)
	=
	\cM^{(2)}_{\rm int}(s, w, v)
	+
	\Phi^{(2)}(s, w, v)
	.
\ee
and
\begin{multline}\label{e:M3}
	\cM_3(s,w,v)
	=
	\ell_1^{(k-1)/2}\ell_2^{1/2-s-k/2} 
	\frac{\G(1-w-v/2)\G(w+v/2)}{\G((1-v+k)/2)\G((1+v-k)/2} 
	\\ 
	\times 
	\sum_{{m_2\ge 1} \atop {1 \le m'  <\ell_2m_2}}
	\frac{ \overline{b(m_2)} a((\ell_2 m_2 - m')/\ell_1) \tilde\sigma_{-v}(m') }
	{( m')^{w+ (k-1)/2}( m_2)^{s+k-1}}
	.
\end{multline}
Here
\begin{multline}\label{e:M2int}
	\cM^{(2)}_{\rm int}(s, w, v)
	=
	\sum_\cusp \frac{1}{2\pi i} \int_{\overline{-C_\sigma}} 
	\frac{\vol \pi^{\frac{1}{2}-z} 
	L_{\ell_2} \left(s'+\frac{v}{2}, \bar g \otimes E_\cusp\left(*, 1/2-z\right)\right)}
	{\Gamma(1/2 - z)} 
	\\
	\times 
	\frac{\G(w + v/2-1/2+z)\G(w + v/2-1/2-z)}{\Gamma\left(\frac{1}{2}+z\right)\Gamma\left(\frac{1}{2}-z\right) \zeta^*(1+2z)}
	\overline{\<U,E^*_\cusp (*,1/2 + z)\>}dz
\end{multline}
and
\begin{multline}\label{e:Phi2}
	\Phi^{(2)}(s, w, v)
	=
	\sum_\cusp 
	\sum_{b=0}^{\lfloor\frac{1}{2}-\sigma\rfloor}
	\vol 
	\frac{(-1)^b \sqrt\pi 2^{\frac{1}{2}-w-\frac{v}{2}} \Gamma\left(1-w-\frac{v}{2}\right) 
	\Gamma\left(2w+v+b-1\right) }
	{b! \Gamma\left(w-\frac{v}{2}+b\right) \Gamma\left(1-w-\frac{v}{2}-b\right)}
	\\
	\times
	\left( \frac{\pi^{1-(w+v/2)-b} 
	L_{\ell_2} \left(s'+\frac{v}{2}, \bar g \otimes E_\cusp\left(*, w+v/2+b\right)\right)
	\overline{\left<U, E^*_\cusp \left(*, w+v/2+b \right)\right>} }
	{\Gamma(1-(w+v/2)-b) \zeta^*(2w+v+2b)} 
	\right.
	\\
	\left.
	+ 
	(1-\pmb{\delta}_{\sigma, b}) 
	\frac{\pi^{(w+v/2)+b} 
	L_{\ell_2} \left(s'+\frac{v}{2}, \bar g \otimes E_\cusp\left(*, 1-w-v/2-b\right)\right)
	\overline{\left<U, E^*_\cusp\left(*, 1-w-v/2-b\right)\right>} }
	{\Gamma\left(w+\frac{v}{2}+b\right) \zeta^*(2-2w-v-2b)} 
	\right)
	.
\end{multline}	
Then $\cM(s, w, v)- \Phi^{(2)}(s, w, v)$ has poles in $R_2$ only at 
\begin{itemize}
	\item $w+v/2=-r$, for $r\geq 0$ 
	\item $w+v/2=1/2+it_j-r$, for $r\geq 0$, $j\geq 1$
	.
\end{itemize}

\end{prop}

Before proving Proposition \ref{prop:M2}, we will find it convenient to state and prove the following proposition, which describes the meromorphic continuation of the Dirichlet series $\cM_3$.
\begin{prop}\label{prop:dsums}
The Dirichlet series $\cM_3(s,w,v)$ is absolutely convergent for $\Re \, s >1$ 
	and has a meromorphic continuation to $R_2$. 
	For $\Re(s')>1$, this continuation is given by 
\begin{multline}\label{e:M3spec}     
	\cM_3(s,w,v)
	=
	\ell_1^{(k-1)/2}\ell_2^{1/2-s-k/2} 
	\frac{\G(1-w-v/2)\G(w+v/2)}{\G((1-v+k)/2)\G((1+v-k)/2)} 
	\\
	\times
	\frac{2^{2k-1} \pi^k }{\G(s+k-1)}
	\left\{
	\sum_j \frac{\G(s-1/2 + it_j)\G(s-1/2 - it_j)\< u_j,V\>} 
	{ \G(s)}
	\right.
	\\
	\times
	L(s', \overline{u_j}\otimes E_1(*(1+v)/2))
	\\
	+
	\frac{1}{4\pi} \sum_\cusp \int_{-\infty}^\infty 
	\frac{2\vol \G(s-1/2 + it)\G(s-1/2 - it) \left<E_\cusp^*(*, 1/2+it), V\right>}
	{ \G(s)}
	\\
	\left.
	\times
	\frac{\zeta(s'+it) \zeta(s'-it) \zeta(s'+v+it) \zeta(s'+v-it) 
	\tilde\cP_\cusp(s'; it, v)}
	{\zeta^*(1-2it)  \zeta^*(1+2it) \zeta(2s'+v)}
	\; dt
	\right\}
%
\end{multline}
where
$V(z) = \overline{f(\ell_1 z) } g(\ell_2z ) y^k$.
Here the Rankin-Selberg $L$-function $L(s', \overline{u_j} \otimes E_1(*, (1+v)/2))$ is defined in \eqref{e:RSuj}.  
The poles in the region $R_2$, defined in \eqref{R2def}, are at $w+v/2 =-r$, $r \ge 0$. 
Moreover, $\cM_3(s, w, v)$ may have poles of the same order as the zeros of zeta function $\zeta(2s'+v)$. 
\end{prop}
\begin{proof}
The inner sum over $m_2$ is essentially an inner product of $f$ and $g$ with a standard Poincare series.   
	In particular, for $m'\geq 1$, if
$$
	P(z, s; m') 
	= 
	\sum_{\g\in\Gi\bk\G}(\Im \, \g z)^s e^{2 \pi i m' \g z},
$$
then 
$$
	\frac{\G(s+k-1) }{\vol}
	\sum_{m_2 \ge 1}\frac{a((\ell_2m_2 - m')/\ell_1)
	\overline{b (m_2)}}
	{(4 \pi\ell_2 m_2)^{s+k-1}} 
	= 
	\<P(*, s; m'),V \>.
$$
The inner product can also be expressed in terms of a spectral expansion
\begin{multline*}
	\<P(*, s; m'), V\>
	= 
	\sum_j \<P(*, s;m'),u_j \>\< u_j,V \>
	\\
 	+
	\frac{1}{4\pi} \sum_\cusp
	\int_{-\infty}^\infty  \<P(*, s; m'),E_\cusp (*; 1/2 +it) \>
	\<E_\cusp (*; 1/2 +it),V \>
	\; dt.
 \end{multline*}
A simple computation shows that
\begin{multline*}
 	\<P(*, s; m'),u_j \> 
	= 
	\frac{1}{\vol}
	\int_0^\infty 
	\frac{y^{s-1} \overline{\rho_j(m')}e^{-y}K_{it_j}(y)}
	{(2 \pi m')^{s -1/2}} \; \frac{dy}{y}
	\\
 	=
	\frac{ \overline{\rho_j(m')} \sqrt{\pi} 2^{1/2-s}\G(s-1/2 + it_j)\G(s-1/2 - it_j)}
	{\vol (2 \pi  m')^{s -1/2}\G(s)}
\end{multline*}
and that
\begin{multline*}
	\left<P(*, s; m'), E_\cusp(*; 1/2+it)\right>
	\\
	=
	\frac{\vol 2\pi^{\frac{1}{2}-it} \overline{\rho_\cusp (1/2+it, m')} |m'|^{-it}}
	{\vol \Gamma(1/2-it)}
	\frac{\sqrt{\pi} 2^{1/2-s}\G(s-1/2 + it)\G(s-1/2 - it)}
	{(2 \pi  m')^{s -1/2}\G(s)}
	.
\end{multline*}

Substituting these evaluations into \eqref{e:M3}, we have
\begin{multline*}
	\ell_1^{(k-1)/2}\ell_2^{1/2-s-k/2} 
	\frac{\G(1-w-v/2)\G(w+v/2)}{\G((1-v+k)/2)\G((1+v-k)/2} 
	\\
	\times
	\sum_{m' \ge 1}
	\frac{\tilde\sigma_{-v}(m')}{(m')^{w +(k-1)/2}}
	\sum_{m_2 \ge 1} \frac{a((\ell_2m_2 -\ell m')/\ell_1)\bar b (m_2)}{(\ell_2 m_2)^{s+k-1}}
	\\
	= 
	\ell_1^{(k-1)/2}\ell_2^{1/2-s-k/2} 
	\frac{\G(1-w-v/2)\G(w+v/2)}{\G((1-v+k)/2)\G((1+v-k)/2} 
	\frac{2^{2k-1} \pi^k }{\G(s+k-1)}
	\\
	\times
	\left\{
	\sum_j \frac{\G(s-1/2 + it_j)\G(s-1/2 - it_j)\< u_j,\bar fg y^k \>} 
	{ \G(s)}
	\right.
	L(s', \overline{u_j} \otimes E_1(*, (1+v)/2))
	\\
	+
	\frac{1}{4\pi} \sum_\cusp \int_{-\infty}^\infty 
	\frac{2\vol \G(s-1/2 + it)\G(s-1/2 - it) \left<E_\cusp^*(*, 1/2+it), \bar fgy^k\right>}
	{ \G(s)}
	\\
	\left.
	\times
	\frac{\zeta(s'+it) \zeta(s'-it) \zeta(s'+v+it) \zeta(s'+v-it) 
	\tilde\cP_\cusp(s'; it, v)}
	{\zeta^*(1-2it)  \zeta^*(1+2it) \zeta(2s'+v)}
	\; dt
	\right\}
\end{multline*}
where $\tilde{\cP}_\cusp(s'; it, v)$ is defined in \eqref{e:tildePcusp}.

The cuspidal part may contribute simple poles at $w+v/2 = -r$, $r\ge0$.  
The poles of the continuous contribution to the $m' \ge 1$ piece are analyzed in an identical way to those of the continuous contribution in the region $\Re \, s <1$, yielding the following:
\begin{multline*}
	\frac{1}{4\pi} \sum_\cusp \int_{-\infty}^\infty 
	\frac{2\G(s-1/2 + it)\G(s-1/2 - it) \left<E_\cusp^*(*, 1/2+it), \bar fgy^k\right>}
	{ \G(s)}
	\\
	\times
	\frac{\zeta(s'+it) \zeta(s'-it) \zeta(s'+v+it) \zeta(s'+v-it) 
	\tilde\cP_\cusp(s'; it, v)}
	{\zeta^*(1-2it)  \zeta^*(1+2it) \zeta(2s'+v)}
	\; dt
	\\
	=
	\sum_\cusp \frac{1}{2\pi i }  \int_{(0)}
	\frac{\G(s-1/2 + z)\G(s-1/2 - z) \left<E_\cusp^*(*, 1/2+z), \bar fgy^k\right>}
	{ \G(s)}
	\\
	\times
	\frac{\zeta(s'+z) \zeta(s'-z) \zeta(s'+v+z) \zeta(s'+v-z) 
	\tilde\cP_\cusp(s'; z, v)}
	{\zeta^*(1-2z)  \zeta^*(1+2z) \zeta(2s'+v)}
	\; dz
	\\
	+
	\sum_\cusp 
	\frac{1}{\zeta(2s'+v)}
	\left\{
	\frac{\G(s-s'+1/2)\G(s+s'-3/2) \left<E_\cusp^*(*, 3/2-s'), \bar fgy^k\right>}
	{ \G(s)}
	\right.
	\\
	\times
	\frac{\zeta(2s'-1) \zeta(v+1) \zeta(2s'+v-1) }
	{\zeta^*(-1+2s')  \zeta^*(3-2s')}
	\left(
	\tilde\cP_\cusp(s'; 1-s', v)- \tilde\cP_\cusp(s'; -1+s', v)
	\right)
	\\
	+
	\frac{\G(s-s'-v+1/2)\G(s+s'+v-3/2) \left<E_\cusp^*(*, 3/2-s'-v), \bar fgy^k\right>}
	{ \G(s)}
	\\
	\times
	\frac{\zeta(1-v)) \zeta(2s'+v-1))  \zeta(2s'+2v-1) }
	{\zeta^*(-1+2s'+2v))  \zeta^*(3-2s'-2v))}
	\left(
	\tilde\cP_\cusp(s'; 1-s'-v, v)
	-
	\tilde\cP_\cusp(s'; -1+s'+v, v)
	\right)
	\bigg\}
	.
\end{multline*}
The only potential polar lines in $R_2$ are at $2s'+v=2$, but the two polar contributions on this line cancel each other. 
The poles of the cuspidal part at $s = 1/2 + it_j -r$  lie outside of $R_2$.
\end{proof}

\begin{proof}[Proof of Proposition \ref{prop:M2}]
Recall that in the region $\Re(v)\ge 0$ and $\Re(w)>1$, the series 
	$D'(w, v; \ell_2m_2) \ll m_2^{(k-1)/2}$, and so  the expression for 
$\cM(s,w,v)$ given in \eqref{e:Mswv}, 
\be\label{e:Mswv2}
	\cM(s,w,v)
	=
	(\ell_1\ell_2)^{(k-1)/2}\sum_{{m_2\ge 1}}
	\frac{ \overline{b(m_2)} D'(w,v; \ell_2 m_2) }{(\ell_2 m_2)^{s+k-1}}
	,
\ee
converges absolutely for $\Re(s) > 1$. 

For $c< \Re\left(w+v/2\right) <2$, with $c< 1/2-k/2$, 
	$D'(w,v; \ell_2 m_2)+D_{\rm finite}'(w, v; \ell_2m_2)$ satisfies the upper bound of \eqref{e:I0upper2}.
	It follows that $\cM(s,w,v)$ converges for $\Re(s) > 3/2+\theta$ 
	and $1/2-k/2\leq \Re \left( w+v/2 \right) <2$.

Also, applying the upper bound \eqref{e:I0upper}, it follows that for 
$ \Re \left( w+v/2 \right) <1/2-k/2$, $\Re(s') \ge 1/2$, $\Re (v) >0$,
\be\label{wupper}
	\cM(s,w,v) + \cM_3(s, w, v)
	\ll 
	\frac{ \ell_2^{-s'-v/2 + \theta}
	(1+|v|)^{k/2}}{(1+|w|)^{w + k/2 -1}(1+|w+v|)^{w+k/2 -1+v}}.
\ee
Here $s'=s+w + k/2-1$.  
Substituting \eqref{e:D'exp} into \eqref{e:Mswv2} gives us
\begin{multline*}
	\cM(s,w,v) + \cM_3(s, w, v)
	= 
	(\ell_1\ell_2)^{(k-1)/2}\sum_{{m_2\ge 1}}
	\frac{ \overline{b(m_2)} }{(\ell_2 m_2)^{s+k-1}} 
	\\ 
	\times
	 \pi^{(k-v-1)/2}2^k (\ell_2 m_2)^{1/2 -w -v/2} 
	\frac{\G(1-w-v/2)\G(w+v/2)\G((k+1+v)/2)}{\G(w +(k-1)/2)\G(w +(k-1)/2 +v)}
	\\
	\times
	\left(S'_{\rm{cusp}}(w, v;\ell_2 m_2)+S'_{\rm{cont}}(w, v;\ell_2 m_2) \right) 
	,
\end{multline*}
with $S'_{\rm{cusp}}(w, v; \ell_2m_2)$, $S'_{\rm{cont}}(w, v; \ell_2m_2)$ given by \eqref{e:S'cusp} and \eqref{e:S'cont}.

The sum over $m_2$ can be brought inside,   
leading to
\begin{multline*}
	\cM(s,w,v)+\cM_3(s, w, v)
	=
	\ell_1^{(k-1)/2}\ell_2^{1-s- w -k/2-v/2} 2^{k}\pi^{k/2-1/2-v/2} 
	\\
	\times
	\frac{\G((k+1+v)/2)\G(w+v/2) \G(1-w-v/2)} {\G(w+(k-1)/2)\G(w+(k-1)/2+v)}
	\left(\cM^{(2)}_{\rm cusp} (s,w,v) +\cM^{(2)}_{\rm cont} (s,w,v) \right) 
	,
\end{multline*}
where
$$
	\cM^{(2)}_{\rm cusp} (s,w,v)
	=
	\sum_j
	L_{\ell_2}(s'+v/2, \bar g \otimes \bar u_j)\overline{\<U,u_j\>} 
	\frac{\G(w+v/2-1/2 +it_j)\G(w+v/2-1/2 - it_j)}{\G(1/2+it_j)\G(1/2-it_j)}
$$
and
$$
	\cM^{(2)}_{\rm cont}(s, w, v)
	=
	\cM^{(2)}_{\rm int}(s, w, v)
	+
	\Phi^{(2)}(s, w, v)
	.
$$
Here $\cM^{(2)}_{\rm int}(s, w, v)$ is defined in \eqref{e:M2int} 
and $\Phi^{(2)}(s, w, v)$ is defined in \eqref{e:Phi2}.

In the case of $\cM^{(2)}_{\rm cusp}(s, w, v)$, 
	the series $L_{\ell_2}(s'+v/2, \bar g \otimes \bar u_j)$ is defined by
$$
	L_{\ell_2}(s'+v/2, \bar g \otimes \bar u_j)
	= 
	\sum_{m_2 \ge 1}\frac{\overline{B(m_2)} \overline{\rho_j(-\ell_2 m_2)}}{m_2^{s'+v/2}}
$$
and equals the Rankin-Selberg convolution of $\bar g$ with $\bar u_j$, 
	up to the Euler factors corresponding to primes dividing $\ell_2$.  
The series   $L_{\ell_2}$ appearing in $\cM^{(2)}_{\rm cont}$ 
	is the Rankin-Selberg convolution of $\bar g$ with the Eisenstein series at cusps:
$$
	L_{\ell_2} \left(s'+\frac{v}{2}, \bar g \otimes E_\cusp\left(*, 1/2-z\right)\right)
	=
	\sum_{m_2\geq 1} \frac{\overline{B(m_2)} (\ell_2m_2)^{-z} \rho_\cusp\left(\frac{1}{2}-z, -\ell_2m_2\right)}
	{m_2^{s'+\frac{v}{2}}}
	.
$$

For $\Re(s'+v/2)<1$, we have
$$
	\left| L_{\ell_2} (s'+v/2, \bar g\otimes\overline{u_j}) \left<U, u_j\right>\right| 
	\ll
	|t_j|^{2-4\Re(s'+v/2) + \frac{k+\Re(v)}{2}-1}
	.
$$
For $\Re\, v \ge 0$, $\Re(w+v/2)< 1/2-k/2$, $\Re(w) > -2\Re(s) -\Re(v)/4 -3k/4 +5/2$, 
	the series and integral $\cM^{(2)}_{\rm cusp}$ and $\cM^{(2)}_{\rm cont}$ converge 
	absolutely. 

The only possible poles of  $\cM^{(2)}_{\rm cont}(s,w,v)$ are  
	the poles of $\Phi^{(2)}(s, w, v)$. 
We have now obtained the meromorphic continuation of $\cM(s, w, v)$ to $R_2$ 
	and will proceed to identify the poles of $\cM(s,w,v)$ in this region.  
	By the previous discussion, the expression for $\cM$ given in \eqref{e:M2} 
	converges absolutely in the part of $R_2$ with $\Re \left(w+v/2\right) <1/2-k/2$, 
	when a distance of $\epsilon$ from the polar points.   
	There are no other poles in $R_2$ with $\Re(s), \Re(w)>1$, and $\Re(v)\geq 0$, 
	as the series expansion in \eqref{e:Mswv2} converges absolutely there.  
Referring to \eqref{e:M2} - \eqref{e:M2cont} and Proposition~\ref{prop:dsums}, we see that these poles are at the following points:
\begin{itemize}
\item
$w+v/2 = 1/2 +it_j -r $ (from $\cM^{(2)}_{({\rm cusp})}$)
 \item
 The points $w + v/2 = -r$, with $B >r \ge 0$ (from $\Gamma(w+v/2)$).
\end{itemize}

Writing 
$$
	e_{j,r} := \Res_{w+v/2= 1/2 +it_j -r } \cM(s,w,v),
$$
we find that
\begin{multline}\label{ejrdef}
	e_{j,r}
	=
	\frac{\G((k+1+v)/2)\G(1/2 +it_j -r )\G(1/2 -it_j +r )}{\G(1/2 +it_j-r-v/2)\G(1/2 +it_j-r+v/2)}
	\\ 
	\times 
	\frac{(-1)^r\G(2it_j-r)\overline{\<U,u_j\> }
	L_{\ell_2}(s+k/2-1/2 -r  +it_j,\bar g \otimes \bar u_j)}{r!\G(1/2+it_j)\G(1/2-it_j)}.
 \end{multline}
\end{proof}

\subsection{Proof of Theorem \ref{t:triple}}

Our goal is to now relate  $\cM(s,w,v)$ to an entire function on $R_1$ by  subtracting off the cuspidal polar contribution and then multiplying to clear the polar lines.   
Specifically, as mentioned in Proposition \ref{prop:M1}, 
	the poles of $\cM(s,w,v)$ arising from the cuspidal contribution occur in the same places as the poles of $D(s;n)$, i.e at the points $s = 1/2 + it_j -r$, for $r \ge 0$.
Setting
$$
	R^{(2)}(1/2 -r + it_j, w ,v) = \Res_{s= 1/2 -r + it_j} \cM(s,w,v)
,	
$$
we compute
\begin{multline}\label{Pres}
	R^{(2)}(1/2 -r + it_j, w ,v)
	\\ 
	=
	(\ell_1\ell_2)^{\frac{k-1}{2}}
	\frac{2^{2k-1}\pi^{k}L\left(s', \overline{u_j} \otimes E_1\left(*, (v+1)/2\right)\right)}
	{\Gamma(k-1/2 -r + it_j)}
	\overline{\<V,u_j\>} 
 	\frac{(-1)^r\Gamma(1/2 - it_j +r) \Gamma ( 2 it_j - r) }
	{r ! \Gamma (1/2 + it_j)	\Gamma(1/2 - it_j)}
\end{multline}
and let
\be\label{P2def}
	P^{(2)}(s,w,v) 
	= 
	\sum_{r <A}\sum_j 
	\frac{R^{(2)}(1/2 -r + it_j, w ,v) e^{(1/2 -r + it_j)^2}}{s-(1/2 -r + it_j)}
	.
\ee
The absolutely convergent series $P^{(2)}(s,w,v)$ now has poles and residues at those poles in the region $R_1$ that are identical to the poles of $\cM(s,w,v) e^{(1/2 -r + it_j)^2}$.  
Thus the function
$$
	\cM(s,w,v) e^{s^{2}} - P^{(2)}(s,w,v)
$$ 
has only potential polar lines arising from $\Phi^{(1)}(s, w, v)$. 
Define $P^{(3)}(s, w, v)$ to be the product canceling polar lines and zeta functions of $\Phi^{(1)}(s, w, v)$, discussed above, along with the lines canceling the poles of the zeta function. 
	Then
$$
	P^{(3)}(s, w, v) \left(\cM(s, w, v) e^{s^2} - P^{(2)}(s, w, v)\right)
$$
is holomorphic in $R_1$.

By the discussion in the proof of Proposition \ref{prop:M2}, 
	the expression for $\cM(s, w, v)$ given in \eqref{e:M2} 
	converges absolutely in the part of $R_2$ with $\Re \left(w+v/2\right) <1/2-k/2$, 
	when a distance of $\epsilon$ from the polar points.   
	There are no other poles in $R_2$ with $\Re(s), \Re(w)>1$, and $\Re(v)\geq 0$, 
	as the series expansion in \eqref{e:Mswv2} converges absolutely there.  
	Let $\kappa$ run over the poles of $\cM(s,w,v)$ in $R_2$ and set
\be\label{Pres2}
	R^{(3)}(s, \kappa ,v) 
	:= 
	\Res_{w= \kappa} \cM(s,w,v)
	.
\ee
Referring to \eqref{e:M2} - \eqref{e:M2cont} and Proposition~\ref{prop:dsums}, we see that these poles are at the following points:
\begin{itemize}
\item
$w+v/2 = 1/2 +it_j -r $ (from $\cM^{(2)}_{({\rm cusp})}$)
 \item
 The points $w + v/2 = -r$, with $B >r \ge 0$ (from $\Gamma(w+v/2)$).
\end{itemize}

Let
\be\label{P4def}
	P^{(4)}(s,w,v) 
	= 
	\sum_{r<B}\sum_{\kappa=1/2 +it_j -r -v/2} \frac{R^{(3)}(s,\kappa ,v) e^{\kappa^2}}{w-\kappa},
\ee
	(see \eqref{ejrdef} for a description of $R^{(3)}(s, 1/2 +it_j -r -v/2 ,v)$).
	This  series  has poles and residues at  poles in the region $R_2$ 
	that are identical to those of $\cM(s,w,v)$ and converges absolutely.  
	Thus the function
$$
	\cM(s,w,v)e^{w^2} - P^{(4)}(s,w,v)
$$ 
has only potential polar lines arising from $\Phi^{(2)}(s, w, v)$. 
	Define $P^{(5)}(s, w, v)$ to be the product canceling polar lines and zeta functions of $\Phi^{(2)}(s, w, v)$, along with the lines canceling the poles of the zeta function, multiplied by $\prod_{0\le r <B}(w+v/2+r)$.
	Then 
$$
	P^{(5)}(s, w, v) \left(\cM (s, w, v)e^{w^2} - P^{(4)}(s, w, v)\right)
$$
is holomorphic in $R_2$. 

We have shown that the function
$$
	P^{(3)}(s, w, v) \left( \cM(s,w,v)e^{s^2} - P^{(2)}(s,w,v))\right)
 $$
	is holomorphic in the region $R_1$, where $P^{(2)}(s, w, v)$ is defined in \eqref{P2def}. 
	
The function $P^{(4)}(s,w,v)$, defined in \eqref{P4def}, has no poles in $R_1$ and the function $P^{(2)}(s,w,v)$ has no poles in $R_2$.   
	Thus 
\be\label{R4}
	P^{(3)}(s, w, v) P^{(5)}(s, w, v)
	\left( \cM(s,w,v)e^{s^2+w^2}-P^{(2)}(s,w,v)e^{w^2} - P^{(4)}(s,w,v)e^{s^2}\right)
\ee
is holomorphic in $R_1 \cup R_2$.     
	As  $A,B \rightarrow \infty$ the convex hull of  
$R_1 \cup R_2$ includes all of $\C^3$ such that $\Re \, v > 0$.     
Applying the theorem of B\"ochner, we have completed the proof of the meromorphic continuation of 
$\cM(s,w,v) $.  
Summarizing the above, we have proved Theorem \ref{t:triple}.

\section{An average special value estimate}\label{s:svest}

Define
$$
	\tilde\cM(s,w,v)
	:=
	\cM(s, w, v)- \Phi^{(1)}(s, w, v)
	,
$$ 
where $\Phi^{(1)}(s, w, v)$ is given in \eqref{e:Phi1}.
In this section we will estimate an average of the special value $\tilde\cM(1-k/2,1/2,v)$ 
by writing it as an integral around a boundary where the spectral expansions for $\cM(s,w,v)$ converge.  
The integral transform of $\tilde M(1-k/2, 1/2, v)$ is of particular interest 
because it represents a smoothed average of special values $\tilde Z_q(1-k/2, 1/2)$ over an interval of length $Y/y$, centered at $Y$:
$$
	\sum_{q\geq 1, \atop (q, N_0)=1} \tilde Z_q(1-k/2, 1/2; f, g; \ell_1, \ell_2) e^{-\frac{y^2 \left(\log \frac{Y}{q}\right)^2}{4\pi}} 
	=
	\frac{1}{2\pi i} \int_{(2)} 
	\tilde\cM(1-k/2, 1/2,v; \ell_1, \ell_2)
	\frac{e^{\pi v^2/y^2}}{y} Y^v\; dv
	.
$$

We will prove the following proposition via an estimate for the average of the special value $\tilde\cM(1-k/2, 1/2, v)$. 
\begin{prop}\label{prop:Mest}
	For $Y, y \gg 1$ and square free $\ell_1, \ell_2 \geq 1$, with $(\ell_1\ell_2, N_0)=1$, we have
\be\label{e:Mest_1}
	\sum_{q\geq 1, \atop (q, N_0)=1} \tilde Z_q(1-k/2, 1/2; f, g; \ell_1, \ell_2) e^{-\frac{y^2 \left(\log \frac{Y}{q}\right)^2}{4\pi}} 
	\ll 
	Y^{1/2} (\ell_1\ell_2)^\epsilon
	,
\ee
for any $\epsilon>0$.

Moreover, we have
\be\label{e:Mest_2}
	\sum_{q\geq 1, \atop (q, N_0)=1} \tilde Z_q(1-k/2, 1/2; f, g; \ell_1, \ell_2) e^{-\frac{y^2 \left(\log \frac{Y}{q}\right)^2}{4\pi}} 
	=
	\cO( y^{\epsilon} (\ell_1\ell_2)^{1/4+\epsilon})
	, 
\ee
for any $\epsilon>0$. 
\end{prop}

\begin{proof}
The meromorphic continuation of $\cM(s, w, v)$ has established that $\cM(s, w, v)$ has at most polynomial growth in vertical strips in the variables $s$, $w$ and $v$. Consequently, the integral converges absolutely. 

For $\nu >0$, let $v = \nu + i\gamma$.
Recall that by the description of polar planes in Proposition~\ref{prop:M1} and Theorem \ref{t:triple}, $\tilde\cM(s, w, v)$ 
has potential poles only at $s=1/2-r+it_j$, and at $w+v/2 = 1/2+it_j -r$, and $w+v/2= -r$, for $r\geq 0$. 
It follows that for $1/2-k/2-\epsilon \le \Re \, s \le 1+3\epsilon/2$, 
$$
	\cM_P(s,v):=\prod_{0 \le r <k/2}(3/2-k/2-s+v/2+r)\tilde\cM(s, 3/2-k/2-s,v)
$$
has its poles at $s=1/2-r+it_j$ and at $s = 1+r-k/2+v/2-it_j$.  The residues at these poles are described in \eqref{Pres} and \eqref{ejrdef}.
Note that 
$$
\hat\cM(s,v)=\frac{\cM_P(s,v)}{\prod_{0 \le r <k/2}(1/2+v/2+r)}
$$
has poles in the same locations, and satisfies 
$$
\hat\cM(1-k/2,v) = \tilde\cM(1-k/2, 1/2,v).
$$

Set
$$
	I_1(v) 
	:= 
	\frac{1}{2 \pi i} \int_{(1+\nu/2+ \epsilon)} \frac{\hat\cM(s, v)e^{s^2}}{s-1+k/2} \; ds
	.
$$
As $\hat\cM(s,v)$ has at most polynomial growth in $s$ we may move the line of integration, obtaining 
$$
	I_1(v) 
	=
	\hat \cM(1-k/2,v)e^{(1-k/2)^2}+ T_1(v)+T_2(v) + I_2(v),
$$
where
$$
	I _2(v)
	= 
	\frac{1}{2 \pi i} \int_{(1/2-k/2-\epsilon)} \frac{\hat\cM(s,v)e^{s^2}}{s-1+k/2}
	\; ds
	,
$$
$$
	T_1(v)
	=
	\sum_{r, t_j}\frac{\Res_{s=1/2-r+it_j}\hat\cM(s,v)e^{(1/2-r+it_j)^2}}{-1/2-r+k/2+it_j}
	,
$$
and
$$
	T_2(v)
	=
	\sum_{r, t_j}\frac{\Res_{s= 1+r-k/2+v/2-it_j}\hat\cM(s,v)e^{( 1+r-k/2+v/2-it_j)^2}}{ r+v/2-it_j}
	.
$$
To estimate
$$
	\frac{1}{2\pi i} \int_{(2)} 
	\tilde\cM(1-k/2, 1/2,v; \ell_1, \ell_2)
	\frac{e^{\pi v^2/y^2}}{y} \left(\frac{Q}{\ell}\right)^v\; dv,
$$
it will suffice to estimate
$$
\frac{1}{2\pi i} \int_{(2)} I_1(v)
	\frac{e^{\pi v^2/y^2}}{y} \left(\frac{Q}{\ell}\right)^v\; dv,
	\quad
	\frac{1}{2\pi i} \int_{(2)} I_2(v)
	\frac{e^{\pi v^2/y^2}}{y} \left(\frac{Q}{\ell}\right)^v\; dv,
$$
\be\label{e:T12}
\frac{1}{2\pi i} \int_{(2)} T_1(v)
	\frac{e^{\pi v^2/y^2}}{y} \left(\frac{Q}{\ell}\right)^v\; dv, \quad \text{and} \quad
\frac{1}{2\pi i} \int_{(2)} T_2(v)
	\frac{e^{\pi v^2/y^2}}{y} \left(\frac{Q}{\ell}\right)^v\; dv.
\ee
To accomplish this, the following will be useful:
\begin{lemma}\label{lindav}
Let $F$ be a function of $v$, analytic  for $0\le\Re \, v <3$, such that $F(it) \ll  (1+|t|)^\epsilon$ for $t$ real.  Then
\be\label{e:Lest_1}
\int_{(2)} \frac{L(1/2 +i\gamma+v,\overline{\tilde u_j}) e^{\pi v^2/y^2} F(v)X^vdv}{y}
	\ll 
	(\ell_1\ell_2)^{\epsilon'}X^{1/2}
\ee
for any $\epsilon'>0$, 
where $\tilde u_j$ is a new form, normalized so the first Fourier coefficient equals 1, such that all but finitely many Hecke eigenvalues agree with those of $u_j$. 

Moreover, we have
\be\label{e:Lest_2}
	\int_{(2)} \frac{L(1/2 +i\gamma+v,\overline{\tilde u_j}) e^{\pi v^2/y^2} F(v)X^vdv}{y}
	\ll 
	(\ell_1\ell_2)^{1/4+\epsilon'} y^\epsilon(1 + |\gamma|)^{1/2+ \epsilon}(1+|t_j|)^{1/2+ \epsilon}
\ee
for any $\epsilon'>0$.
\end{lemma}
\begin{proof}
The first part of the lemma, \eqref{e:Lest_1}, follows immediately upon moving the $v$ line of integration to $\Re(v)=1/2$. 

To prove the second part of the lemma,  we will integrate by parts.  
Set 
$$
	F_1 (v)= e^{\pi v^2/y^2}
$$
and
$$
	F_2(v) = \int_0^v \frac{L(1/2 +i\gamma+v' ,\overline{\tilde u_j}) F(v') X^{v'}dv'}{y}.
$$
Then 
$$
	I(y)
	= 
	\int_{2-i\infty}^{2+i\infty} F_1(v)dF_2(v)dv = F_1(v)F_2(v)|_{2-i\infty}^{2+i\infty}-
 \int_{2-i\infty}^{2+i\infty} F_2(v) F_1'(v)dv.
$$
The boundary term vanishes as  $e^{\pi v^2/y^2}$ decays exponentially in the imaginary part of $v$.    
	Moving the line of integration to $\Re \, v = 0$ and differentiating $F_1$ we now have
$$
I(y) = -\int_{-i\infty}^{i\infty}\frac{2vF_2(v)e^{\pi v^2/y^2}dv}{y^2}.
$$
As $v = it$ and  $f(it) \ll |t|^\epsilon $,
it follows that 
$$
	F_2(it) X^{it} \ll |t|^\epsilon.
$$
Then, by Cauchy-Schwarz,
\begin{multline*}
	|F_2(iT)| 
	\ll 
	\frac{(1+|t|)^\epsilon}y \int_0^T|L(1/2 +i\gamma+it,\overline{\tilde u_j})|dt 
	\\ 
	 \ll 
	 \frac{T^{1/2+\epsilon}}{y}\left(\int_0^T |L(1/2 +i\gamma+it,\overline{\tilde u_j})|^2dt\right)^{1/2}.
\end{multline*}

The weak Lindel\"of on average result
$$
	\int_0^T |L(1/2 +i\gamma+it,\overline{\tilde u_j})|^2dt 
	\ll 
	\left(T (\ell_1\ell_2)^{1/2} (1 + |\gamma|)(1+|t_j|)\right)^{1+ \epsilon}
$$
is easily obtainable by an application of Montgomery's mean value theorem for Dirichlet polynomials (after replacing the $L$-series by a Dirichlet polynomial of length approximately 
$\sqrt{\ell_1\ell_2 ( \gamma + T + t_j)( \gamma + T - t_j)}$).

It follows from this  that 
\begin{multline*}
	I(y)  
	\ll 
	(\ell_1\ell_2)^{\frac{1}{4}+\epsilon} (1 + |\gamma|)^{1/2+ \epsilon}(1+|t_j|)^{1/2+ \epsilon} \int_{-\infty}^\infty \frac{t^{2+\epsilon}e^{-\pi t^2/y^2}}{y^3}dt
	\\
	 \ll 
	 (\ell_1\ell_2)^{1/4+\epsilon} y^\epsilon(1 + |\gamma|)^{1/2+ \epsilon}(1+|t_j|)^{1/2+ \epsilon}.
\end{multline*}
\end{proof}

To estimate 
$$
	\frac{1}{2\pi i} \int_{(2)} I_1(v)
	\frac{e^{\pi v^2/y^2}}{y} \left(\frac{Q}{\ell}\right)^v\; dv
	,
$$
we will use Propositions~\ref{prop:M2} and ~\ref{prop:dsums} as the integration in $I_1(v)$ takes place in $R_2$.  
The contribution to $\tilde\cM(s,w,v)$ in this range comes from $\cM^{(2)}_{\rm cusp} (s,w,v)$, $\cM^{(2)}_{\rm cont} (s,w,v)$ and $\cM_3(s, w, v)$, described in \eqref{e:M2}, \eqref{e:M2cusp}, \eqref{e:M2cont} and Proposition \ref{prop:dsums}. 
The integrand of $I_1(v)$ decays exponentially outside of the range $|\Im(s)|<1$. 
	The sum over $it_j$'s in $\cM^{(2)}_{\rm cusp}$ converges absolutely in this range, with exponential decay in the imaginary part of $v$, 
	We find that for $\Re(s) = 1+\nu/2+\epsilon$, 
$$
	\cM^{(2)}_{\rm cusp}(s, -s-k/2+3/2, v) 
	\ll
	e^{-\frac{\pi|\Im(v)|} {2}} (\ell_1\ell_2)^{1/4+\epsilon}
	.
$$
The same applies to $\cM^{(2)}_{\rm cont}(s, -s-k/2+3/2, v)$. 

To estimate the piece of double integral involving with $\cM_3$, 
	we note that the sum and the integral in \eqref{e:M3spec} converge absolutely. 
	Therefore, we may exchange the order of the $v$ and $s$ integrations, bringing the $v$ integral 
	inside. 

The Rankin-Selberg $L$-series,  
$L(s', \overline{u_j} \otimes E_1(*, (1+v)/2))$, 
defined in \eqref{e:RSuj}, appears as a factor of the term of $\cM_3$ corresponding to $t_j$. 
This Dirichlet series can be factored as follows:
\begin{multline*}
	\sum_{n\geq 1} \frac{\overline{\rho_j(-n)} \tilde\sigma_{-v}(n)}{n^{s'}}
	=
	\left( \sum_{n\mid (\ell_1 N_0)^\infty} \frac{\overline{\rho_j(-n)} \tilde\sigma_{-v}(n)}{n^{s'}} \right)
	L^{(\ell_1 N_0)} (s', \overline{\tilde u_j}) L^{(\ell_1 N_0)} (s'+v, \overline{\tilde u_j})
	\\
	=
	\cP_j(s', v; \ell_1 N_0)
	L (s', \overline{\tilde u_j}) L(s'+v, \overline{\tilde u_j})
	,
\end{multline*}
where $\tilde u_j$ is a new form, normalized so the first Fourier coefficient equals 1, such that all but finitely many Hecke eigenvalues agree with those of $u_j$. 

We now apply \eqref{e:Lest_2}, with the $L$-series equal to $L(1/2+v, \overline{\tilde u_j})$ 
and 
$$
	F(v) 
	= 
	\frac{\Gamma(1-w-v/2) \Gamma(w+v/2) } {\Gamma((1-v+k)/2) \Gamma((1+v-k)/2)}
	\cP_j(1/2, v; \ell_1 N_0)
	.
$$
Then $F(v)$ satisfies the condition of the Lemma, giving us
$$
	\left(\frac{1}{2\pi i} \right) 
	\int_{(2)} \int_{(2+\epsilon)} \frac{\cM_3(s, -s-k/2+3/2, v) e^{s^2}}{s-1+k/2} \; ds 
	\frac{e^{\pi v^2/y^2}}{y} \left(\frac{Q}{\ell}\right)^v \; dv
	\ll 
	y^\epsilon (\ell_1\ell_2)^{1/4+\epsilon}
	. 
$$
Combining the proceeding upper bounds gives us
$$
	\frac{1}{2\pi i} \int_{(2)} I_1(v)
	\frac{e^{\pi v^2/y^2}}{y} \left(\frac{Q}{\ell}\right)^v\; dv
	\ll 
	y^\epsilon (\ell_1\ell_2)^{1/4+\epsilon}
	.
$$

To demonstrate
$$
	\frac{1}{2\pi i} \int_{(2)} I_2(v)
	\frac{e^{\pi v^2/y^2}}{y} \left(\frac{Q}{\ell}\right)^v\; dv
	\ll 
	y^\epsilon (\ell_1\ell_2)^{1/4+\epsilon}
	,
$$
we will use Proposition \ref{prop:M1}, as the integration in $I_2(v)$ takes place in $R_1$. 
	The contribution to $\tilde\cM(s, w, v)$ in this range comes from 
	$\cM_{\rm cusp}^{(1)}(s, w, v)$ and $\cM_{\rm int}^{(1)}(s, w, v)$, 
	described in \eqref{e:M1cusp} and \eqref{e:M1int}. 
	The corresponding upper bound is obtained in a way precisely analogous to the 
	computation involving $\cM_3$. 

The sum $T_2(v)$ behaves similarly to $\cM^{(2)}_{\rm cusp} (s, -s-k/2+3/2, v) +\cM^{(2)}_{\rm cont}(s, -s-k/2+3/2, v)$ 
and $T_1(v)$ behaves similarly to $\cM^{(1)}_{\rm cusp}(s, -s-k/2+3/2, v) +\cM^{(1)}_{\rm int}(s, -s-k/2+3/2, v)$, 
leading to corresponding upper bounds for the two integrals in \eqref{e:T12}.
This completes the proof of the second part of the proposition, \eqref{e:Mest_2}. 
Applying \eqref{e:Lest_1}, instead of \eqref{e:Lest_2}, gives us the first part of the proposition, \eqref{e:Mest_1}.

\end{proof}

\subsection{Averaging over $q$}

Let $N_0$ be a square-free positive integer.
Let $f$ and $g$ be holomorphic cusp forms of even weight $k$, 
which are newforms for $\Gamma_0(N_0)$, 
with normalized Fourier coefficients $A(m)$, $B(m)$, as in \eqref{def1}, so that $A(1) = B(1)=1$.
Assume that $f$ and $g$ are eigenfunctions for all Hecke operators.

Fix square free integers $\ell_1, \ell_2\geq 1$ such that $(N_0, \ell_1\ell_2)=1$. 
For $q\geq 1$ with $(q, N_0)=1$, define
\be\label{e:tildeZ}
	\tilde Z_q(s, w; \ell_1, \ell_2)
	:=
	Z_q(s, w; \ell_1, \ell_2) - \Psi_q(s, w; \ell_1, \ell_2)
\ee
where $\Psi_q$ is given in \eqref{e:Psi0}. 

As given in Theorem \ref{theorem:SQ}, define
\begin{multline}\label{e:H1qfg}
	H^{(1)}_{f, g}(Q; \ell_1, \ell_2)
\\	:=
	\frac{1}{4}
	\frac{(\ell_1, \ell_2)}{\ell_1\ell_2}
	L(1, f\otimes g)
	\sum_{a\mid N} 
	\left\{
	\left(\frac{1}{2^{r(N)}} + \frac{1}{Q} \prod_{p\mid a, p^\alpha\|Q , \atop \alpha\geq 0} (p^\alpha-1) \prod_{p^\alpha\|Q, p\nmid N, \atop \alpha\geq 1} p^\alpha\right)
	\right.
	\\
	\times
	\left(\frac{N_0}{(a, N_0)}\right) A\left(\frac{N_0}{(a, N_0)}\right) 
	\overline{B\left(\frac{N_0}{(a, N_0)}\right)}
	\\
	\left.
	\times
	\left(\prod_{p\mid \frac{\ell_1 (a, \ell_2)}{(\ell_1, \ell_2) (a, \ell_1)}}
	\frac{ A(p)-B(p)p^{-1}}{1-p^{-2}}\right)
	\left( \prod_{p\mid \frac{\ell_2 (a, \ell_1)}{(\ell_1, \ell_2) (a, \ell_2)}}
	\frac{ B(p)-A(p)p^{-1}}{1-p^{-2}} \right)
	\right\}
	,
\end{multline}
if $f\neq g$.  
Here $N= N_0 \frac{\ell_1\ell_2}{(\ell_1,\ell_2)}$.  
Then, for $Q\geq 1$ with $(Q, N_0)=1$, define
\be\label{e:H1fg}
	H^{(1)}_{f, g}(Q)
	:=
	\sum_{d\mid Q} \frac{\mu(d)}{d} 
	\sum_{d_1, d_2\mid d} \mu(d_1) \mu(d_2) A(d/d_1) \overline{B(d/d_2)}
	H^{(1)}_{f, g}(Q/d; d_1, d_2)
	.
\ee

\begin{thm}\label{thm:SfgQ}

Assume $N_0$, $f$ and $g$ as above. 
For $Q\geq 1$, with $(Q, N_0)=1$, 
when $f=g$, there exists a constant $C_2(k)$, given in \eqref{C1def}, 
independent of $f$ (but depending on $k$), such that
\begin{multline}\label{e:f=g}
	S_{f, f}(Q)
	=
	H_{f, f} (Q)
	+
	2\sum_{d\mid Q} \frac{\mu(d)}{d} 
	\sum_{d_1, d_2\mid d} \mu(d_1) \mu(d_2) A(d/d_1) \overline{B(d/d_2)}
	\tilde Z_{\frac{Q}{d}} \left(1-k/2, 1/2; d_1, d_2\right)
	\\
	+
	\cO\left(Q^{\theta+\epsilon -1/2}\right)
	,
\end{multline}
where $H_{f, f}(Q)$ is defined in \eqref{e:Hqff}. 
When $f\neq g$, 
\begin{multline}\label{e:fneg}
	S_{f, g}(Q)
	=
	L^{(Q)}(1, f\otimes g)
	+
	H_{f, g}^{(1)}(Q) + H_{g, f}^{(1)}(Q)
	\\
	+
	\sum_{d\mid Q} \frac{\mu(d)}{d} 
	\sum_{d_1, d_2\mid d} \mu(d_1) \mu(d_2) A(d/d_1) \overline{B(d/d_2)}
	\tilde Z_{\frac{Q}{d}} \left(1-k/2, 1/2; d_1, d_2\right)
	\\
	+
	\sum_{d\mid Q} \frac{\mu(d)}{d} 
	\sum_{d_1, d_2\mid d} \mu(d_1) \mu(d_2) B(d/d_1) \overline{A(d/d_2)}
	\tilde Z_{\frac{Q}{d}} \left(1-k/2, 1/2; d_1, d_2\right)
	\\
	+
	\cO\left(Q^{\theta+\epsilon -1/2}\right)
\end{multline}
for any $\epsilon>0$. 
  
The implied constant in the error term depends upon  $f$, $g$ and $\epsilon$, but is independent of $Q$.
\end{thm}

\begin{proof}

By Proposition \ref{prop:Sblockdef}, for any $Q\geq 1$ with $(Q, N_0)=1$, 
we have
$$
	S_2(X)
	=
	\sum_{d\mid Q} \frac{\mu(d)}{d} 
	\sum_{d_1, d_2\mid d} \mu(d_1) \mu(d_2) A(d/d_1) \overline{B(d/d_2)} 
	S\left(\frac{X}{d}, \frac{Q}{d}, d_1, d_2\right)
$$
where $S_2(X)$ is defined in (\ref{S2}).


When $f\neq g$, by Theorem \ref{theorem:SQ}, we have
\begin{multline*}
	S_2(X, Q)
	=
	H_{f, g}^{(1)}(Q) 
	+
	\sum_{d\mid Q} \frac{\mu(d)}{d} 
	\sum_{d_1, d_2\mid d} \mu(d_1) \mu(d_2) A(d/d_1) \overline{B(d/d_2)}
	\tilde Z_{\frac{Q}{d}} \left(1-k/2, 1/2; d_1, d_2\right)
	\\
	+
	\cO\left(Q^{\theta+\epsilon -1/2}\right) 
	+
	\cO(X^{-\epsilon})
	.
\end{multline*}
Combining with \eqref{e:Stotal}, \eqref{S1}, \eqref{S2}, \eqref{S3} and \eqref{S4}, 
we get \eqref{e:fneg}. 
When $f=g$, by Lemma \ref{lem:innerprod_cusp}, we have
$$
	\Res_{s=1} L(s, f\otimes f; \ell_1, \ell_2)
	= 
	\frac{A\left(\frac{\ell_1\ell_2}{(\ell_1, \ell_2)^2}\right)}{(\ell_1, \ell_2)}
	\left(\prod_{p\mid \frac{\ell_1\ell_2}{(\ell_1,\ell_2)^2}}\frac{1}{p+1}\right)
	\Res_{s=1} L(s, f\otimes f)
	.
$$
Combining with Theorem \ref{theorem:SQ}, define
\begin{multline*}
	H^{(2)}_{f, f}(Q; \ell_1, \ell_2)
	:=
	\frac{1}{4}
	\left\{
	\sum_{p\mid(\ell_1, \ell_2), p^\alpha\|Q, \atop \alpha\geq 0} 
	\log p \left(p^{-\alpha} +\frac{1}{2} \right)
	- \sum_{p\mid \frac{\ell_1\ell_2}{(\ell_1, \ell_2)}} \log p \frac{2p-1}{2(p-1)}
	\right.
	\\
	\left.
	+
	\sum_{p\mid \frac{\ell_1\ell_2}{(\ell_1, \ell_2)}, p^\alpha\|Q, \atop \alpha\geq 0} 
	\log p \left(p^{-\alpha} \frac{3p-1}{p-1}\right)
	\right\}
	\frac{A\left(\frac{\ell_1\ell_2}{(\ell_1, \ell_2)^2}\right)}{(\ell_1, \ell_2)}
	\left(\prod_{p\mid \frac{\ell_1\ell_2}{(\ell_1,\ell_2)^2}}\frac{1}{p+1}\right)
\end{multline*}
and
$$
	\tilde H_{f, f}^{(2)}(Q)
	:=
	\sum_{d\mid Q} \frac{\mu(d)}{d} 
	\sum_{d_1, d_2\mid d} \mu(d_1) \mu(d_2) A(d/d_1) \overline{B(d/d_2)}
	H_{f, f}^{(2)}(Q/d; d_1, d_2)
	,
$$
then
\begin{multline*}
	S_2(X, Q)
	=
	- 
	\frac{1}{2}
	\Res_{s=1} L^{(Q)} (s,f \otimes f) 
	\log X
	+
	\Res_{s=1} L^{(Q)} (s,f \otimes f) 
	\log Q
	\\ 
	- 
	\sum_{p^\alpha\|Q, \alpha\geq 1} \log p \frac{1-p^{-\alpha}}{p-1}
	\Res_{s=1} L^{(Q)}(s, f \otimes f)
	\\
	+
	\left\{ C_2(k) 
	+ \frac{1}{4}\sum_{p\mid N_0} \log p \left(\frac{4p-1}{2(p-1)}\right)
	+ \frac{1}{2}\frac{\Gamma'(k)}{\Gamma(k)} - \frac{1}{2} \log (4\pi)
	\right\} 
	\Res_{s=1} 	L^{(Q)} (s,f \otimes f)
	\\
	- 
	\frac{1}{2} \sum_{d\mid Q} \frac{\mu(d)}{d} 
	\left(\log d -2\sum_{p\mid d} \frac{\log p }{p}\right)
	\Res_{s=1} L(s, f_d\otimes f_d)
	\\
	+ 
	\tilde H_{f, f}^{(2)}(Q) \Res_{s=1} L(s, f\otimes f)
	+ 
	\frac{1}{2}
	\left.\frac{d}{ds} \left( (s-1) L^{(Q)} (s, f\otimes f)\right)\right|_{s=1}	
	\\
	+
	\sum_{d\mid Q} \frac{\mu(d)}{d} 
	\sum_{d_1, d_2\mid d} \mu(d_1) \mu(d_2) A(d/d_1) \overline{B(d/d_2)}
	\tilde Z_{\frac{Q}{d}} \left(1-k/2, 1/2; d_1, d_2\right)
	\\
	+
	\cO\left(Q^{\theta+\epsilon -1/2}\right) 
	+
	\cO(X^{-\epsilon})	
	.
\end{multline*}	
When $f=g$, 
for each $d\mid Q$, we have
$$
	\Res_{s=1} L(s, f_d\otimes f_d) 
	=
	\left(\prod_{p\mid d} \frac{A(p)^2 + (1-A(p)^2) p^{-1}+p^{-2}}{1+p^{-1}} \right)
	\Res_{s=1} L(s, f\otimes f)
	.
$$
Again, from \eqref{e:Stotal}, \eqref{S1}, \eqref{S2}, \eqref{S3} and \eqref{S4}, 
we get \eqref{e:f=g}. 

\end{proof}

We are now ready to prove Theorem \ref{thm:main}. 
\subsection{Proof of Theorem \ref{thm:main} and Corollary \ref{cor:main}}

From Theorem \ref{thm:SfgQ}, recall $H_{f, f}(q)$ as in \eqref{e:Hqff}, 
when $f=g$.
Define 
$$
	H_{f, g}(q)
	:=
	L^{(q)}(1, f\otimes g)
	+
	H_{f, g}^{(1)}(q) + H_{g, f}^{(1)}(q)
	,
$$
when $f\neq g$. Here $H_{f, g}^{(1)}(q)$ is defined in \eqref{e:H1fg}. 

For either $f\neq g$ or $f=g$, we have
\begin{multline*}
	\sum_{q\geq 1, \atop (q, N_0)=1} S_{f, g}(q) e^{-\frac{y^2\left(\log \frac{Q}{q}\right)^2}{4\pi}} 
	=
	\sum_{q\geq1, \atop (q, N_0)=1} H_{f, g}(q) e^{-\frac{y^2\left(\log \frac{Q}{q}\right)^2}{4\pi}} 
	+
	\cO\left(\sum_{q\geq1, \atop (q, N_0)=1} q^{\theta+\epsilon-1/2} e^{-\frac{y^2\left(\log \frac{Q}{q}\right)^2}{4\pi}} \right)
	\\
	+
	\sum_{q\geq 1, \atop (q, N_0)=1} 
	\left(
	\sum_{d\mid q} \frac{\mu(d)}{d} 
	\sum_{d_1, d_2\mid d} \mu(d_1) \mu(d_2) A(d/d_1) \overline{B(d/d_2)}
	\tilde Z_{\frac{q}{d}} \left(1-k/2, 1/2; d_1, d_2\right)
	\right)
	e^{-\frac{y^2\left(\log \frac{Q}{q}\right)^2}{4\pi}}
	\\
	+
	\sum_{q\geq 1, \atop (q, N_0)=1} 
	\left(
	\sum_{d\mid q} \frac{\mu(d)}{d} 
	\sum_{d_1, d_2\mid d} \mu(d_1) \mu(d_2) B(d/d_1) \overline{A(d/d_2)}
	\tilde Z_{\frac{q}{d}} \left(1-k/2, 1/2; d_1, d_2\right)
	\right)
	e^{-\frac{y^2\left(\log \frac{Q}{q}\right)^2}{4\pi}}
	.
\end{multline*}
By applying the identity 
\be\label{identity}
	\frac{1}{2\pi i} \int_{(2)} \frac{e^{\frac{\pi v^2}{y^2}}}{y} X^v \; dv
	=
	e^{-\frac{y^2 (\log X)^2}{4\pi}}
	,
\ee
we have
\begin{multline*}
	\sum_{q\geq 1, \atop (q, N_0)=1} 
	\left(
	\sum_{d\mid q} \frac{\mu(d)}{d} 
	\sum_{d_1, d_2\mid d} \mu(d_1) \mu(d_2) A(d/d_1) \overline{B(d/d_2)}
	\tilde Z_{\frac{q}{d}} \left(1-k/2, 1/2; d_1, d_2\right)
	\right)
	e^{-\frac{y^2\left(\log \frac{Q}{q}\right)^2}{4\pi}}
	\\
	=
	\frac{1}{y}
	\frac{1}{2\pi i} \int_{(2)}
	\sum_{q\geq 1, \atop (q, N_0)=1} 
	\sum_{d\mid q} \frac{\mu(d)}{d} 
	\sum_{d_1, d_2\mid d} \mu(d_1)\mu(d_2) A(d/d_1) \overline{B(d/d_2)} 
	\\
	\times 
	\frac{\tilde Z_{\frac{q}{d}}\left(1-k/2, 1/2; d_1, d_2\right)}{q^v}
	e^{\frac{\pi v^2}{y^2}}Q^v
	\; dv
	.
\end{multline*}
Then, 
\begin{multline*}
	\sum_{q\geq 1, \atop (q, N_0)=1} 
	\sum_{d\mid q} \frac{\mu(d)}{d} 
	\sum_{d_1, d_2\mid d} \mu(d_1)\mu(d_2) A(d/d_1) \overline{B(d/d_2)} 
	\frac{\tilde Z_{\frac{q}{d}}\left(1-k/2, 1/2; d_1, d_2\right)}{q^v}
	\\
	=
	\sum_{\ell\geq 1, \atop (\ell, N_0)=1} 
	\frac{\mu(\ell)}{\ell^{1+v}}
	\sum_{\ell_1, \ell_2\mid \ell} 
	\mu(\ell_1)\mu(\ell_2) A(\ell/\ell_1) \overline{B(\ell/\ell_2)} 
	\tilde\cM\left(1-k/2, 1/2, v; \ell_1, \ell_2\right)
	.
\end{multline*}

By \eqref{e:Mest_1} in Proposition~\ref{prop:Mest}, we have
$$
	\frac{1}{2\pi i} \int_{(2)} 
	\tilde\cM(1-k/2, 1/2,v; \ell_1, \ell_2)
	\frac{e^{\pi v^2/y^2}}{y} \left(\frac{Q}{\ell}\right)^v\; dv
	\ll 
	\left(\frac{Q}{\ell}\right)^{1/2} (\ell_2\ell_2)^\epsilon
	, 
$$
for any $\epsilon>0$. 
Consequently
\be\label{e:withy}
	\frac{y}{Q}
	\sum_{q\geq 1, \atop (q, N_0)=1} S_{f, g}(q) e^{-\frac{y^2\left(\log \frac{Q}{q}\right)^2}{4\pi}}
	=
	\frac{y}{Q}
	\sum_{q\geq 1, \atop (q, N_0)=1} H_{f, g}(q) e^{-\frac{y^2\left(\log \frac{Q}{q}\right)^2}{4\pi}}
	+
	\cO(Q^{\theta+\epsilon-1/2})
	+
	\cO(y Q^{-1/2})
	.
\ee

When $f=g$, we get \eqref{e:mainf=g}.

Now assume that $f\neq g$. Consider 
$$
	\frac{y}{Q} 
	\sum_{q\geq 1, \atop (q, N_0)=1} H_{f, g}^{(1)}(q) e^{-\frac{y^2 \left(\log \frac{Q}{q}\right)^2}{4\pi}} 
	=
	\frac{1}{Q} \frac{1}{2\pi i} \int_{(2)} 
	\left( \sum_{q\geq 1, \atop (q, N_0)=1} 
	\frac{H_{f, g}^{(1)}(q) }{q^v} \right)
	Q^v e^{\frac{\pi v^2}{y^2}}\; dv
	.
$$
We have
\begin{multline*}
	\sum_{q\geq 1, \atop (q, N_0)=1} 
	\frac{H_{f, g}^{(1)}(q) }{q^v}
	\\	
	=
	\sum_{\ell\geq 1, \atop (\ell, N_0)=1}
	\frac{\mu(\ell)}{\ell^{1+v} } 
	A(\ell) \overline{B(\ell)}
	\sum_{\ell_1, \ell_2\mid \ell}
	\mu(\ell_1)\mu(\ell_2) A(\ell_1)^{-1} \overline{B(\ell_2)}^{-1} 
	\sum_{q\geq 1, \atop (q, N_0)=1}
	\frac{H_{f, g}^{(1)}(q; \ell_1, \ell_2)}{q^v}
	.
\end{multline*}

For each $q$, $\ell_1$ and $\ell_2$, we have defined $H_{f, g}^{(1)}(q; \ell_1, \ell_2)$ in \eqref{e:H1qfg}. 
For $a\mid N_0 \frac{\ell_1\ell_2}{(\ell_1, \ell_2)}$, define
\begin{multline*}
	H_{f, g} (a; \ell_1, \ell_2)
	:=
	\left(\frac{N_0}{(a, N_0)}\right) 
	A\left(\frac{N_0}{(a, N_0)}\right) 
	\overline{B\left(\frac{N_0}{(a, N_0)}\right)}
	\\
	\left.
	\times
	\left(\prod_{p\mid \frac{\ell_1 (a, \ell_2)}{(\ell_1, \ell_2) (a, \ell_1)}}
	\frac{ A(p)-B(p)p^{-1}}{1-p^{-2}}\right)
	\left( \prod_{p\mid \frac{\ell_2 (a, \ell_1)}{(\ell_1, \ell_2) (a, \ell_2)}}
	\frac{ B(p)-A(p)p^{-1}}{1-p^{-2}} \right)
	\right\}
	.
\end{multline*}
Then
\begin{multline*}
	H_{f, g}^{(1)}(q;\ell_1, \ell_2)
	=
	\frac{1}{4} L(1, f\otimes g) 
	\sum_{a\mid N_0\frac{\ell_1\ell_2}{(\ell_1, \ell_2)}}
	\frac{1}{2^{r\left(N_0\frac{\ell_1\ell_2}{(\ell_1, \ell_2)}\right)}}
	H_{f, g} (a; \ell_1, \ell_2)
	\\
	+
	\frac{1}{4} L(1, f\otimes g) 
	\sum_{a\mid N_0\frac{\ell_1\ell_2}{(\ell_1, \ell_2)}}
	\left(
	\prod_{p^\alpha\|Q, p\mid a} (p^\alpha-1) 
	\prod_{p^\alpha\|Q, p\mid\frac{\ell_1\ell_2}{(\ell_1, \ell_2)}, \atop \alpha\geq 1} p^{-\alpha }
	\right)
	H_{f, g} (a; \ell_1, \ell_2)
	.
\end{multline*}
So, 
\begin{multline*}
	\sum_{q\geq 1, \atop (q, N_0)=1}
	\frac{H_{f, g}^{(1)}(q; \ell_1, \ell_2)}{q^v}
	=
	\zeta^{(N_0)}(v) 
	\frac{1}{4} L(1, f\otimes g) 
	\frac{1}{2^{r\left(N_0\frac{\ell_1\ell_2}{(\ell_1, \ell_2)}\right)}}
	\sum_{a\mid N_0\frac{\ell_1\ell_2}{(\ell_1, \ell_2)}}
	H_{f, g} (a; \ell_1, \ell_2)
	\\
	+
	\zeta^{(N_0)}(v) 
	\frac{1}{4} L(1, f\otimes g)
	\prod_{p\mid\frac{\ell_1\ell_2}{(\ell_1, \ell_2)} } 
	\frac{1-p^{-v}}{p^{v+1}-1}
	\sum_{a\mid N_0 \frac{\ell_1\ell_2}{(\ell_1, \ell_2)}}
	\prod_{p\mid \left(a, \frac{\ell_1\ell_2}{(\ell_1, \ell_2)}\right)} \frac{p-1}{1-p^{-v}}
	H_{f, g} (a; \ell_1, \ell_2)
	.
\end{multline*}
So we have, 
\begin{multline*}
	\frac{1}{2^{r\left(N_0\frac{\ell_1\ell_2}{(\ell_1, \ell_2)}\right)}} 
	\sum_{a\mid N_0\frac{\ell_1\ell_2}{(\ell_1, \ell_2)}} 
	H_{f, g}(a; \ell_1, \ell_2)
	\\
	=
	\frac{1}{2^{r\left(N_0\frac{\ell_1}{(\ell_1, \ell_2)} \frac{\ell_2}{(\ell_1, \ell_2)}\right)}}
	N_0 A(N_0) \overline{B(N_0)}
	\prod_{p\mid N_0} 
	\left(1+ p^{-1} A(p)^{-1}\overline{B(p)}^{-1}\right) 
	\prod_{p\mid\frac{\ell_1\ell_2}{(\ell_1, \ell_2)^2}}
	\left( \frac{A(p)+B(p)}{1+p^{-1}}\right)
\end{multline*}
and
\begin{multline*}
	\prod_{p\mid\frac{\ell_1\ell_2}{(\ell_1, \ell_2)} } 
	\frac{1-p^{-v}}{p^{v+1}-1}
	\sum_{a\mid N_0 \frac{\ell_1\ell_2}{(\ell_1, \ell_2)}}
	\prod_{p\mid \left(a, \frac{\ell_1\ell_2}{(\ell_1, \ell_2)}\right)} 
	\frac{p-1}{1-p^{-v}}
	H_{f, g} (a; \ell_1, \ell_2)
	\\
	=
	\left(\prod_{p\mid\frac{\ell_1\ell_2}{(\ell_1, \ell_2)^2} } 
	\frac{1-p^{-v}}{p^{v+1}-1}\right)
	(\ell_1, \ell_2)^{-v} 
	N_0 A(N_0) \overline{B(N_0)} 
	\prod_{p\mid N_0} \left(1+p^{-1}A(p)^{-1}\overline{B(p)}^{-1}\right)
	\\
	\times
	\left(
	\prod_{p\mid \frac{\ell_1}{(\ell_1, \ell_2)} } 
	\frac{A(p)-B(p)p^{-1}}{1-p^{-2}} 
	+ 
	\frac{(p-1) (B(p)-A(p)p^{-1})}{(1-p^{-2})(1-p^{-v})}
	\right)
	\\
	\times
	\left(
	\prod_{p\mid \frac{\ell_2}{(\ell_1, \ell_2)}} 
	\frac{B(p)-A(p)p^{-1}}{1-p^{-2}}
	+\frac{(p-1)(A(p)-B(p)p^{-1})}{(1-p^{-2})(1-p^{-v})}
	\right)
	.
\end{multline*}
For $f\neq g$, gives us \eqref{e:mainfneqg}.

For $p\mid N_0$, $A(p)=\pm \sqrt p^{-1}$ and $B(p)=\pm \sqrt{p}^{-1}$, 
so the second piece of the main term is non-negative.  It could, however, be equal to zero if $pA(p)\overline{B(p)}+1=0$ for some  $p\mid N_0$.
The first piece of the main term, however, is positive and dominates the error term.   Specifically, we will now show that 
\begin{multline}\label{lastbit}
	\frac{y}{Q}
	\sum_{q\geq 1, \atop (q, N_0)=1} L^{(q)}(1, f\otimes g)
	e^{-\frac{y^2\left(\log \frac{Q}{q}\right)^2}{4\pi}}\\
	= 
	\frac{L(1, f\otimes g)}{\zeta^{(N_0)}(2)}e^{\frac{\pi}{y^2}} \Res_{v=1}\zeta^{(N_0)}(v)E_{f, g}^{(N_0)}(1) 
	+ \cO(Q^{\theta+\epsilon-1}).
\end{multline}
Here $\zeta^{(N_0)}(v)$ refers to the zeta function with Euler factors corresponding to primes dividing $N_0$ removed, and 
\be\label{Edef}
	E^{(N_0)}_{f, g}(v) \\= \prod_{p \nmid N_0}E_p(v; f, g),
\ee
with
\begin{multline*}
	E_p(v; f, g)
	=
	1-p^{-2} +p^{-2-v}-p^{-1-v}A(p)\overline{B(p)}+p^{-2-v}(A(p^2)+\overline{B(p^2)}) \\-p^{-3-v}A(p)\overline{B(p)} +p^{-4}.
\end{multline*}
As the main term is positive, and independent of $Q$, this dominates the error term, giving us the result.
To verify this, we first note that 
\begin{multline*}
L^{(q)}(1, f\otimes g) = L(1, f\otimes g)\\ \times \prod_{p \mid q}\frac{(1-\alpha_p \beta_p p^{-1})(1-\alpha_p^{-1} \beta_p p^{-1})(1-\alpha_p \beta_p^{-1}  p^{-1})(1-\alpha_p^{-1}  \beta_p^{-1}  p^{-1})}{(1-p^{-2})},
\end{multline*}
where $A(p) = \alpha_p + \alpha_p^{-1} $ and $\overline{B(p) }=  \beta_p + \beta_p^{-1} $.
It follows, after a calculation, that 
$$
\sum_{(q,N_0)=1}\frac{L^{(q)}(1, f\otimes g)}{q^v} =\frac{L(1, f\otimes g)\zeta^{(N_0)}(v)}{\zeta^{(N_0)}(2)}\prod_{p \nmid N_0}E_p(v;f, g),
$$
with $E_p(v)$ as given above.   Applying the identity \eqref{identity}, with $X=Q/q$, we obtain
\begin{multline*}
	\frac{y}{Q}
	\sum_{q\geq 1, \atop (q, N_0)=1} L^{(q)}(1, f\otimes g)
	e^{-\frac{y^2\left(\log \frac{Q}{q}\right)^2}{4\pi}}
	=\frac{1}{Q}
	\sum_{q\geq 1, \atop (q, N_0)=1} L^{(q)}(1, f\otimes g)
	\frac{1}{2 \pi i}\int_{(2)}e^{\frac{\pi v^2}{y^2}}(Q/q)^v dv
	\\
	=\frac{L(1, f\otimes g)}{Q\zeta^{(N_0)}(2)}\frac{1}{2 \pi i}\int_{(2)}e^{\frac{\pi v^2}{y^2}}Q^v \zeta^{(N_0)}(v)E_{f, g}^{(N_0)}(v)dv.
\end{multline*}
The infinite product $E_{f, g}^{(N_0)}(v)$ is analytic for $\Re \, v >-1$, and so moving the line of integration to $\Re \, v =-1+\epsilon$ gives us
\eqref{lastbit}, with the residue at $v=1$ being the main term.


\thispagestyle{empty}
{\footnotesize
\nocite{*}
\bibliographystyle{amsalpha}
\bibliography{reference}
}

\end{document}